\documentclass[12pt,reqno]{amsart}

\usepackage{times}
\usepackage[T1]{fontenc}
\usepackage{mathrsfs}
\usepackage{mathtools}
\usepackage{listofitems}
\usepackage{latexsym}
\usepackage[dvips]{graphics}
\usepackage{epsfig}
\usepackage{hyperref, amsmath, amsthm, amsfonts, amscd, flafter,epsf}
\usepackage{amsmath,amsfonts,amsthm,amssymb,amscd}
\input amssym.def
\input amssym.tex
\usepackage{color}
\usepackage{fix-cm}
\usepackage{physics}
\usepackage{hyperref}
\usepackage{color}
\usepackage{url}
\usepackage{breakurl}
\newcommand{\bburl}[1]{\textcolor{blue}{\url{#1}}}

\usepackage{comment}
\usepackage{multirow}

\newcommand\be{\begin{equation}}
\newcommand\ee{\end{equation}}
\newcommand\bea{\begin{eqnarray}}
\newcommand\eea{\end{eqnarray}}
\newcommand\bi{\begin{itemize}}
\newcommand\ei{\end{itemize}}
\newcommand\ben{\begin{enumerate}}
\newcommand\een{\end{enumerate}}
\newcommand\numberthis{\addtocounter{equation}{1}\tag{\theequation}}

\newtheorem{thm}{Theorem}[section]
\newtheorem{conj}[thm]{Conjecture}
\newtheorem{cor}[thm]{Corollary}
\newtheorem{lem}[thm]{Lemma}
\newtheorem{prop}[thm]{Proposition}

\newtheorem{defi}[thm]{Definition}

\newcommand{\N}{\mathbb{N}}

\newcommand{\E}{\mathbb{E}}

\newcommand\cycle[2][\,]{%
  \readlist\thecycle{#2}%
  (\foreachitem\i\in\thecycle{\ifnum\icnt=1\else#1\fi\i})%
}

\newcommand{\ga}{\alpha}                  
\newcommand{\gb}{\beta}
\newcommand{\gep}{\epsilon}
\newcommand{\gl}{\lambda}
\newcommand{\p}{\ensuremath{\mathbb{P}}}
\newcommand{\f}{\mathcal{F}}

\numberwithin{equation}{section}

\makeatletter
\def\@tocline#1#2#3#4#5#6#7{\relax
  \ifnum #1>\c@tocdepth 
  \else
    \par \addpenalty\@secpenalty\addvspace{#2}%
    \begingroup \hyphenpenalty\@M
    \@ifempty{#4}{%
      \@tempdima\csname r@tocindent\number#1\endcsname\relax
    }{%
      \@tempdima#4\relax
    }%
    \parindent\z@ \leftskip#3\relax \advance\leftskip\@tempdima\relax
    \rightskip\@pnumwidth plus4em \parfillskip-\@pnumwidth
    #5\leavevmode\hskip-\@tempdima
      \ifcase #1
       \or\or \hskip 1em \or \hskip 2em \else \hskip 3em \fi%
      #6\nobreak\relax
    \hfill\hbox to\@pnumwidth{\@tocpagenum{#7}}\par
    \nobreak
    \endgroup
  \fi}
\makeatother

\begin{document}

\title{Distribution of Eigenvalues of Random Real Symmetric Block Matrices}

\author[Blackwell]{Keller Blackwell}
\email{kellerb@mail.usf.edu}
\address{Department of Mathematics, University of South Florida, Tampa, FL 33620}

\author[Borade]{Neelima Borade}
\email{nborad2@uic.edu}
\address{Department of Mathematics, University of Illinois at Chicago, Chicago, IL 60607}

\author[Devlin VI]{Charles Devlin VI}
\email{chatrick@umich.edu}
\address{Department of Mathematics, University of Michigan, Ann Arbor, MI 48109}

\author[Luntzlara]{Noah Luntzlara}
\email{nluntzla@umich.edu}
\address{Department of Mathematics, University of Michigan, Ann Arbor, MI 48109}

\author[Ma]{Renyuan Ma}
\email{renyuanma01@gmail.com}
\address{Department of Mathematics, Bowdoin College, ME, 04011}

\author[Miller]{Steven J. Miller}
\email{sjm1@williams.edu}
\address{Department of Mathematics, Williams College, MA 01267}

\author[Wang]{Mengxi Wang}
\email{mengxiw@umich.edu}
\address{Department of Mathematics, University of Michigan, Ann Arbor, MI 48109}

\author[Xu]{Wanqiao Xu}
\email{wanqiaox@umich.edu}
\address{Department of Mathematics, University of Michigan, Ann Arbor, MI 48109}



\thanks{This work was supported by NSF grants DMS1561945 and DMS-1659037, Bowdoin College, the University of Michigan, and Williams College; it is a pleasure to thank them for their support. We would also like to thank Shiliang Gao and Dr. Jun Yin for helpful comments, and Zhijie Chen, Jiyoung Kim, and Samuel Murray for providing a counter-example to a conjecture on the contribution of terms in the expansion of formulas for the moments of the disco of two ensembles.}

\maketitle

\begin{abstract} Random Matrix Theory (RMT) has successfully modeled diverse systems, from energy levels of heavy nuclei to zeros of $L$-functions. Many statistics in one can be interpreted in terms of quantities of the other; for example, zeros of $L$-functions correspond to eigenvalues of matrices, and values of $L$-functions to values of the characteristic polynomials. This correspondence has allowed RMT to successfully predict many number theory behaviors; however, there are some operations which to date have no RMT analogue. The motivation of this paper is to try and find an RMT equivalent to Rankin-Selberg convolution, which builds a new $L$-functions from an input pair. We report on one attempt; while it does not model convolution, it does create new matrix families with properties in between those of the constituents.


For definiteness we concentrate on two specific families, the ensemble of palindromic real symmetric Toeplitz (PST) matrices and the ensemble of real symmetric (RS) matrices, whose limiting spectral measures are the Gaussian and semicircle distributions, respectively; these were chosen as they are the two extreme cases in terms of moment calculations. For a PST matrix $A$ and a RS matrix $B$, we construct an ensemble of random real symmetric block matrices whose first row is $\lbrace A, B \rbrace$ and whose second row is $\lbrace B, A \rbrace$. By Markov's Method of Moments, we show this ensemble converges weakly and almost surely to a new, universal distribution with a hybrid of Gaussian and semicircle behaviors. We extend this construction by considering an iterated concatenation of matrices from an arbitrary pair of random real symmetric sub-ensembles with different limiting spectral measures. We prove that finite iterations converge to new, universal distributions with hybrid behavior, and that infinite iterations converge to the limiting spectral measures of the component matrices.
\end{abstract}

\tableofcontents

\section{Introduction}
\subsection{History}

Random Matrix Theory (RMT) is well-suited to the fundamental problem of studying spacings between observed values arising from large, complex systems such as energy levels of heavy nuclei and vertical spacings of zeros of the Riemann zeta function. Similar to the Central Limit Theorem, the behavior of a typical element is often close to the system average, which frequently can be computed.



For example, the intractability of the three-body problem is only exacerbated in the study of heavy nuclei, characterized by the interactions of hundreds of protons and neutrons. The fundamental equation governing such quantum systems is Schr\"{o}dinger's Equation $H \Psi_n = E_n \Psi_n$ where $H$, the \emph{Hamiltonian matrix}, is an infinite dimensional matrix whose entries are computed from the little-understood quantum system. Wigner \cite{Wig1} in 1955 opened a new avenue into the study of heavy nuclei by considering, rather than the true $H$ of the system, a random $N \times N$ real symmetric matrices with entries i.i.d.r.v. from appropriate probability distributions. Average eigenvalue density and spacings can be then computed for any finite $N$, and the eigenvalue behavior of a single typical random matrix converges to the limits of system averages as $N \to \infty$. A key result is Wigner's Semi-Circle Law \cite{Wig2}, which states that the distribution of normalized eigenvalues of a random real symmetric or complex Hermitian matrix with entries i.i.d.r.v. from a fixed probability distribution with mean $0$ and variance $1$ converges to the semi-circle density.

The ensemble of $N \times N$ real symmetric matrices has $N(N+1)/2$ independent parameters; a natural question is how placing additional structural constraints, and thereby reducing the degrees of freedom, affects eigenvalue behavior. Recently the density of eigenvalues of a thin subset of real symmetric matrices was studied.\footnote{There are many other ensembles of matrices one can investigate, yielding new behavior. An extreme example are checkerboard ensembles \cite{BCDHMSTVY,CKLMSW}, where most of the eigenvalues follow the semi-circle law but a fixed number diverge to infinity as the matrix size grows, with a scaled limiting distribution equal to that of hollow standard ensembles. For more choices see the references in these works.} Recall an $N\times N$ symmetric
palindromic Toeplitz matrix $A_N$ is of the form
\be\label{eq:defrsptmat}
A_N\ =\ \left(
\begin{array}{ccccccc}
b_0    &  b_1   &  b_2   &  \cdots  &  b_2   &  b_1  &  b_0  \\
b_1    &  b_0   &  b_1   &  \cdots  &  b_3   &  b_2  &  b_1  \\
b_2    &  b_1   &  b_0   &  \cdots  &  b_4   &  b_3  &  b_2  \\
\vdots & \vdots & \vdots &  \ddots  & \vdots & \vdots &  \vdots  \\
b_2    & b_3    & b_4    &  \cdots  & b_0    & b_1   &  b_2   \\
b_1    & b_2    & b_3    &  \cdots  & b_1    & b_0   &  b_1  \\
b_0    & b_1    & b_2    &  \cdots  & b_2    & b_1   &  b_0
\end{array}\right),
\ee
which is a symmetric Toeplitz matrix whose first row is a palindrome.
Bai \cite{Bai} first posed the problem of studying the limiting eigenvalue
distribution associated with random symmetric (non-palindromic) Toeplitz matrices,
along with Hankel and Markov matrices. Subsequent work by
Bose-Chatterjee-Gangopadhyay \cite{BCG}, Bryc-Dembo-Jiang \cite{BDJ}, and
Hammond-Miller \cite{HM} have independently observed that the limiting
distribution of random symmetric Toeplitz matrices is less than Gaussian.
In particular, \cite{HM} interpreted the deviations from the Gaussian in
terms of obstructions to Diophantine equations. Extending this work, Massey-Miller-Sinsheimer
\cite{MMS} proved that such Diophantine obstructions (and the deviations
they cause) vanish altogether if one considers symmetric palindromic Toeplitz
matrices. The analysis in \cite{MMS}
shows that the moments of the symmetric palindromic Toeplitz ensemble
are those of the standard Gaussian, and that the limiting spectral
measure converges weakly to the same. An $N \times N$ real
symmetric matrix $B$ has $N(N+1)/2$ degrees of freedom; in contrast, a symmetric palindromic Toeplitz matrix of the same dimensions has only $N/2$ degrees of
freedom. The PST ensemble is then a very thin sub-ensemble of all real symmetric matrices,
and the imposed structure leads to new behavior.
Thus by examining sub-ensembles of real symmetric matrices, one has the exciting possibility of
seeing new, universal distributions.

The entrance of random matrix theory into number theory would come two decades later in a fortuitous meeting between Hugh Montgomery and Freeman Dyson, yielding the observation that
the pair correlation function of Riemann zeta zeros matched that of the eigenvalues of random Hermitian matrices in the Gaussian Unitary Ensemble (see \cite{BFMT-B,FM} for a fuller treatment and history).  Work by Hejhal \cite{Hej} and Rudnick and Sarnak \cite{RS} extended this random matrix connection to $n$-level correlations of zeros of $L$-functions, generalizations of the Riemann zeta function which arise throughout number theory. Studying the zero density of an individual $L$-function can be then recast as the study of eigenvalue behavior of random complex Hermitian matrices.

In the study of $L$-functions, Rankin-Selberg convolution allows the creation of a new $L$-function from two input $L$-functions. Given families of $L$-functions $\left \lbrace L (s, f_i)_{f_i \in \mathcal{F}_i}\right\rbrace$ with  $i \in \lbrace 1, 2, \ldots, I \rbrace)$, the Rankin-Selberg convolution
\begin{equation}
\left \lbrace L ( s, f_1 \otimes \cdots \otimes f_I ) \right \rbrace_{ ( f_1, \ldots, f_I) \in \mathcal{F}_1 \times \cdots \times \mathcal{F}_I }
\end{equation}
gives a new family of $L$ functions; for details see \cite{IK}. Due\~nez and Miller \cite{DM1, DM2} were able to describe the behavior of the zeros of the convolution in terms of the behavior of the constituent families in many situations (see also \cite{SST}). As RMT has successfully modeled so many properties of $L$-functions, it is thus natural to ask if there is an RMT analogue of convolutions; trying to find this by combining properties of two families of matrices is the goal of this work.



The work of Goldmakher-Khoury-Miller-Ninsuwan \cite{GKMN} on the limiting eigenvalue distributions of weighted $d$-regular graphs provides one possibility of understanding combined ensemble behavior in terms of component behaviors. Given an adjacency matrix $A$ and a random weight matrix $\mathcal{W}$ populated by i.i.d.r.v. from appropriately bounded distributions, the analysis of \cite{GKMN} studies the limiting spectral measure of the Hadamard product $W \ast \mathcal{W}$ (this is the pointwise product of entries of the two matrices). Ongoing work by the authors of this paper generalizes the Hadamard product as a Kronecker product of two random square matrices from arbitrary ensembles (see \cite{Mor} for some results on the distribution of eigenvalues of Kronecker products).

Motivated by the preceding questions arising from the confluence
of quantum physics, number theory, and random matrix theory,
we consider the eigenvalue behavior of the
ensemble constructed as the ''disco'' concatenation of symmetric
palindromic Toeplitz matrices $A$ and real symmetric matrices $B$:
\begin{equation}\label{eq:1-disco-construction}
    \mathcal{D}_1 \left( A, B \right)\ = \ \begin{bmatrix}
    A & B \\
    B & A
    \end{bmatrix}.
\end{equation}
The whimsical
naming of the ``disco'' construction arises from the entries of the block
matrix ``ABBA'', a quintessential icon of disco music's heyday.
The resulting ensemble of $2N \times 2N$ symmetric block matrices have
only $(N/2) + N(N+1)/2$ degrees of freedom and constitute another thin
subset of all real symmetric matrices that may give rise to new
eigenvalue behavior of interest. The ensemble's construction from
known ensembles (symmetric palindromic Toeplitz and real symmetric)
furthermore poses the question of how the disco ensemble's
limiting eigenvalue distribution may be described in terms
of its constituent distributions.

We chose the PST and RS ensembles as their limiting distributions
(Gaussian and semicircle, respectively) exhibit behavior at polar
extremes. Computing the $2k$\textsuperscript{th} moments of the
Gaussian and semicircle distributions may be reformulated as a
combinatorics problem in which one must pair $2k$ points on the
circumference of a circle with chords possibly subject to additional
constraints. For the Gaussian case no such constraints
are placed, and all possible pairings of points on a circle
contribute equally to the moment in the limit $N \to \infty$. In contrast,
the semicircle case of the real symmetric matrices has equal contribution from all pairings that have no crossings, while pairings with a crossing contribute zero in the
limit $N \to \infty$. Furthermore, the Gaussian distribution
features a sharp decay rate but unbounded support, while the
semicircle distribution is strictly bounded within the interval $[-2,2]$.

\begin{figure}[ht]
    \centering
    \includegraphics[width=0.4\textwidth]{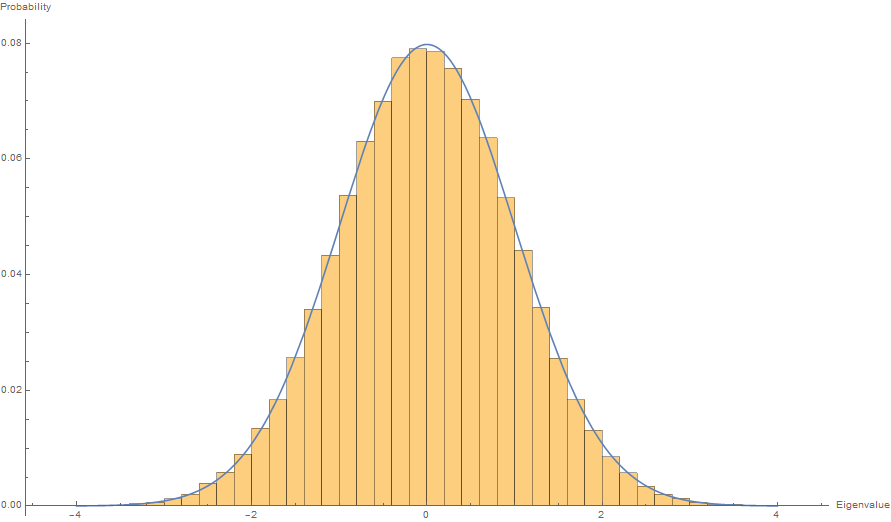}\ \ \  \includegraphics[width=0.4\textwidth]{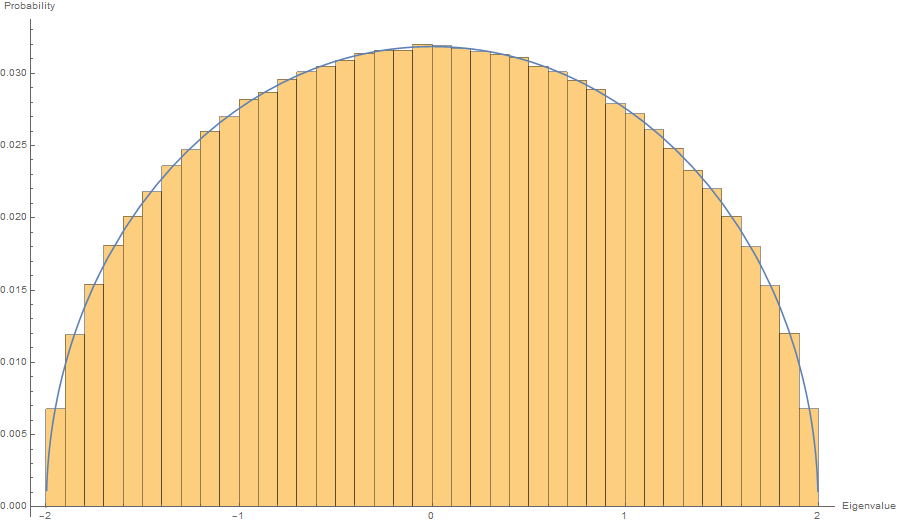}
    \caption{Eigenvalue distribution of $10,000 \times 10,000$ matrices: left is symmetric palindromic Toeplitz (plotted against a Gaussian), right is real symmetric (plotted against a semi-circle).}
    \label{fig:my_label}
\end{figure}

While the resulting ensembles do not appear to model convolution, our motivating question, the construction is of interest in its own right as another way to create ensembles and see how the properties of the constituent components are reflected in the new family. Our analysis shows that the new
construction of \eqref{eq:1-disco-construction}
exhibits hybrid behaviors that bear resemblance to the
limiting distributions of its component matrices, converging to a new
universal distribution distinct from both the Gaussian and semicircle,
while retaining similarities to both. We then extend this construction
in two ways. We consider arbitrary $A$, $B$ drawn from real symmetric
ensembles (with possibly additional structure imposed). We then delve
into the behavior of random block matrices constructed by successively
concatenating $\mathcal{D}_1 (A, B)$ with additional matrices drawn
from the same ensemble as $B$. Our work shows that given any two
random matrix from ensemble $\mathcal{E}_A$ and $\mathcal{E}_B$, one can
construct an infinite number of block matrix ensembles that converge to
any distribution intermediate to that of $\mathcal{E}_A$ and
$\mathcal{E}_B$. An entire spectrum of fascinating hybrid behavior
exists between any two limiting eigenvalue distributions, uncovering
a galaxy of new, universal distributions.


\subsection{Notation}\label{subsec:Notation}

We briefly review the notions of convergence examined in this paper
and define the quantities studied. We let $A$ be a random real symmetric
(with possibly additional structure imposed)
matrix of dimension $N \times N$ chosen from ensemble $\mathcal{E}_A$.
For all $i \in \mathbb{N}$, we let $B_i$ be a random real symmetric
(with possibly additional structure imposed)
matrix
of dimension $2^{i-1} N \times 2^{i-1} N$ drawn from an ensemble
$\mathcal{E}_B$. We then construct $\mathbf{B} = \lbrace B_i \rbrace$
as an infinite sequence of matrices. We assume both the limiting eigenvalue distribution of $\mathcal{E}_A$ and $\mathcal{E}_B$ have all moments
finite, and that the entries of $A$ and the $B_i$'s are drawn from a
fixed probability distribution $p(x)$ with mean 0 and variance 1.
We now define the $d$-Disco of $A$ and $\mathbf{B}$, denoted
$\mathcal{D}_d (A, \mathbf{B})$, as the following.

\begin{defi}
For $d \in \mathbb{Z}^+$, the $d$-Disco of $A$ and $\mathbf{B}$,
denoted $\mathcal{D}_d (A, \mathbf{B} )$, is given by
\begin{equation}
    \mathcal{D}_d (A, \mathbf{B} ) = \left[ \begin{array}{ccc}
        \begin{array}{cc}
         \begin{array}{cc}
            A & B_1 \\
            B_1 & A
        \end{array} & \text{\LARGE $B_2$}\\
        \text{\LARGE $B_2$} & \begin{array}{cc}
            A & B_1 \\
            B_1 & A
        \end{array} \end{array} & \cdots & \text{\fontsize{45}{0} $B_d$} \\
        \vdots & \ddots & \vdots \\
        \text{\fontsize{45}{0} $B_d$} & \cdots &  \begin{array}{cc}
         \begin{array}{cc}
            A & B_1 \\
            B_1 & A
        \end{array} & \text{\LARGE $B_2$}\\
        \text{\LARGE $B_2$} & \begin{array}{cc}
            A & B_1 \\
            B_1 & A
        \end{array} \end{array}
    \end{array} \right].
\end{equation}
\end{defi}
Observe that \eqref{eq:1-disco-construction} is a specific instance of the
preceding construction.

For each integer $2^d N$ let $\Omega_{2^d N}$ denote the set of
$2^d N \times 2^d N$
real symmetric matrices $\mathcal{D}_d = \mathcal{D}_d (A, \mathbf{B})$.
We construct a probability space $(\Omega_{2^d N},\f_N,\p_{2^d N})$ by setting
\begin{align}
\begin{split}
    &\p_{2^d N}\left(\left\{\mathcal{D}_d\in\Omega_{2^d N}: s_{i,j}(\mathcal{D}_d) \in [\ga_i, \gb_i]\
\right\}\right)\\
    &\hspace{2cm} = \left( \prod_{i=1}^{R_A} \int_{x_i=\ga_i}^{\gb_i} p(x_i)\; dx_i \right) \left( \prod_{k=1}^d \prod_{j=1}^{R_B} \int_{x_j=\ga_j}^{\gb_j} p(x_j)\; dx_j \right)
\end{split}
\end{align}
where each $dx_i$, $dx_j$ is the Lebesgue measure and $R_A$, $R_B$ are the degrees of
freedom in $A$ and the $B_i$'s, respectively. To each $\mathcal{D}_d\in\Omega_{2^d N}$
we attach a spacing measure by placing a point  mass of size $1/ 2^d N$ at
each normalized eigenvalue\footnote{From the eigenvalue trace lemma
($\Tr(\mathcal{D}_d^2) = \sum_i \gl_i^2( \mathcal{D}_d)$) and the Central
Limit Theorem, we see that the eigenvalues of $\mathcal{D}_1$ are of order
$\sqrt{2^d N}$. Since $\Tr(\mathcal{D}_1^2) = \sum_{i,j=1}^{2^d N}
s_{i,j}^2$ and each $s_{i,j}$ is drawn from a mean $0$, variance $1$ distribution,
$\Tr(\mathcal{D}_1^2)$ is of size $(2^d N)^2$, suggesting the appropriate scale
for normalizing the eigenvalues is to divide each by $\sqrt{2^d N}$.} $\gl_i(\mathcal{D}_d)$: \begin{equation}\label{eq:MkANl}
\mu_{\mathcal{D}_d}(x)dx \ = \ \frac{1}{2^d N} \sum_{i=1}^{2^d N} \delta\left( x -
\frac{\gl_i(\mathcal{D}_d)}{\sqrt{2^d N}} \right)dx,
\end{equation}
where $\delta(x)$ is the standard Dirac delta function. We call $\mu_{\mathcal{D}_d}$
the \emph{normalized spectral measure} associated to $\mathcal{D}_d$.

\begin{defi}[Normalized empirical spectral
distribution]\label{defi:nesd} Let $\mathcal{D}_d$ be an $2^d N \times 2^d N$
real symmetric matrix with eigenvalues $\lambda_{2^d N} \geq \cdots \geq
\lambda_1$. The normalized empirical spectral distribution (the empirical
distribution of normalized eigenvalues) $F^{\mathcal{D}_d /\sqrt{2^d N}}$ is
defined by
\begin{equation}
F^{\mathcal{D}_d /\sqrt{2^d N}}(x) \ =\ \frac{\#\{i \le 2^d N:
\lambda_i/\sqrt{2^d N} \le x\}}{2^d N}.
\end{equation}
\end{defi}

As $F^{ \mathcal{D}_d / \sqrt{2^d N}}(x) = \int_{-\infty}^x \mu_{\mathcal{D}_d}(t)dt$,
we see that $F^{\mathcal{D}_d /\sqrt{2^d N}}$ is the cumulative distribution function
associated to the measure $\mu_{\mathcal{D}_d}$. Our main tool to understand the
$F^{\mathcal{D}_d/\sqrt{2^d N}}$ is the Moment Convergence Theorem
(see \cite{Ta} for example).

\begin{thm}[Moment Convergence Theorem]\label{thm:momct} Let $\{F_N(x)\}$ be a
sequence of distribution functions such that the moments \be M_{m;N} = \int_{-\infty}^\infty x^m dF_N(x) \ee exist for all $m$. Let
$\Phi$ be the distribution function of the standard normal (whose
$m$\textsuperscript{{\rm th}} moment is $M_m$). If
$\lim_{N\to\infty} M_{m,N} = M_m$ then $\lim_{N\to\infty} F_N(x) =
\Phi(x)$. \end{thm}

\begin{defi}[Limiting spectral distribution]\label{defi:lsd}
If as $N\to\infty$ we have $F^{\mathcal{D}_d /\sqrt{2^d N}}$ converges in some
sense (for example, weakly) to a distribution $F$, then we say $F$ is the limiting
spectral distribution of the ensemble.
\end{defi}

The analysis proceeds by examining the convergence of the moments; to aid
in this analysis we separate $\mathcal{D}_d (A,\mathbf{B})$ into the sum
of random block matrices. We choose $B_0$ to be a random
$N \times N$ matrix from ensemble $\mathcal{E}_B$, and derive a random real symmetric matrix $C$ dependent on $A$ and $B_0$, given by
\begin{equation}\label{eq:dependent_matrix_cheat}
    C = A - B_0
\end{equation}
and construct
\begin{equation}\label{eq:curlyC_construction}
    \mathcal{C}_d = \begin{bmatrix}
    C & & \\
    & \ddots & \\
    & & C
    \end{bmatrix}
\end{equation}
with $d$ copies of $C$ placed along the diagonal. We then define
\begin{align}\label{eq:curlyB}
    \mathcal{B}_d &= \mathcal{D}_d (A, \mathbf{B}) - \mathcal{C}_d \nonumber \\
    &= \left[\begin{array}{ccc}
    \begin{array}{cc}
         \begin{array}{cc}
            B_0 & B_1 \\
            B_1 & B_0
        \end{array} & \text{\LARGE $B_2$}\\
        \text{\LARGE $B_2$} & \begin{array}{cc}
           B_0 & B_1 \\
            B_1 & B_0
        \end{array} \end{array} & \text{\resizebox{1.4cm}{!}{\itshape $B_3$}} & \cdots \\
        \text{\resizebox{1.4cm}{!}{\itshape $B_3$}} &  \begin{array}{cc}\begin{array}{cc}
            B_0 & B_1 \\
            B_1 & B_0
        \end{array} & \text{\LARGE $B_2$}\\
        \text{\Large $B_2$} & \begin{array}{cc}
            B_0 & B_1 \\
            B_1 & B_0
        \end{array} \end{array} & \cdots \\
        \vdots & \vdots & \ddots
    \end{array}\right]
\end{align}
so that $\mathcal{D}_d = \mathcal{B}_d + \mathcal{C}_d$.


\subsection{Main Results}\label{subsec:Main_Results}

By analyzing the moments of the $\mu_{\mathcal{D}_d}(x)$, we obtain results on
the convergence of $F^{\mathcal{D}_d /\sqrt{2^d N}}$ to a new, universal
distribution for each finite $d$ and to the distribution of $\mathcal{E}_B$
in the limit as $d \to \infty$.

The $m$\textsuperscript{th} moment of $\mu_{\mathcal{D}_d}(x)$ is
\begin{equation}
M_m(\mathcal{D}_d, A, \mathbf{B}, N)  =  \int_{-\infty}^\infty x^m
\mu_{\mathcal{D}_d}(x)dx =  \frac{1}{(2^d N)^{\frac{m}{2}+1}} \sum_{i=1}^{2^d N}
\gl_i^m(\mathcal{D}_d).
\end{equation}

\begin{defi}\label{defi:MmN}
Let $M_{m}(\mathcal{D}_d,N)$ be the average of $M_m(\mathcal{D}_d, A,
\mathbf{B}, N)$ over the ensemble, with each $\mathcal{D}_d$ weighted
by its distribution. Set $M_m(\mathcal{D}_d) = \lim_{N\to\infty} M_{m}
(\mathcal{D}_d, N)$. We call $M_m(\mathcal{D}_d)$ the average $m^{\rm th}$
moment.
\end{defi}

In Theorems \ref{thm:weak_upper_bound_even_moments} and
\ref{thm:weak_lower_bound_even_moments} we prove for the special case of
$d=1$ and $A$, $B$ being $N \times N$ symmetric palindromic Toeplitz matrices
and real symmetric matrices, respectively, that $M_m (\mathcal{D}_1 (A, B), N)$
converges to moments bounded above by the Gaussian's and below by the
semicircle's. We show in Section \ref{sec:combo_prob}
that computation of $M_m (\mathcal{D}_1 (A, B), N)$ may
be reformulated as a combinatorics problem of independent interest - namely,
counting the number of ways points on a circle may be paired when subjected
to restrictions on chord intersections. In Section \ref{subsec:weak_convergence}
we show that the
limiting spectral measure of $\mathcal{D}_1 (A, B)$ converges weakly to
a new universal distribution, and obtain a stronger result in Section
\ref{subsec:almost_sure_convergence} by proving almost sure
convergence. In Section \ref{sec:bounds_mixed_products} we use the $p$-Schatten norm and prove
a generalization H\"older's Inequality to bound the contribution of arbitrary
Hermitian matrix products. In Section \ref{sec:bounds_on_moments} we apply this bound to the special case of $d=1$
with arbitrary Hermitian $A$, $B$,
and bound its moments in terms of moments of component matrices.

In Section \ref{sec:LW_disco} we show that when taking the $d$-Disco of matrices from ensembles
with the same limiting spectral measure that the spectral measure of the
resulting matrix converges weakly to that of the original ensembles.
This result allows us to consider
both finite $d \in \mathbb{Z}^+$ and the limit as $d \to \infty$ for $A$,
$\mathbf{B}$ drawn from a pair of arbitrary real symmetric
(with possibly additional
structure) ensembles and prove that $M_m (\mathcal{D}_d, N)$ converges weakly
to the
moments of a new, universal distribution for $d$ finite and to the moments of
the $\mathbf{B}$ ensemble in the limit as $d \to \infty$. Once we show this,
then the same techniques used in \cite{HM} allow us to conclude the following.

\begin{thm}\label{thm:main}
For finite $d \in \mathbb{Z}^+$, the limiting spectral distribution of
$\mathcal{D}_d(A, \mathbf{B})$ whose independent entries are independently
chosen from a probability distribution $p$ with mean $0$, variance $1$ and
finite higher moments converges weakly to a new, universal distribution
independent of $p$. As $d \to \infty$, the limiting spectral distribution
of $\mathcal{D}_d(A, \mathbf{B})$ converges weakly to that of the $\mathbf{B}$
ensemble.
\end{thm}
We sketch the proof, which relies on Markov's method of moments,
which is well suited to random matrix theory problems and many
questions in probabilistic number theory (see \cite{Ell}).
By the Eigenvalue Trace Lemma,
\begin{equation}
    \sum_{i=1}^{2^d N} \lambda_i^m \left( \mathcal{D}_d \right) \ = \
    \Tr\left( \mathcal{D}_d^m \right),
\end{equation}
which applied to the ensemble of $\mathcal{D}_d$ matrices yields
\begin{align}
    M_{m} (\mathcal{D}_d, N) &= \frac{1}{\left( 2^d N \right)^{\frac{m}{2}+1} }
    \mathbb{E} \left[ \Tr\left( \mathcal{D}_d^m \right) \right] \nonumber \\
    &=  \frac{1}{\left( 2^d N \right)^{\frac{m}{2}+1} } \mathbb{E} \left[ \Tr
    \left( \left( \mathcal{B}_d + \mathcal{C}_d \right)^m \right) \right] \nonumber \\
    &= \frac{1}{\left( 2^d N \right)^{\frac{m}{2}+1} }  \sum_{\ell = 0}^m \;
    \sum_{\substack{i_1 + \cdots i_p = m - \ell \\ j_1 + \cdots + j_p = \ell}}
    \mathbb{E} \left[ \Tr\left( \prod_{i=1}^p \mathcal{B}_d^{i_p}
    \mathcal{C}_d^{j_p} \right) \right]
\end{align}
where by $\E[ \cdots ]$ we mean averaging over the $2^d N\times 2^d N$
$\mathcal{D}_d(A, \mathbf{B})$ ensemble with each matrix $\mathcal{D}_d$
weighted by its probability of occurring. Expansion of the product
$\left( \mathcal{B}_d + \mathcal{C}_d \right)^m$ yields a non-commutative,
bivariate matrix polynomial; the chief obstacle becomes determining in the
limit as $N \to \infty$ the contribution of terms with general form
\begin{equation}
    \mathbb{E}\left[ \Tr\left( \mathcal{B}_d^{i_1} \mathcal{C}_d^{j_1}
    \cdots \mathcal{B}_d^{i_p} \mathcal{D}_d^{j_p} \right) \right], \quad
    \text{where}\quad \sum_{k=1}^p \left( i_k + j_k \right) = m.
\end{equation}
Weak convergence for the case $d=1$ follows from
\begin{equation}
\lim_{N \to \infty} \left(\E \left[ M_m( \mathcal{D}_1 )^2 \right] - \E
\left[ M_m( \mathcal{D}_1 ) \right]^2 \right) \ = \ 0
\end{equation}
and applying Chebyshev's inequality and the Moment Convergence Theorem.
We then establish convergence for $d \in \mathbb{Z}^+$ by inducting on
the parameter $d$.

We conclude in Section $8$ by investigating the spacings between normalized eigenvalues of $\mathcal{D}_d(A,\mathbf{B})$ constructed from $A$ a symmetric palindromic Toeplitz matrix and $\mathbf{B}$ a sequence of real symmetric matrices, and posing conjectural bounds on moments of $\mathcal{D}_1(A,B)$ with $A,B$ drawn from ensembles with different limiting spectral measure.


\section{1-Disco of PST and RS Matrices}

The 1-Disco of a symmetric palindromic Toeplitz matrix and a real symmetric
matrix highlights the challenge of analyzing the concatenation of matrices
from ensembles with different limiting spectral distributions. For sake of
completeness, we restate the construction of the $d$-Disco with $d=1$.
\begin{defi}\label{def: 1-disco}
For two $N \times N$ real matrices $A$ and $B$, write
$\mathcal{D}_1 = \mathcal{D}_1(A, \mathbf{B})$ for the $2N \times 2N$ block
matrix
\begin{equation}
    \mathcal{D}_1 = \left[
        \begin{array}{cc}
             A & B \\
             B & A
        \end{array} \right].
\end{equation}
\end{defi}
Let $p$ be a probability distribution with mean $0$, variance $1$, and
finite moments of all orders. Let $A = [a_{i,j}]$ denote an $N \times N$
symmetric palindromic Toeplitz (PST) random matrix whose entries are
i.i.d.r.v. with probability distribution $p$, and $B = [b_{i,j}]$ denote
an $N \times N$ symmetric random matrix whose entries are i.i.d.r.v. with
probability distribution $p$.


\subsection{Determination of the Moments \texorpdfstring{$M_k \left( \mathcal{D}_1 \right)$}{Lg}}

We wish to study the limiting behavior of the $k^{\text{th}}$ moment
$M_k(\mathcal{D}_1,N)$ of the distribution of normalized eigenvalues
of $\mathcal{D}_1$ as $N \to \infty$. Observe that we may diagonalize
$\mathcal{D}_1$ in the following manner:
\begin{align}
    \mathbb{E}[\Tr(\mathcal{D}_1^k)] \ &=\ \mathbb{E}\left[\Tr
    \begin{bmatrix} A & B \\ B & A \end{bmatrix}^k\right] \nonumber \\
    \ &=\ \mathbb{E}\left[\Tr\begin{bmatrix} I/2 & I/2 \\
    I/2 & -I/2 \end{bmatrix} \begin{bmatrix} (A+B)^k & 0 \\
    0 & (A-B)^k \end{bmatrix} \begin{bmatrix} I & I \\
    I & -I \end{bmatrix}\right] \nonumber\\
    \ &= \ \mathbb{E}\left[\Tr\begin{bmatrix} (A+B)^k & 0 \\
    0 & (A-B)^k \end{bmatrix}\right].\label{eq:1-disco_decomp}
\end{align}
Recalling the Eigenvalue Trace Lemma and applying the preceding
diagonalization yields
\begin{align}
    M_k(\mathcal{D}_1,N) \ &= \ \frac{1}{(2N)^{\frac{k}{2} + 1}}
    \mathbb{E}[\Tr(\mathcal{D}_1^k)] \nonumber \\
    \ &= \ \frac{1}{(2N)^{\frac{k}{2} + 1}} \mathbb{E}[\Tr\left((A+B)^k
    +(A-B)^k\right)] \label{eq:1-disco-diagonalized}  \\
    \ &= \ \frac{2}{(2N)^{\frac{k}{2} + 1}} \sum_{\substack{l=0\\
    l \text{ even}}}^k \sum_{\substack{I_1 + \cdots + I_p = k-l\\
    J_1 + \cdots + J_p = l}} \mathbb{E}[\Tr(A^{I_1} B^{J_1} A^{I_2}
    B^{J_2} \cdots )] \label{eq1.1}  \\
    \begin{split}
        \ &= \ \frac{2}{(2N)^{\frac{k}{2} + 1}} \sum_{\substack{l=0\\
        l \text{ even}}}^k \sum_{\substack{I_1 + \cdots + I_p = k-l\\
        J_1 + \cdots + J_p = l }}\\
        &\qquad \sum_{1 \leq i_1, \ldots, i_k \leq 2N} \mathbb{E}
        [a_{i_1, i_2} a_{i_2, i_3} \cdots a_{i_{I_1}, i_{I_1 + 1}}
        b_{i_{I_1 + 1}, i_{I_1 + 2}} \cdots b_{i_k, i_1}]. \label{eq1.2}
    \end{split}
\end{align}

Consider an arbitrary term $\mathbb{E}[a_{i_1, i_2} \cdots a_{i_{I_1},
i_{I_1+1}} b_{i_{I_1 + 1}, i_{I_1 + 2}} \cdots b_{i_k, i_1}]$. As the
expected value of a product of independent random variables is the
product of their expected values, if one of the $a$'s occurs exactly
once in the term, then since $p$ has mean $0$, the entire term vanishes.
A similar principle applies to the $b$'s, so the only nonzero terms of
$\eqref{eq1.2}$ are those in which each of the $a$'s and $b$'s occur at
least twice in the product.

\subsubsection{Second Moment and Odd Moments}

\begin{lem}\label{oddmoments}
Assume $p$ has mean zero, variance one and finite higher moments.
Then $M_2 \left( \mathcal{D}_1 \right) = 1$, and for all odd
$k$, $M_{k} \left( \mathcal{D}_1 \right) = \lim_{N \to \infty}
M_{k} \left( \mathcal{D}_1, N \right) = 0$.
\begin{proof}
From equation \eqref{eq1.2} we have
\begin{align}
    M_2 \left( \mathcal{D}_1 \right) \ &= \ \lim_{N \to \infty}
    M_2 \left( \mathcal{D}_1, N \right) \nonumber\\
    \ &= \ \lim_{N \to \infty} \frac{2}{(2N)^2}\left(\mathbb{E}
    [\Tr(A^2)] + \mathbb{E}[\Tr(B^2)]\right) \nonumber\\
    \ &= \ \frac{1}{2} \left(M_2(A) + M_2(B)\right).
\end{align}
We know from \cite{MMS} and \cite{Wig2} that $M_2(A) = 1$ and $M_2(B) = 1$,
respectively. It follows immediately that that $M_2(\mathcal{D}_1) = (1 + 1) / 2 = 1$.

For odd $k$ we adopt a similar argument to that used in Lemma $2.3$ of
\cite{MMS}. Assume $k = 2r + 1$ is odd. In each nonzero term of
\eqref{eq1.2}, one of the $a$'s or $b$'s occurs with multiplicity at least
$3$, and as established above each of the $a$'s and $b$'s must occur with
multiplicity at least two. A non-zero term of \eqref{eq1.2} is then completely
determined by first specifying a diagonal for each grouping of $a$'s and
each grouping of $b$'s (making at most $r$ choices), then choosing the index $i_1$.
Such designations force the values of all subsequent indices (up to a fixed
number of choices), so there are at most $O(N^{r+1})$ choices.

Let $n_a$ and $m_b$ be the multiplicity of a grouping of $a$ and $b$,
respectively. A given term contributes
\begin{equation}
\left(\prod_{\text{groupings of } a\text{'s}} \mathbb{E}[a^{n_a}] \right)
\left(\prod_{\text{groupings of } b\text{'s}} \mathbb{E}[b^{m_b}] \right)
\ = \ O(1) \end{equation}
since all moments of $p$ are finite. Thus, \eqref{eq1.2} implies that
$M_k(\mathcal{D}_1,N) = O(1 / \sqrt{N})$; hence
$M_k(\mathcal{D}_1) = 0$ as claimed.
\end{proof}
\end{lem}

\subsubsection{Even Moments}

We calculate the even moments $M_{2k}(\mathcal{D}_1)$. By \eqref{eq1.2}, $M_{2k}(\mathcal{D}_1,N)$ can be expanded as
\begin{equation}\label{expression: even moment of 1-disco}
    \frac{2}{(2N)^{k + 1}} \sum_{\stackrel{l=0}{l \text{ even}}}^{2k}
    \sum_{\stackrel{I_1 + \cdots + I_p = 2k-l}{J_1 + \cdots + J_p = l}}
    \sum_{1 \leq i_1, \ldots, i_{2k} \leq 2N} \mathbb{E}[a_{i_1, i_2}
    a_{i_2, i_3} \cdots a_{i_{I_1}, i_{I_1 + 1}} b_{i_{I_1 + 1}, i_{I_1 + 2}}
    \cdots b_{i_{2k}, i_1}].
\end{equation}
If any $a$'s or $b$'s occur to the first power, the expected value is
zero as each $a$, $b$ is drawn from a mean 0 distribution; hence, all
$a$'s and $b$'s must be at least paired. If on the other hand any $a$'s
or $b$'s occur to a third or higher power, there are then fewer than
$k+1$ degrees of freedom, and there will be no contribution in the limit.
Since $a$'s and $b$'s are entries of matrices
from different ensembles, they exhibit different matching behaviors.

If $a_{i_m.i_{m+1}} = a_{i_n, i_{n+1}}$, there are three possibilities:
\begin{align}\label{eq:MMS_system}
    i_{m+1}-i_m \ &=\ \pm(i_{n+1}-i_n) \nonumber\\
    i_{m+1}-i_m \ &=\ \pm(i_{n+1}-i_n)+(N-1)\\
    i_{m+1}-i_m \ &=\ \pm(i_{n+1}-i_n)-(N-1). \nonumber
\end{align}

If $b_{i_m, i_{m+1}} = b_{ i_n, i_{n+1}}$, we must have:
\begin{equation}
    i_{m+1}-i_m \ = \ \pm(i_{n+1}-i_n).
\end{equation}

Thus the equations in \eqref{eq:MMS_system} can be written more concisely
by considering a choice of
$C_t \in \{0, N-1, -(N-1)\}$ (where $t$ is a function of $i_m$, $i_{m+1}$,
$i_n$, $i_{n+1}$) such that
\begin{equation}\label{eq1.6}
    i_{m+1}-i_m \ = \ \pm(i_{n+1}-i_n)+C_t.
\end{equation}

We have in total $k$ such equations since there are $2k$ terms and
everything is paired. Notice that since the $b$'s are from a real
symmetric matrix, the only possible choice for $C_t$ is $0$.

Let $x_1,\dots,x_k$ denote the absolute values $\abs{i_{m+1}-i_m}$ of
the left hand side of these $k$ equations. Define $\Tilde{x}_1=i_2-i_1,\;\Tilde{x}_2=i_3-i_2,\dots,\Tilde{x}_{2k}=i_1-i_{2k}$.
We have
\begin{align}
    i_2 \ &= \ i_1 + \Tilde{x}_1 \nonumber \\
    i_3 \ &= \ i_1+\Tilde{x}_1+\Tilde{x}_2 \nonumber \\
    &\;\;\vdots \nonumber \\
    i_1 \ &= \ i_1+\Tilde{x}_1+\Tilde{x}_2+\dots+\Tilde{x}_{2k}.
\end{align}
From the last equation, we get
\begin{equation}\label{eq1.7}
    \Tilde{x}_1+\dots+\Tilde{x}_{2k} \ = \ 0.
\end{equation}
Arguing as in \cite{MMS}, there exists an $\eta_t = \pm 1$ such that
$i_{m+1}-i_m=\eta_t x_t$. Substituting into \eqref{eq1.6}, we have
\begin{equation}
    \Tilde{x}_n \ = \ \eta_t \epsilon_t x_t - \epsilon_t C_t
\end{equation}
where $\epsilon_t=\pm1$. Therefore each $x_t$ is associated to two
$\Tilde{x}$'s, and occurs exactly twice, once as $\Tilde{x}_m = \eta_t x_t$
and again as $\Tilde{x}_n = \eta_t \epsilon_t x_t - \epsilon_t C_t$.
Substituting for the $\Tilde{x}$s in \eqref{eq1.7},
\begin{equation}\label{epsilon}
    \sum\limits_{m=1}^{2k} \Tilde{x}_m \ = \ \sum\limits_{t=1}^k
    (\eta_t(1+\epsilon_t)x_t - \epsilon_tC_t) \ = \ 0.
\end{equation}
If any $\epsilon_t=1$, then the $x_t$ are not linearly independent,
and there are less than $k+1$ degrees of freedom. Thus the terms
where at least one $\epsilon_t = 1$ contribute $O_k(1 / N)$ to
$M_{2k}(\mathcal{D}_1)$, and are negligible in the limit. Hence
$\epsilon_t=-1$ for all $t$.
Substituting $\epsilon_t = -1$ into \eqref{epsilon} gives
$\sum_{t=1}^k C_t = 0$. We have proven the following lemma.

\begin{lem}\label{lem:symm_elts_opp_diags}
    A summand of \eqref{expression: even moment of 1-disco} that contributes in the limit $N \to \infty$ has all $b$'s paired; furthermore,
    for a given pair $b_{i_m,i_{m+1}}=b_{i_n,i_{n+1}}$, the indices satisfy $i_m = i_{n+1}$ and
    $i_{m+1} = i_n$.
\end{lem}

For any given $k$, the following lemmas allow us to calculate the
coefficient of a term $A^{I_1}B^{J_1} \cdots A^{I_p} B^{J_p}$.

\begin{lem}
The number of terms of \eqref{eq1.1} which are equivalent to
$A^{I_1}B^{J_1} \cdots A^{I_p} B^{J_p}$ up to cyclic permutation
is given by
\begin{equation}
    \# \text{ equivalent terms } \ = \ \frac{2k}{|S_{A^{I_1} B^{J_1}
    \cdots A^{I_p} B^{J_p}}|} \label{eq1.3}
\end{equation}
where $S_{A^{I_1} B^{J_1} \cdots A^{I_p} B^{J_p}}$ is the set of
those cycles which fix $A^{I_1} B^{J_1} \cdots A^{I_p} B^{J_p}$.
\begin{proof}
Let $G = \langle \cycle{1,2,\cdots,k} \rangle \leq S_k$ and let
$X$ denote the set of terms of $\eqref{eq1.1}$ which are equivalent
to $A^{j_1}B^{j_2} \cdots$ up to cyclic permutation. Then $G$ acts
transitively on $X$ by permuting the factors, and $G$ has order $k$.
So \eqref{eq1.3} follows by applying the Orbit-Stabilizer counting
formula and remembering the factor of $2$ from \eqref{eq1.1}.
\end{proof}
\end{lem}

\begin{lem}
To find all terms in \eqref{eq1.1} with coefficient $c$, proceed
as follows.
\begin{enumerate}
    \item Check that $2|c$ and $c|2k$. If not, then there are no
    terms with coefficient $c$. Otherwise, proceed.
    \item The terms are precisely those of the form $A^{I_1} B^{J_1}
    \cdots A^{I_p} B^{J_p}$ for which there is some $m|k$ such that
    $\sum_{\ell=1}^{m} I_{\ell} + J_{\ell} = c / 2$, $(k \sum_{\ell=1}^m
    j_{\ell}) / m$ is even, the sequence $(I_1, J_1, \cdots,
    I_m, J_m)$ is non-repeating, and $(I_1, J_1, \cdots, I_k, J_k)$
    is obtained by repeating $(I_1, J_1, \cdots, I_m, J_m)$ exactly
    $k / m$ times.
\end{enumerate}
\begin{proof}
By Lemma $3$, the term $A^{I_1} B^{J_1} \cdots A^{I_{k}} B^{J_k}$
has coefficient $c$ in \eqref{eq1.1} if and only if
\begin{equation}
    c \ = \ \frac{2k}{|S_{A^{I_1} B^{J_1} \cdots A^{I_{k}} B^{J_k}}|}
\end{equation}
or equivalently, the stabilizer of $A^{I_1} B^{J_1} \cdots A^{I_{k}}
B^{J_k}$ has order $2k / c$. So necessarily $c|2k$, and
furthermore by Lagrange's Theorem we must also have $(2k / c)|k$.
This latter condition is equivalent to the requirement $2|c$. Assume
that both of these hold. Then there is a unique such subgroup of the
cyclic subgroup generated by the permutation $\cycle{1,2,\cdots, k}$,
denoted $\langle \cycle{1,2,\cdots, k} \rangle$ of order $2k / c$,
which is given by $H \coloneqq S_{A^{I_1} B^{J_1} \cdots A^{I_k} B^{J_k}} = \langle
\cycle{1,2,\cdots, k}^{c / 2} \rangle$. One may observe that the
terms with stabilizer $H$ are precisely those of the form $A^{I_1} B^{J_1}
\cdots A^{I_k} B^{J_k}$ such that there is some $m|k$ such that
$\sum_{\ell = 1}^{m} I_{\ell} + J_{\ell} = c / 2$,
$(k \sum_{\ell=1}^m J_{\ell}) / m$ is even, $(I_1, J_1, \cdots, I_m, J_m)$
is non-repeating, and $(I_1, J_1, \cdots, I_k, J_k)$ is obtained by
repeating $(I_1, J_1, \cdots, I_{m}, J_m)$ a total of $k / m$ times.
\end{proof}
\end{lem}

\subsubsection{The Fourth Moment} We calculate the fourth moment in detail,
as the calculation shows the new, hybrid pairing behavior of the indices. This
will establish the techniques that we use to analyze general even moments.
Let $G_{2k}$ and $S_{2k}$ denote the $2k$\textsuperscript{th} moments of
the Gaussian and semicircle distributions, respectively.
We recall that the $2k$\textsuperscript{th} moment of
Gaussian distribution is given by $(2k-1)!!$ while the $2k$\textsuperscript{th}
moment of the semicircle is given by the $k$\textsuperscript{th} Catalan number
\begin{equation}
    S_{2k}\ =\ \frac{1}{k+1} \binom{2k}{k}.
\end{equation}
Comparing the moments of the disco matrix $\mathcal{D}_1 (A, \mathbf{B})$ to those
of the Gaussian and semicircle offer insight into how the disco structure
creates a fascinating hybrid of disparate limiting distributions.

\begin{thm}
The average fourth moment of $\mathcal{D}_1$ is
\begin{equation}
    M_4 \left( \mathcal{D}_1,N \right) \ = \ \frac{9}{4} + O
    \left( \frac{1}{N} \right).
\end{equation}
\begin{proof}

We wish to study the limit as $N \to \infty$ of the following:
\begin{equation}
     M_4 \left( \mathcal{D}_1,N \right) \ = \ \frac{1}{(2N)^3}
     \mathbb{E}[\Tr(\mathcal{D}_1^4)].
\end{equation}
Expanding and applying the cyclic property of the trace operator,
we see that
\begin{align}\label{eq:1D_4th_mom_exp}
    \mathbb{E}[\Tr(\mathcal{D}_1^4)] \ &=\ \mathbb{E}[\Tr((A+B)^4)
    + \Tr((A-B)^4)] \nonumber \\
    \begin{split}
    \ &=\ 2 \mathbb{E}[\Tr(A^4)] + 8 \mathbb{E}[\Tr(A^2 B^2)]\\
    &\hspace{2cm} + 4 \mathbb{E}[\Tr(ABAB)] + 2\mathbb{E}[\Tr(B^4)].
    \end{split}
\end{align}
We proceed by analyzing each term of \eqref{eq:1D_4th_mom_exp};
for the limit
\begin{equation}
    \lim_{N \to \infty} \frac{1}{(2N)^3} \mathbb{E}[\Tr(B^4)] \ = \
    \frac{1}{4} \label{eq2.1}
\end{equation}
one can see \cite{Wig2} and \cite{Meh}. Furthermore, \cite{MMS}
calculated that
\begin{equation}
    \lim_{N \to \infty} \frac{1}{(2N)^3} \mathbb{E} \left[
    \Tr(A^4) \right] \ = \ \frac{3}{8} \label{eq2.2}.
\end{equation}
Thus we need only consider the terms containing $\mathbb{E}[\Tr(A^2 B^2)]$
and $\mathbb{E}[\Tr(ABAB)]$ in \eqref{eq:1D_4th_mom_exp}. Consider the
former, we see by direct computation that
\begin{equation}
    \mathbb{E}[\Tr(A^2 B^2)] \ = \ \sum_{1 \leq i_1, i_2, i_3, i_4 \leq 2N}
    \mathbb{E}[a_{i_1, i_2} a_{i_2, i_3}] \mathbb{E}[b_{i_3, i_4} b_{i_4, i_1}].
\end{equation}
If $a_{i_1, i_2} \neq a_{i_2, i_3}$, then the expected value of the product
is the product of the expected values, both of which are zero by the assumption
that $p$ has mean $0$; the same holds true for the $b$'s. Hence we need only
consider terms wherein $a_{i_1, i_2} = a_{i_2, i_3}$ and
$b_{i_3, i_4} = b_{i_4, i_1}$, which yield the system of equations
\begin{align}
    i_2 - i_1 \ &=\ \epsilon_1 (i_3-i_2) + C \label{eq2.3} \\
    i_4 - i_3 \ &=\ \epsilon_2 (i_1 - i_4) \label{eq2.4}
\end{align}
where $\epsilon_j \in \{\pm 1\}$ and $C_j \in \{0,N-1, 1-N\}$. The discussion
immediately preceding Lemma \ref{lem:symm_elts_opp_diags} implies we need
consider only the cases in which $\epsilon_1 = -1 = \epsilon_2$.

First assume that $C = N-1$. Then $\eqref{eq2.3}$ implies that $i_3 = N$ and
$i_1 = 1$. Substituting into $\eqref{eq2.4}$ gives $0 = 1-N$, which is not
possible for large $N$. A similar argument shows that we cannot have $C = 1-N$,
so $\eqref{eq2.3}$ and $\eqref{eq2.4}$ become
\begin{align}
    i_2 - i_1 \ &=\ -(i_3 - i_2) \label{eq2.5} \\
    i_4 - i_3 \ &=\ -(i_1 - i_4) \label{eq2.6}.
\end{align}
$\eqref{eq2.5}$ and $\eqref{eq2.6}$ imply that $i_3 = i_1$, and that $i_2$
and $i_4$ are free parameters. There are therefore $N^3$ choices for tuples
$(i_1, i_2, i_3, i_4)$ such that $a_{i_1, i_2} = a_{i_2,i_3}$ and
$b_{i_3, i_4} b_{i_4, i_1}$, and since $p$ has variance $1$, we obtain
\begin{equation}
    \mathbb{E}[\Tr(A^2 B^2)] \ = \ N^3 \label{eq2.7}.
\end{equation}

Now consider the term
\begin{equation}
    \mathbb{E}[\Tr(ABAB)] \ = \ \sum_{1 \leq i_1, i_2, i_3, i_4 \leq n}
    \mathbb{E}[a_{i_1, i_2} a_{i_3, i_4}]
    \mathbb{E}[b_{i_2, i_3} b_{i_4,i_1}] \label{eq2.8}.
\end{equation}
By a similar argument to that given for the $\mathbb{E}[\Tr(A^2 B^2)]$
term, we consider only those terms in which $a_{i_1, i_2} = a_{i_3, i_4}$
and $b_{i_2, i_3} = b_{i_4, i_1}$. For a fixed tuple of indices
$(i_1, i_2, i_3, i_4)$, the symmetry of $B$ guarantees that either
$i_2 = i_4$ and $i_3 = i_1$, or $i_2 = i_1$ and $i_3 = i_4$ holds.
We therefore have $O(N^2)$ terms in $\eqref{eq2.8}$ which
are nonzero. Hence
\begin{equation}
    \mathbb{E}[\Tr(ABAB)] \ = \ O(N^2) \label{eq2.9}
\end{equation}
which vanishes in the limit $N \to \infty$.

Combining $\eqref{eq2.1}$, $\eqref{eq2.2}$, $\eqref{eq2.7}$, and
$\eqref{eq2.8}$, we see that
\begin{equation}
    M_4(\mathcal{D}_1) \ = \ \lim_{N \to \infty} \frac{1}{(2N)^3}
    \mathbb{E}[\Tr(\mathcal{D}_1^4)] \ = \ 2 \cdot \frac{3}{8} +
    8 \cdot \frac{1}{8} + 4 \cdot 0 + 2 \cdot \frac{1}{4} \ = \
    \frac{9}{4}.
\end{equation}
Comparing this with $G_{4} = 3$ and $S_{4} = 2$, we can see that
the fourth moment of $\mathcal{D}_1$ is bounded by those of the
Gaussian and semicircle.
\end{proof}
\end{thm}

\subsubsection{Sixth and Eighth Moments}

Any even moment can be determined through brute-force calculation,
though deriving exact formulas as $k \to \infty$ requires handling
involved combinatorics\footnote{See Section \ref{sec:combo_prob}
for a fuller treatment of the combinatorial obstacles to higher
moment calculations.}. Brute-force computation gives the sixth
and eight moments of $\mathcal{D}_1$.

\begin{center}
\begin{tabular}{|c|c|c|c|}\hline
Moment & Semicircle & \hspace{0.5cm} $\mathcal{D}_1$ \hspace{0.5cm}
& Gaussian \\ \hline
6 & 5  & 7    & 15  \\ \hline
8 & 14 & 27.5 & 105 \\ \hline
\end{tabular}
\end{center}

To calculate the $2k$\textsuperscript{th} moment, we may consider
$2h$ of the $a$'s and $2j$ of the $b$'s, where $h, j \in \mathbb{Z}^+$,
placed upon
the circumference of a unit circle. All $a$'s and $b$'s must be paired;
the chords pairing $a$'s are allowed to cross, but nothing may cross
the chords pairing $b$'s. Counting the number of valid pairing
configurations is equivalent to determining the contribution of a
given term in \eqref{eq1.2}. Examples are illustrated in Figures
\ref{fig:all_pairings_A4B2}, \ref{fig:contributing_pairings_A2B2A2B2},
and \ref{fig:non_contributing_pairings_A2B2A2B2}, where hollow dots
represent $a$'s and solid dots represent $b$'s; dashed chords
represent pairings of $a$'s, while solid chords represent pairings
of $b$'s.

In the calculation of the
$6^{\text{th}}$ moment, the contribution of $\mathbb{E} \left[ A^4B^2
\right]$ may be visualized in Figure \ref{fig:all_pairings_A4B2}.
Observe that all pairing configurations contribute; this stems from
the fact that there is only one way to pair the two $b$'s, while
any pairing configuration of the $a$'s does not affect
contribution (a property inherited from the Gaussian behavior of
symmetric palindromic Toeplitz matrices).

\begin{figure}[ht]
    \centering
    \includegraphics[width=0.2\textwidth]{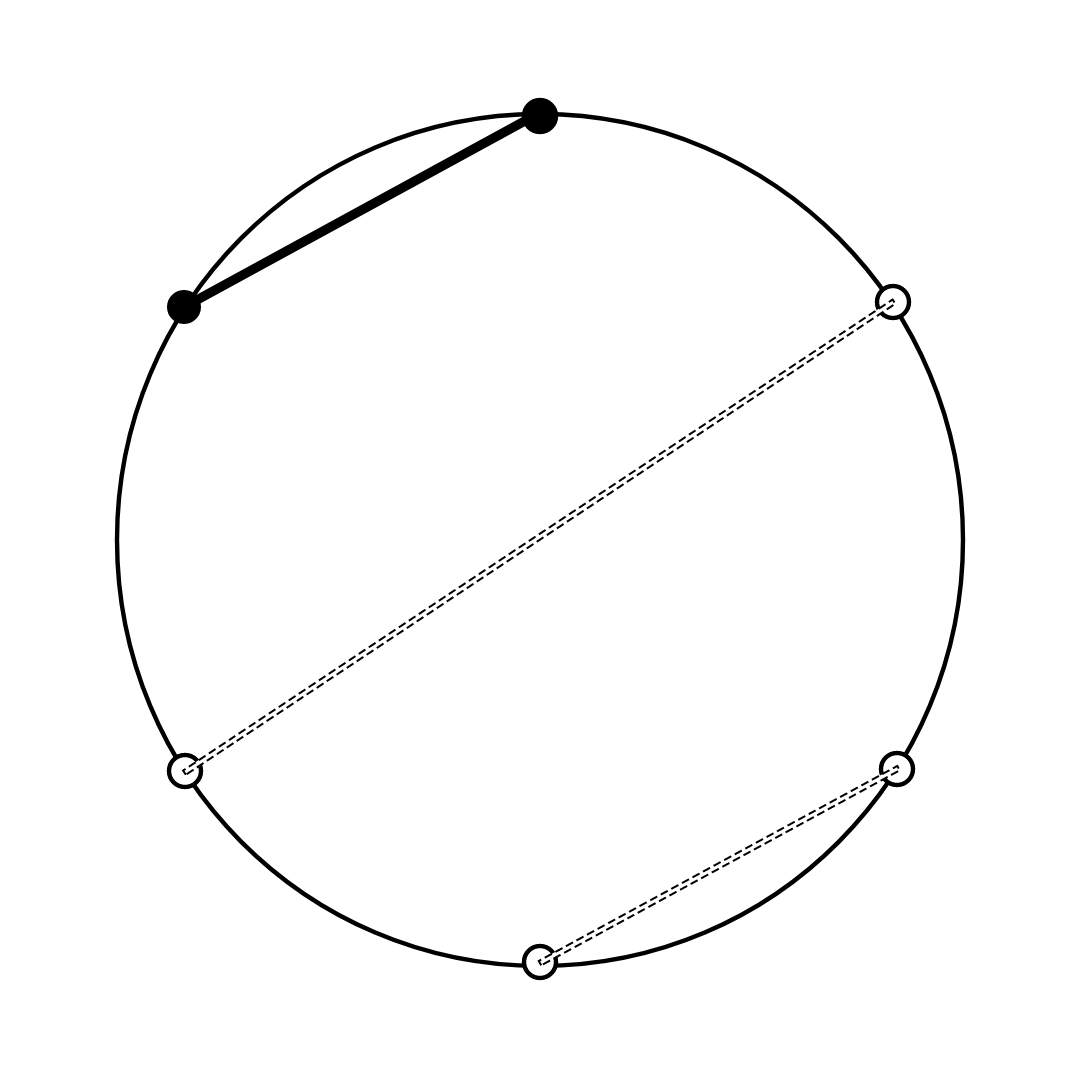}
    \includegraphics[width=0.2\textwidth]{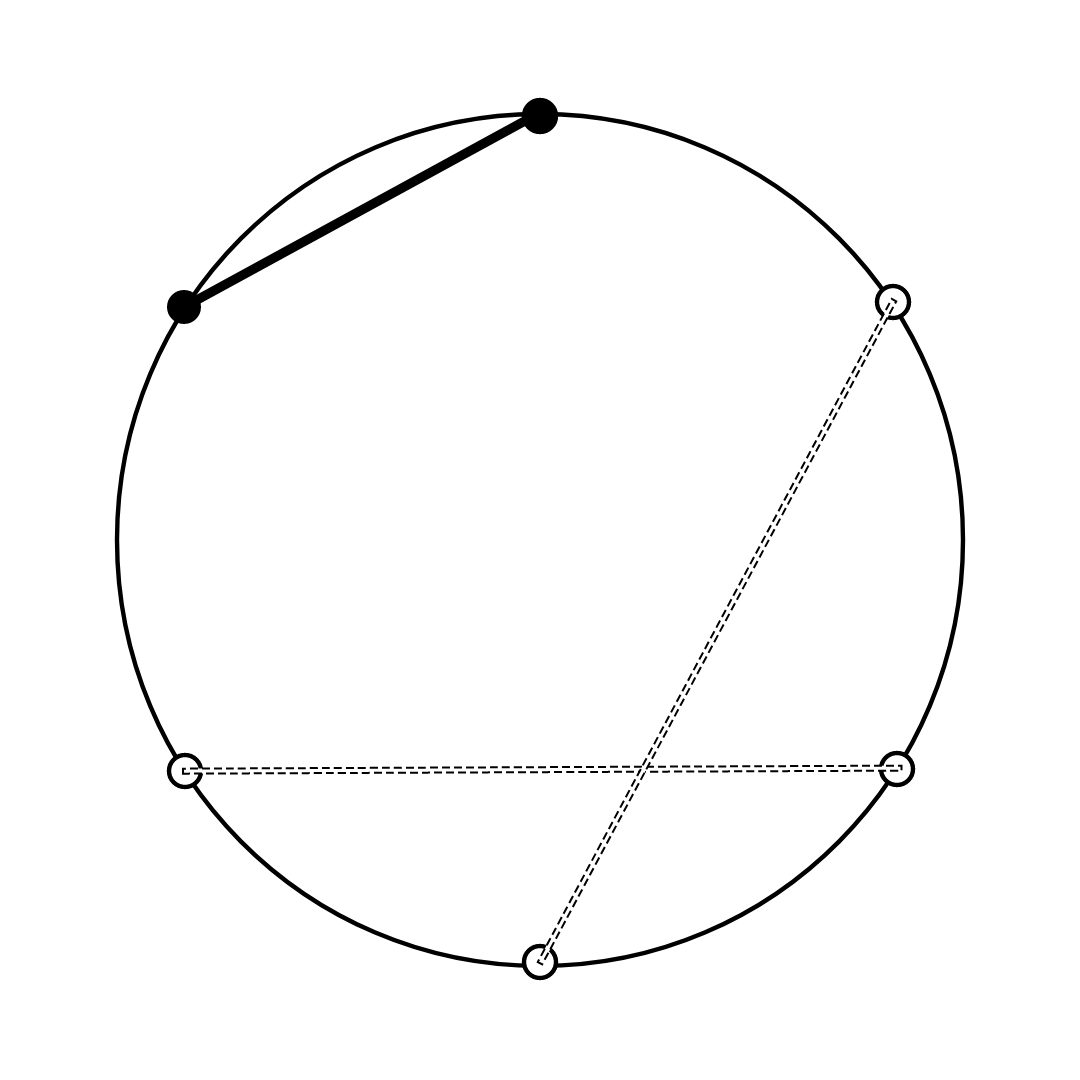}
    \includegraphics[width=0.2\textwidth]{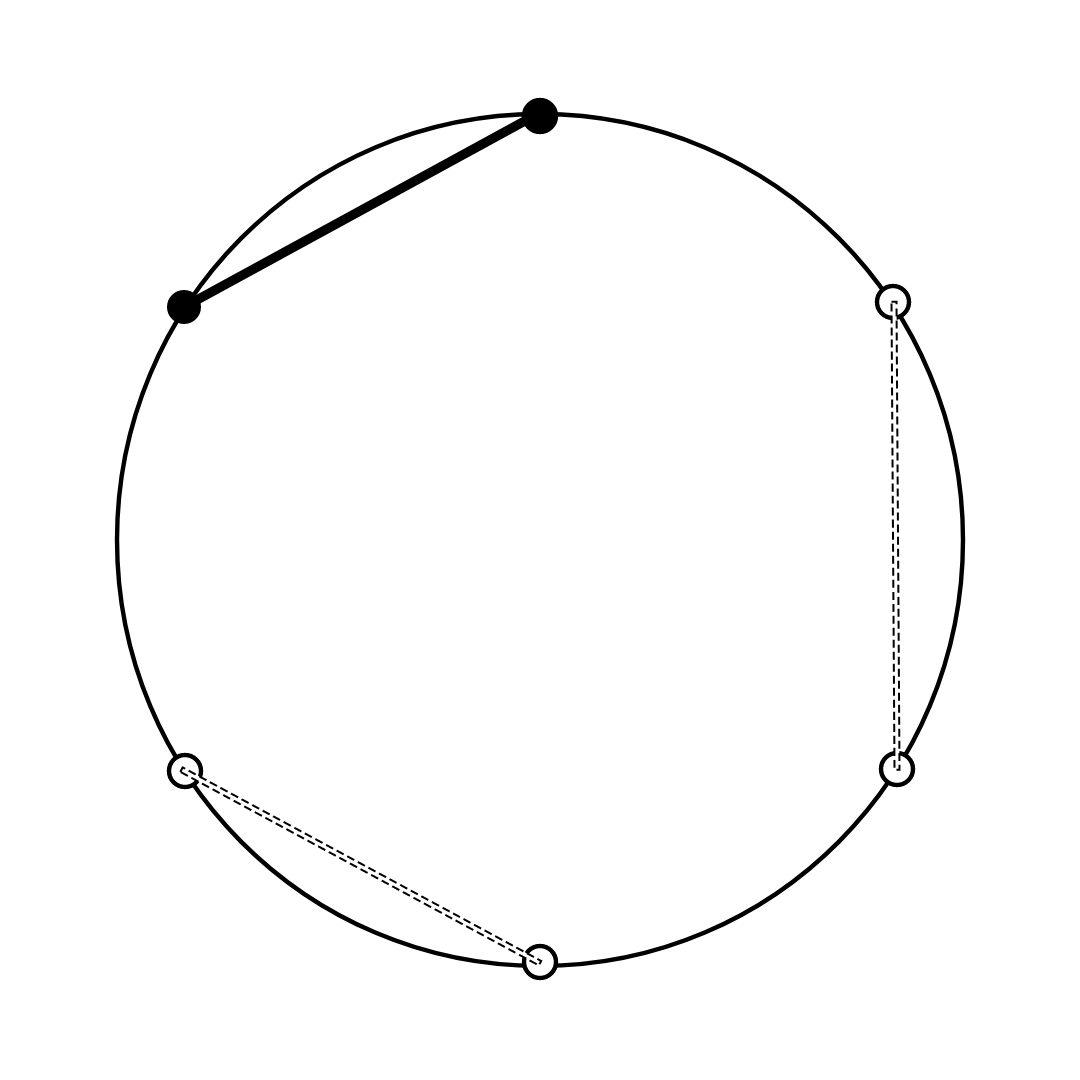}
    \caption{Visualization of all pairings of
    $\mathbb{E} \left[ A^4 B^2 \right]$; notice
    all pairings contribute.}
    \label{fig:all_pairings_A4B2}
\end{figure}

Similarly, the contribution of $\mathbb{E} \left[ A^2 B^2 A^2 B^2 \right]$
in the $8^{\text{th}}$ moment may be visualized by the pairing configurations
in Figure \ref{fig:contributing_pairings_A2B2A2B2}.

\begin{figure}[ht]
    \centering
    \includegraphics[width=0.2\textwidth]{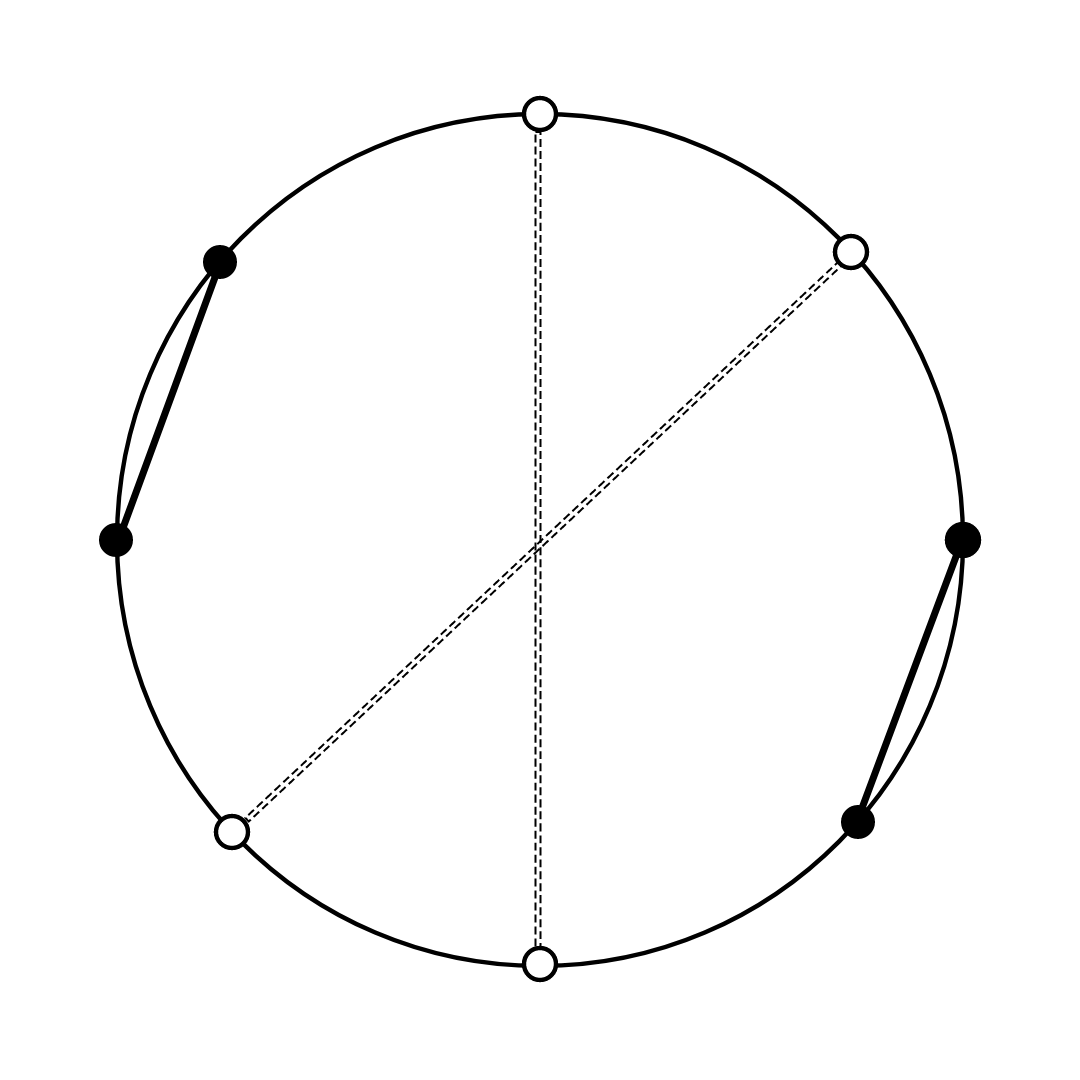} \includegraphics[width=0.2\textwidth]{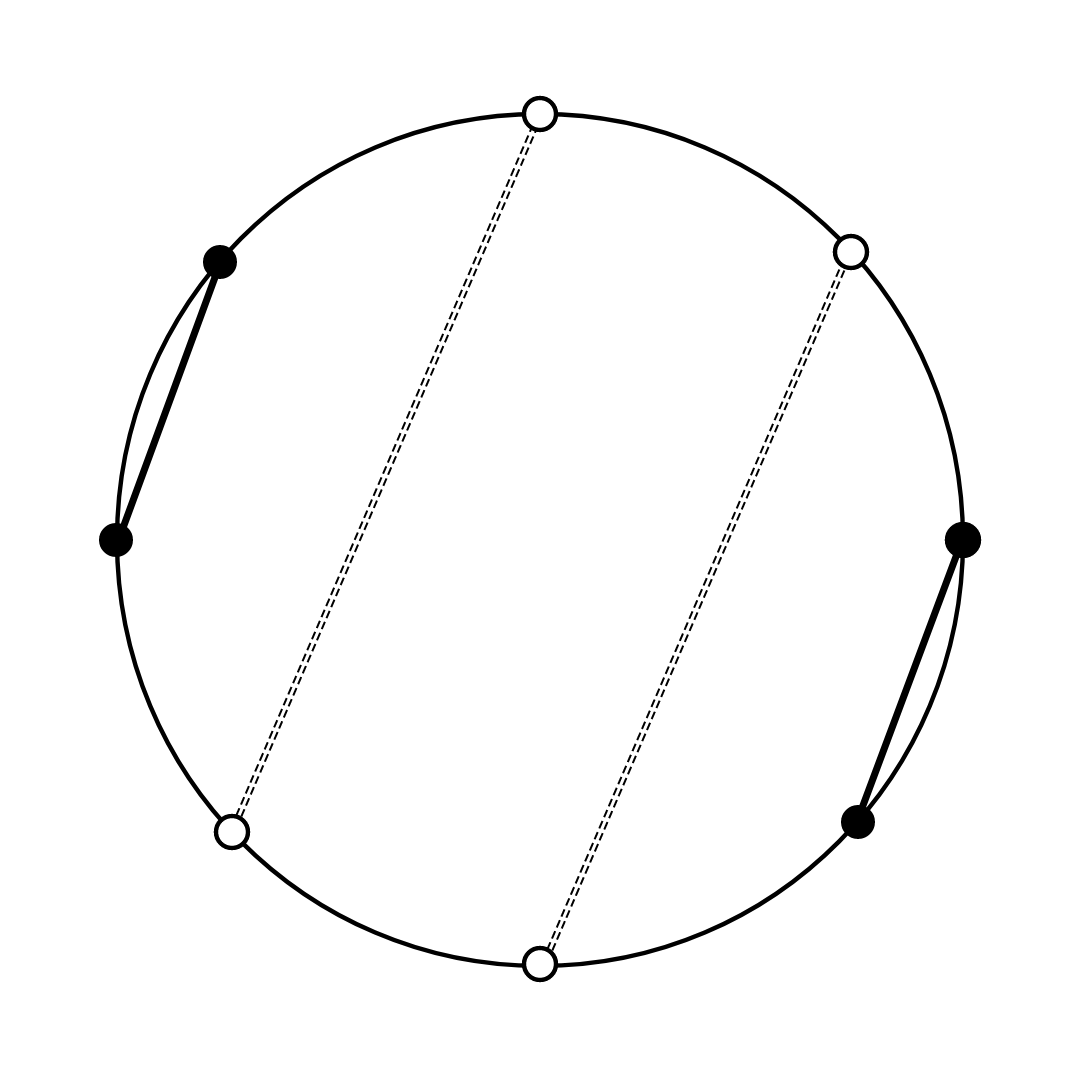}
    \includegraphics[width=0.2\textwidth]{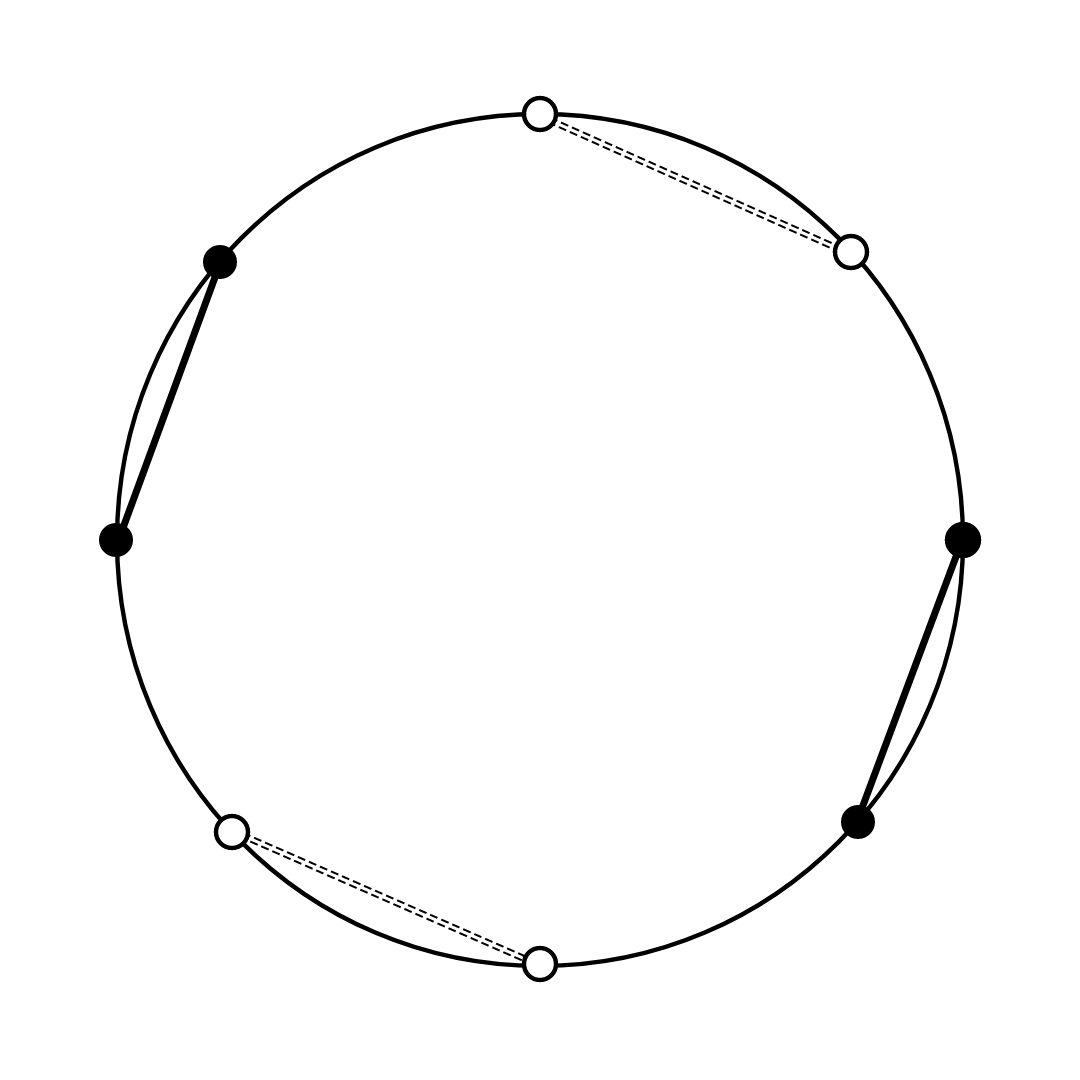}
    \includegraphics[width=0.2\textwidth]{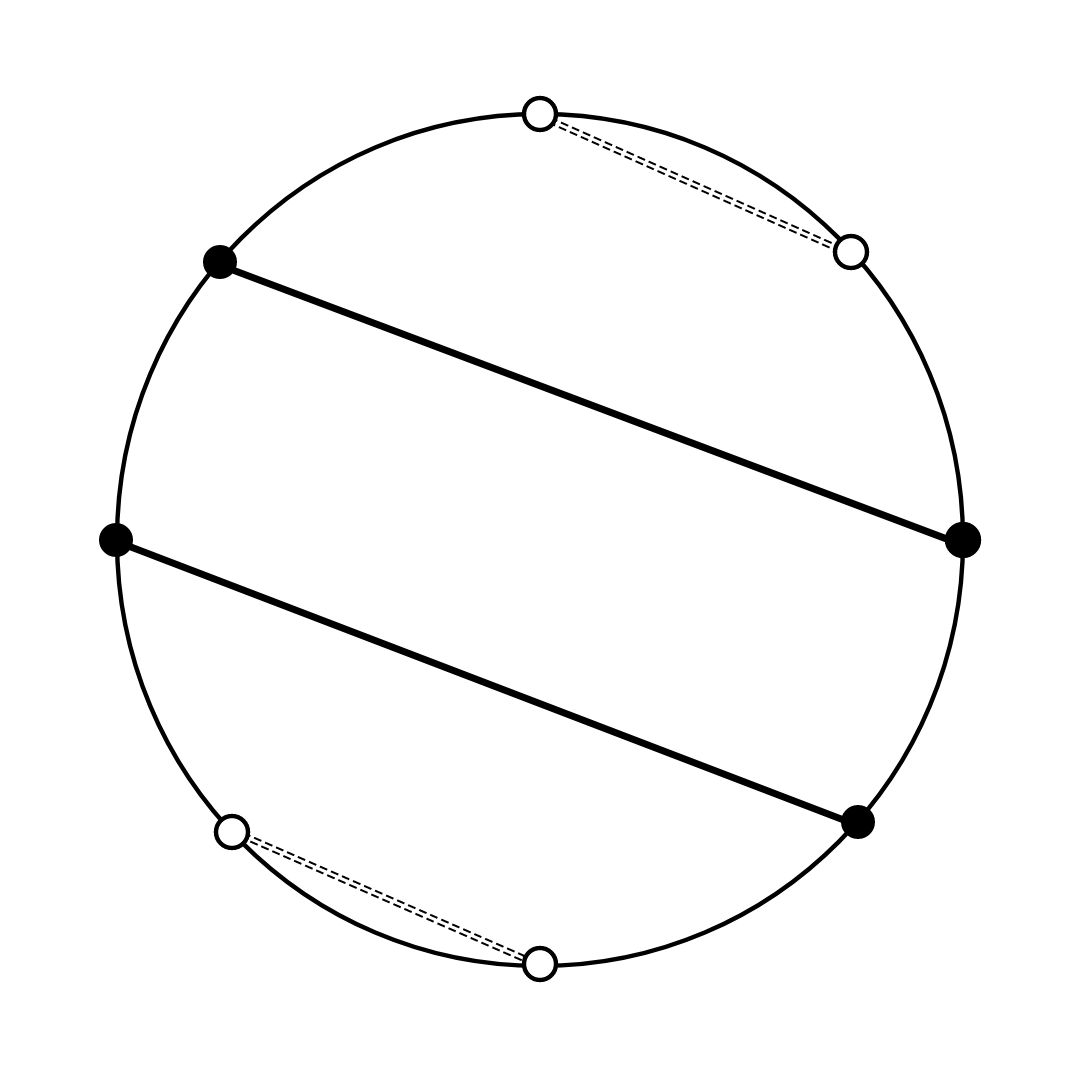}
    \caption{Visualization of contributing pairings of
    $\mathbb{E} \left[ A^2 B^2 A^2 B^2 \right]$.}
    \label{fig:contributing_pairings_A2B2A2B2}
\end{figure}

In contrast, each of the non-contributing pairing configurations
of $\mathbb{E} \left[ A^2 B^2 A^2 B^2 \right]$ may be visualized
by the pairing configurations in Figure \ref{fig:non_contributing_pairings_A2B2A2B2}.
Notice that each non-contributing pairing has either a crossing
of two pairs of $b$'s, or a pair of $a$'s crossing a pair of $b$'s.
\begin{figure}[ht]
    \centering
    \includegraphics[width=0.2\textwidth]{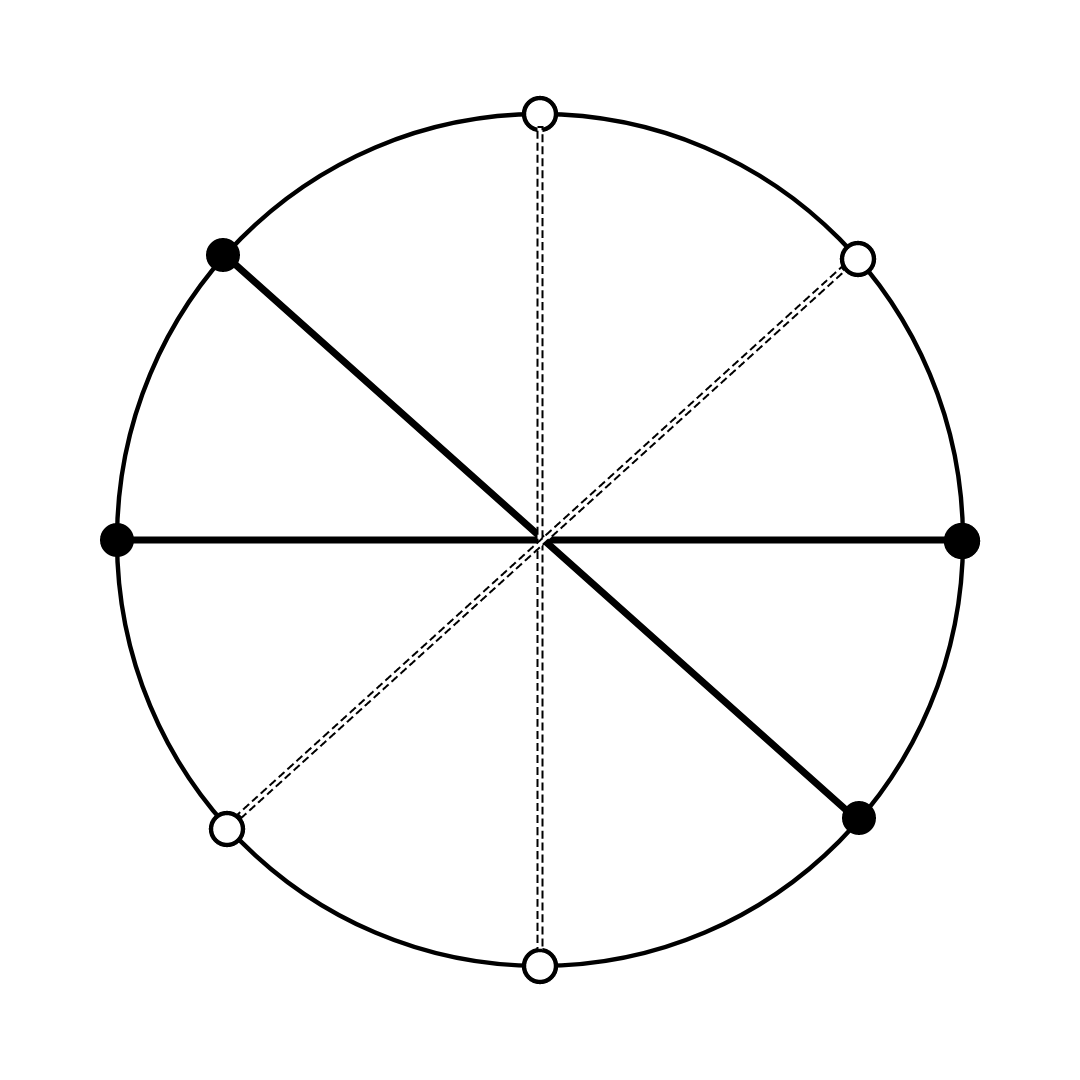} \includegraphics[width=0.2\textwidth]{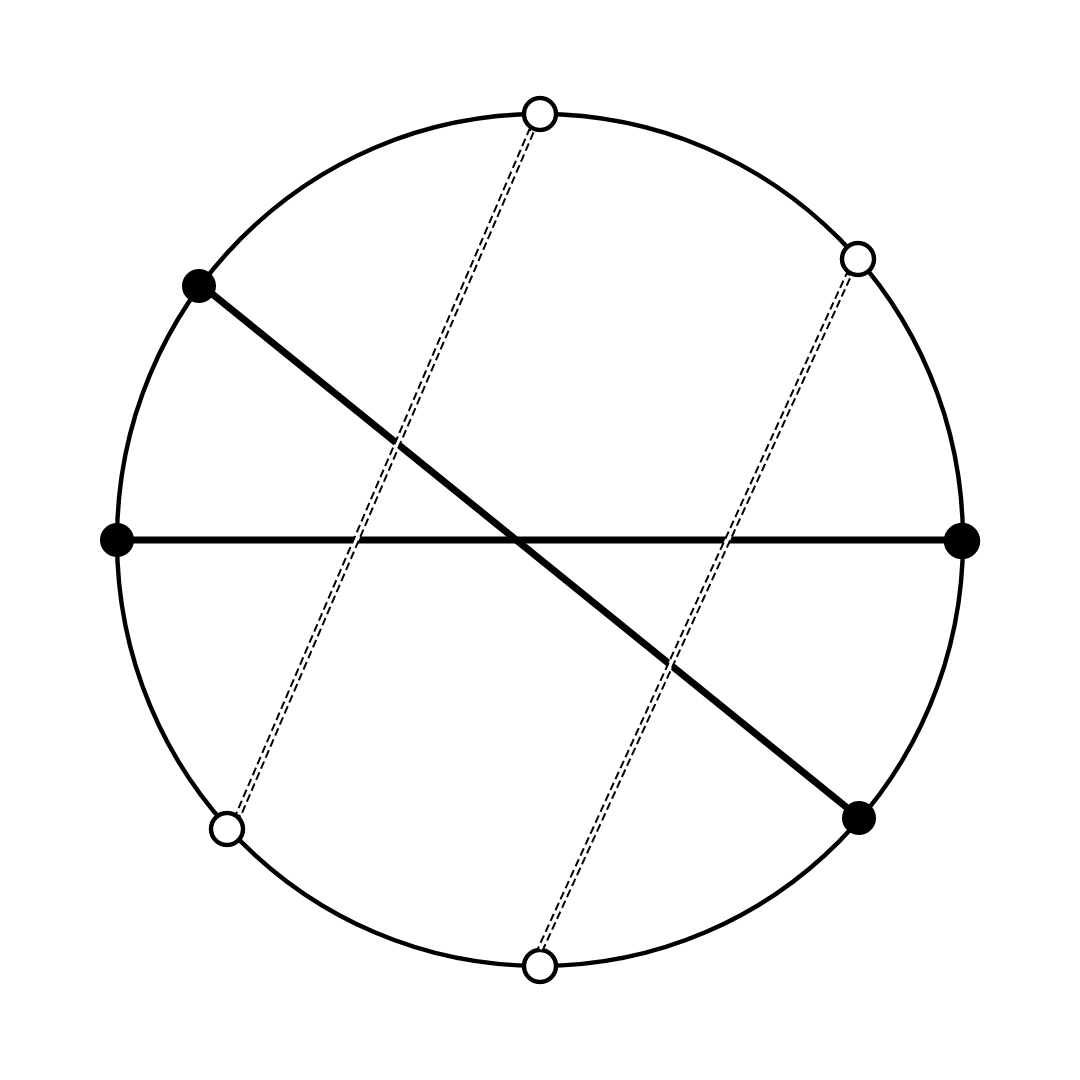}
    \includegraphics[width=0.2\textwidth]{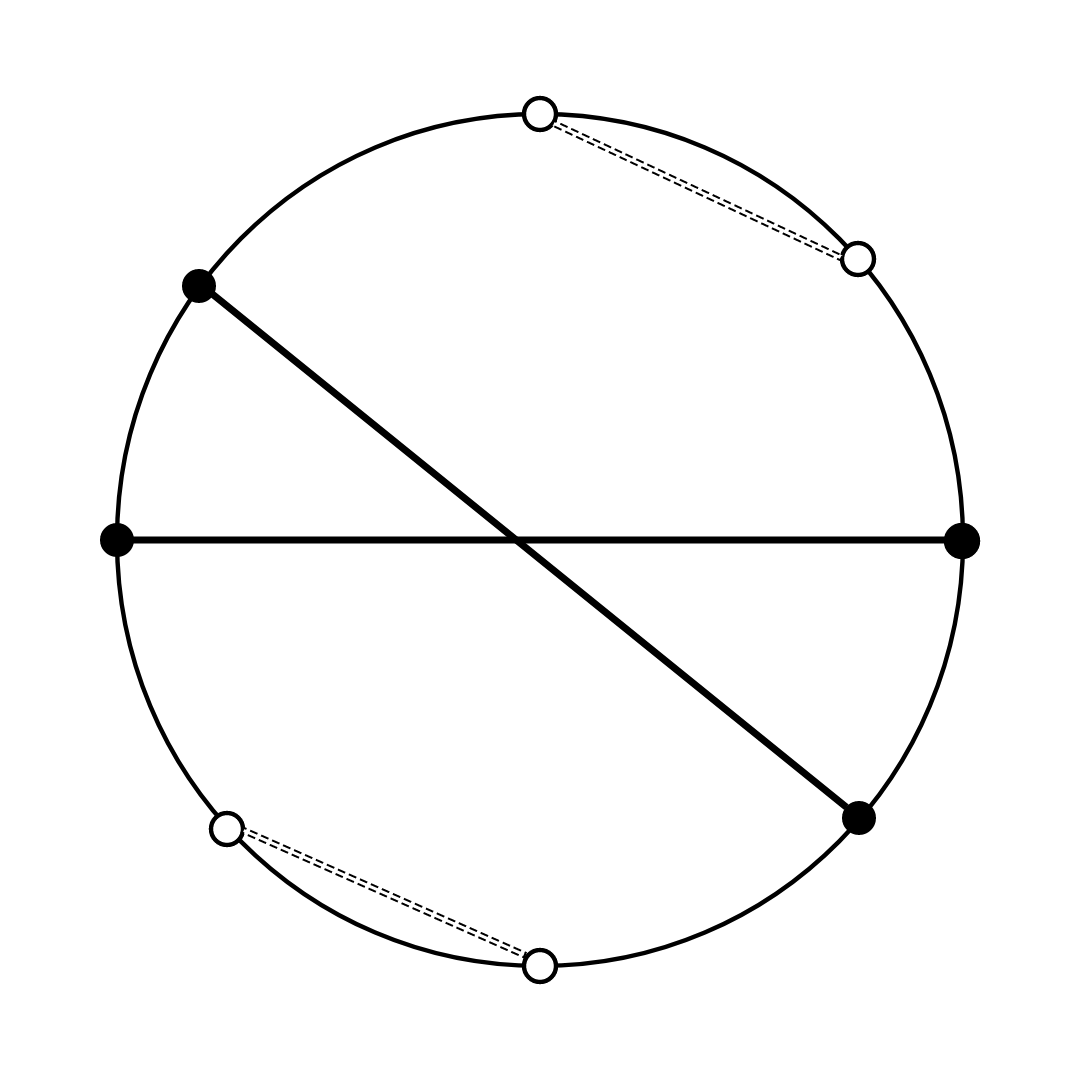}
    \includegraphics[width=0.2\textwidth]{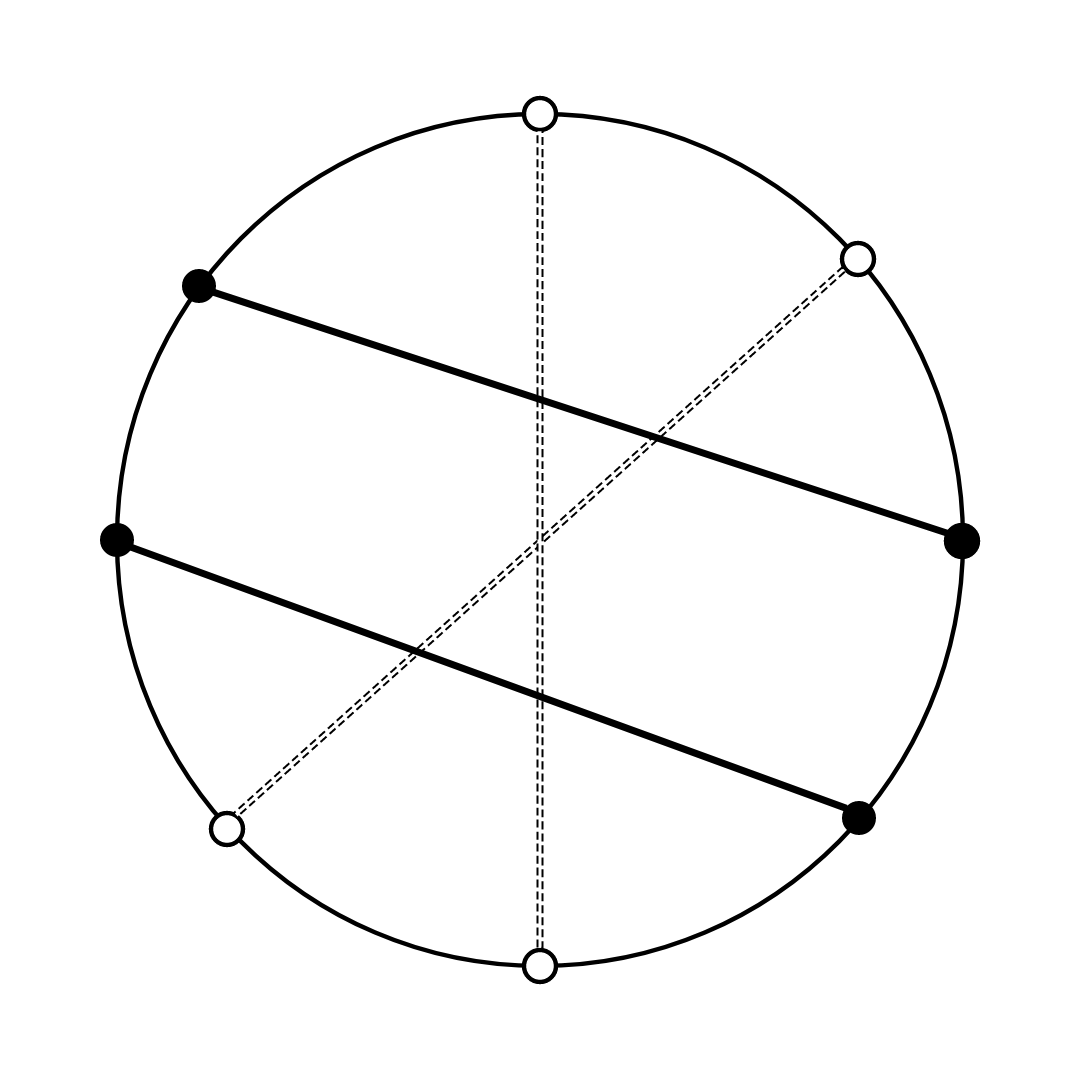}
    \caption{Visualization of non-contributing pairings of
    $\mathbb{E} \left[ A^2 B^2 A^2 B^2 \right]$.}
    \label{fig:non_contributing_pairings_A2B2A2B2}
\end{figure}

Each legal pairing configuration contributes fully in the limit as
$N \to \infty$, while any other pairing configuration contributes
$O \left(1 / N \right)$.

The pairing behavior of elements of $\mathcal{D}_1$ can be understood
as a hybrid of semicircle and Gaussian behaviors. Given $2k$ elements
on the circumference of a unit circle, if all $(2k-1)!!$ possible
pairings contribute fully in the limit as $N \to \infty$, one recovers
precisely the $2k$\textsuperscript{th} moment of the Gaussian; if on
the other hand only the $\binom{2k}{k} / (k+1)$ non-crossing
pairings contribute, one recovers precisely the $2k$\textsuperscript{th}
moment of the semicircle.


\subsection{Upper Bounds of High Moments}
For ease of notation, we make denote by $G_{2k}$ and $S_{2k}$ the
$2k$\textsuperscript{th} moments of the Gaussian and semicircle
distributions, respectively. We show that the $2k$\textsuperscript{th}
moments $M_{2k}\left( \mathcal{D}_1 \right)$ are bounded away from
$G_{2k}$ and $S_{2k}$. We first show that products arising in the
expansion of \eqref{eq1.1} with form $\mathbb{E}\left[ \Tr(A^I B^J)\right]$
have a contribution that is computable in closed form.

\begin{defi}\label{def:power_equivalence}
For fixed $I, J \in 2\mathbb{Z}^+$, let $\prod_{h=1}^{p} A^{i_h} B^{j_h}
\sim \prod_{k=1}^{q} A^{s_k} B^{t_k}$ if and only if $\sum_{h=1}^p i_h =
\sum_{k=1}^q s_k = I$ and $\sum_{h=1}^p j_h = \sum_{k=1}^q t_k = J$.
\end{defi}

For $1<p<\infty$, the preceding relation defines equivalence classes on
the set of all finite products with form $\prod_{h=1}^{p} A^{i_h} B^{j_h}$,
where $\sum_{k=1}^p i_k = I$ and $\sum_{k+1}^p j_k = J$ for $I, J \in
2\mathbb{Z}^+$. Each equivalence class contains a canonical element
$A^I B^J$, and we denote its equivalence class by $\left[ A^I B^J \right]$.

\begin{lem}\label{lemma:std_form_contribution}
Fix $I, J \in 2\mathbb{Z}^+$ and let $k = I + J$. Then the contribution
of the following is:
\begin{equation}
    \lim_{N\to\infty} \frac{1}{N^{\frac{k}{2}+1}} \mathbb{E} \left[
    \Tr(A^I B^J) \right] \ = \ (I - 1)!!\frac{2}{J+2}
    \binom{J}{J/2}.
\end{equation}

\begin{proof}
Observing that
\begin{equation}\label{eq:std_form_trace_exp}
    \mathbb{E}\left[ \Tr(A^I B^J) \right] \ = \
    \sum_{1 \leq i_1, \ldots, j_J \leq 2N} \mathbb{E} \left[
    a_{i_1, i_2} a_{i_2, i_3}
    \cdots a_{i_{I}, j_1} b_{j_1, j_2} b_{j_2, j_3} \cdots
    b_{j_{J}, i_1} \right]
\end{equation}
we see that in each summand the $a$'s are matched in pairs and the $b$'s
are matched in pairs. If any $a$ or $b$ occurs to a third or higher
power, there are fewer than $k / 2 + 1$ degrees of freedom, and
the summand will not contribute under the limit $N \to \infty$. There
are $(I-1)!!$ ways to match the $a$'s in pairs, each resulting in
$I / 2 + 1$ degrees of freedom. Similarly, there are $2 \binom{J}{J/2}
/ (J+2)$ ways to pair the $b$'s that ensure $J / 2$ degrees
of freedom; this may be interpreted as pairing $J$ of the $b$'s placed on the
circumference of a circle with non-intersecting chords. There are no
arrangements wherein a pair of $a$'s crosses over a pair of $b$'s. As
there are $n$ choices for each degree of freedom we have
\begin{align}
    \lim_{N \to \infty} \frac{1}{N^{\frac{k}{2}+1}} \mathbb{E}\left[
    \Tr(A^I B^J) \right] \ &=\ \lim_{N \to \infty} \frac{1}{N^{
    \frac{k}{2}+1}} (I-1)!!\frac{2}{J+2} \binom{J}{J/2} N^{\frac{I}{2}
    +\frac{J}{2}+1} \nonumber \\
    \ &=\ (I-1)!!\frac{2}{J+2} \binom{J}{J/2}.
\end{align}
\end{proof}
\end{lem}

\begin{lem}\label{lemma:std_form_contribution_bound}
Fix $I, J \in 2\mathbb{Z}^+$ and let $k = I + J$.

\begin{equation}
    \lim_{N\to\infty} \frac{1}{N^{\frac{k}{2}+1}}\mathbb{E}\left[
    \Tr(A^I B^J) \right] \geq \max_{\left[ A^I B^J \right] \setminus
    A^I B^J} \left\lbrace \lim_{n\to\infty} \frac{1}{N^{\frac{k}{2}+1}}
    \mathbb{E}\left[ \Tr\left( \prod_{k=1}^p A^{i_k} B^{j_k} \right)
    \right] \right\rbrace.
\end{equation}

\begin{proof}
By Lemma \ref{lemma:std_form_contribution} we have that
\begin{equation}
    \lim_{N \to \infty} \frac{1}{N^{\frac{k}{2}+1}} \mathbb{E}\left[
    \Tr(A^I B^J) \right] \ = \ (I-1)!!\frac{2}{J+2} \binom{J}{J/2}.
\end{equation}
For elements of $[A^I B^J]$ the maximum number of contributing pairings
occurs when no parings of $a$'s crosses over a pairing of $b$'s; this
occurs in the summands of Equation \ref{eq:std_form_trace_exp} as all
$a$'s are adjacent and all $b$'s are adjacent. However, all elements of
$\left[ A^I B^J \right] \setminus A^I B^J$ have form $\prod_{k=1}^p
A^{i_k} B^{j_k} = A^{i_1} B^{j_1} \cdots A^{i_p} B^{j_p}$ for $p \geq 2$,
in which case there exist at least $O(p)$ pairings of $a$'s
and $b$'s that result in mutual crossovers. Each of these crossovers
results in a loss of a degree of freedom, yielding
\begin{equation}
     \lim_{N\to\infty} \frac{1}{N^{\frac{k}{2}+1}}\mathbb{E}
    \left[ \Tr\left( \prod_{k=1}^{p} A^{i_k} B^{j_k} \right) \right]
     \leq \left( (I-1)!! - O(p) \right)
    \frac{2}{J+2} \binom{J}{J/2}
\end{equation}
from which the claim follows.
\end{proof}
\end{lem}

\subsubsection{Weak Upper Bound of Even Moments}

\begin{thm}\label{thm:weak_upper_bound_even_moments}
$For \; k\geq1,\;\lim_{N\rightarrow\infty} M_{2k}(\mathcal{D}_1)
\leq(2k-1)!!= G_{2k}$.
\begin{proof}
We first obtain an upper bound for the $2k$\textsuperscript{th}
moments by assuming our matrix $A$ and $B$ commute. It is clear
that this is an upper bound from Lemma
\ref{lemma:std_form_contribution_bound}. Our equation
from \eqref{eq1.2} then becomes
\begin{align}
     M_{2k}(\mathcal{D}_1,N) \ &=\ \frac{1}{(2N)^{k+1}} \mathbb{E}
     \left[ \Tr[(A+B)^{2k}+(A-B)^{2k}] \right] \nonumber \\
    \ &\leq\ \frac{1}{(2N)^{k+1}}\ 2\sum\limits_{i=0}^{k}\binom{2k}{2i}
    \mathbb{E}\left[\Tr(A^{2i}B^{2k-2i})\right].
\end{align}
Let $N \to\infty$ and by Lemma \ref{lemma:std_form_contribution}
we have:
\begin{align}
    M_{2k}(\mathcal{D}_1) \ &\leq\ \frac{1}{2^k}\ \left[(2k-1)!! +
    \sum\limits_{i=1}^{k-1} \binom{2k}{2i} \binom{2i}{i}
    \frac{(2k-2i-1)!!}{i+1} + \frac{1}{k+1}\binom{2k}{k}\right] \nonumber\\
    \ &=\ \frac{1}{2^k}\ \left[ (2k-1)!! + \sum\limits_{i=1}^{k-1}
    \frac{(2k)!(2k-2i-1)!}{(2k-2i)!(k-i-1)!(i+1)i!i!2^{k-1-i}} +
    \frac{(2k)!}{(k+1)k!k!}\right] \nonumber \\
    \ &=\ \frac{(2k-1)!!}{2^k} \left[ 1 + \sum\limits_{i=1}^{k-1} \frac{k!2^i}{(k-i)!(i+1)!i!}+\frac{2^k}{(k+1)!}\right] \nonumber \\
    \ &=\ (2k-1)!! \left[ \frac{1}{2^k} + \sum \limits_{i=1}^{k-1}
    \frac{2^i}{(i+1)!} \binom{k}{i} \left( \frac{1}{2} \right)^{k-i} \left(
    \frac{1}{2} \right)^i + \frac{1}{(k+1)!} \right].
\end{align}
Note that for non-negative integer $i$, it is clear that $2^i
\leq (i+1)!$, i.e., $2^i / (i+1)! \leq 1$. Hence, we may
replace $2^i / (i+1)!$ in the above sum by $1$ and obtain
the following:
\begin{align}
    M_{2k}(\mathcal{D}_1) \ &\leq\ (2k-1)!!\left[\frac{1}{2^k} +
    \left( \sum\limits_{i=0}^{k} \binom{k}{i}( \frac{1}{2})^{k-i}
    (\frac{1}{2})^i\right) - \frac{2}{2^k} + \frac{1}{(k+1)!}
    \right] \nonumber \\
    \ &=\ (2k-1)!!\left[1 + \frac{1}{(k+1)!} - \frac{1}{2^k}\right].
\end{align}
As noted earlier $2^k \leq (k+1)!$ so that $1 / (k+1)! -
1 / 2^k \leq 0$, from which it follows clearly that
\begin{align}
    M_{2k} &\leq (2k-1)!!\left[1 + \frac{1}{(k+1)!}-\frac{1}{2^k}\right]
    \ \leq\ (2k-1)!!
\end{align}
as desired.
\end{proof}
\end{thm}

\subsubsection{Strong Upper Bound of High Moments}
\begin{thm}
The ratio $M_{2k}(\mathcal{D}_1) / G_{2k} $ tends to $0$ as $k \to \infty$.
\begin{proof}
It suffices to consider $k\geq 2$. From \eqref{eq1.2} and Lemma
\ref{lemma:std_form_contribution_bound}, we have that
\begin{align}
     M_{2k}(\mathcal{D}_1) \ &=\ \frac{1}{(2N)^{k+1}}\mathbb{E}
     \left[\Tr[(A+B)^{2k}+(A-B)^{2k}]\right] \nonumber \\
    \ &\leq\ \frac{1}{(2N)^{k+1}}\ 2 \sum\limits_{i=0}^{k}
    \binom{2k}{2i} \mathbb{E}\left[\Tr(A^{2i}B^{2k-2i}) \right].
\end{align}
By Lemma \ref{lemma:std_form_contribution}, we can explicitly
calculate the contributions of $A^{2i}B^{2k-2i}$. The ratio of
the $2k$\textsuperscript{th} moment of $\mathcal{D}_1$ over the $2k$\textsuperscript{th} moment
of Gaussian is then bounded by
\begin{align*}
    \dfrac{M_{2k}(\mathcal{D}_1)}{G_{2k}} &\ \leq\ \dfrac{\dfrac{2}{(2N)^{k+1}}\sum\limits_{i=0}^k\binom{2k}{2i}(2k-2i-1)!!
    \dfrac{1}{i+1}\binom{2i}{i}N^{k+1}}{(2k-1)!!}\\
    &\ = \  \frac{(2k)!}{2^k(2k-1)!!}\sum\limits_{i=0}^k\left[ \frac{(2k)!}{(2i)!(2k-2i)!}(2k-2i-1)!!\frac{1}{i+1}\frac{(2i)!}{i!i!}\right]\\
    &\ = \  \frac{(2k)!!}{2^k}\sum\limits_{i=0}^k \frac{(2k-2i-1)!!}{(2k-2i)!(i+1)!i!}\\
    &\ = \  \frac{1}{2^k}\sum\limits_{i=0}^k 2^i\frac{k(k-1)\cdots(k-i+1)}{(i+1)!i!}\\
    &\ = \  \frac{1}{2^k}\sum\limits_{i=0}^k \frac{2^i}{(i+1)!}\binom{k}{i}. \numberthis \label{upperboundsimplified}
\end{align*}
For $i\geq 2$, we have
\begin{equation}
    2\cdot 3^{i-1}\ \leq\ (i+1)!
\end{equation}
which implies that
\begin{equation}
    \frac{2^i}{(i+1)!}\ \leq\ \left(\frac{2}{3}\right)^{i-1}.
\end{equation}
Then \eqref{upperboundsimplified} becomes
\begin{align*}
    \frac{1}{2^k}\sum\limits_{i=0}^k \frac{2^i}{(i+1)!}\binom{k}{i} &\ \leq\
    \frac{1}{2^k}\sum\limits_{i=0}^k \left[1+k+\sum\limits_{i=2}^k\left(\frac{2}{3}
    \right)^{i-1}\binom{k}{i}
    \right]\\
    &\ = \  \frac{1}{2^k}\left[\sum\limits_{i=0}^k\left(\frac{2}{3}\right)^i
    \binom{k}{i}-\frac{1}{2}\right]\\
    &\ = \  \frac{1}{2^k}\left[\left(\frac{5}{3}\right)^k-\frac{1}{2}\right]\\
    &\ = \  \left(\frac{5}{6}\right)^k-\frac{1}{2^{k+1}}. \numberthis
\end{align*}
As $k\rightarrow\infty$, the above expression goes to zero, which concludes the proof.
\end{proof}
\end{thm}

\subsection{Lower Bound of High Moments}

We know the moments of the limiting spectral measure are bounded
by $(2k-1)!!$, the moments of the Gaussian. By obtaining
a sufficiently large lower bound for the even moments, we show the
limiting spectral measure has unbounded support. If it had bounded
support, say $[-B,B]$, then the $2k^{\text{th}}$ moment
$M_{2k}(\mathcal{D}_1)$ is at most $B^{2k}$, and $\lim_{k\to\infty}\sqrt[2k]{M_{2k}(\mathcal{D}_1)}<\infty$. We show
this is not the case.
\begin{thm}\label{thm:unbounded_support}
The limiting spectral measure of $\mathcal{D}_1(A,B)$ has unbounded support.
\begin{proof}
It suffices to show for all $k \in \mathbb{Z}^+$ that
\begin{equation}
\sqrt[2k]{M_{2k}(\mathcal{D}_1)} \ > \ \frac{1}{2}\sqrt{k}.
\end{equation}
By \eqref{eq1.1},
\begin{equation}
    \sqrt[2k]{M_{2k}(\mathcal{D}_1)}\ >\ \left(\frac{(2k-1)!!}{2^k}
    \right)^\frac{1}{2k} \ = \ \left(\frac{(2k)!}{4^kk!}\right)^{\frac{1}{2k}}.
    \label{unbdd}
\end{equation}
By Stirling's Formula, the right hand side of \eqref{unbdd} is bounded below by
\begin{align*}
     \left(\dfrac{\sqrt{2\pi}(2k)^{2k+1/2}e^{-2k}}{2^{2k}k^{k+1/2}e^{-k+1}}
     \right)^\frac{1}{2k} &\ = \  \left(2\sqrt{\pi}e^{-k-1}k^k\right)^\frac{1}{2k}\\
     &\ \geq\ e^{-\frac{1}{2}-\frac{1}{2k}}\sqrt{k}\\
     &\ \geq\ e^{-\frac{1}{4}}\sqrt{k}\\
     &\ >\ \frac{1}{2}\sqrt{k}. \numberthis
\end{align*}
\end{proof}
\end{thm}

\begin{thm}\label{thm:weak_lower_bound_even_moments}
$For \; k\geq1,\;\lim_{N\rightarrow\infty}M_{2k}(\mathcal{D}_1)\geq
S_{2k}$.

\begin{proof}
We first obtain a lower bound for the $2k$\textsuperscript{th} moments by dropping products
of the form $A^IB^J$ from \eqref{eq1.1}. It is clear that what we obtain is
a lower bound, and we want to show the following:
\begin{align}
    M_{2k} &\ \geq\ (2k-1)!! + \frac{1}{k+1} \binom{2k}{k} + 2k\sum_{j=1}^{k-1}
    (2k-2j-1)!! \frac{1}{j+1} \binom{2j}{j} \nonumber \\
    &\ \geq\ \frac{2^k}{k+1} \binom{2k}{k} \label{oneofthelowerbounds}
\end{align}
for all $k \in \N$. Dividing the left-hand side of \eqref{oneofthelowerbounds}
by the right-hand side and simplifying, we have
\begin{equation}\label{oneofthelowerboundsindexshift}
    \begin{split}
        &\frac{k+1}{2^k} \binom{2k}{k}^{-1} \left((2k-1)!! + \frac{1}{k+1}
    \binom{2k}{k} + 2k\sum_{j=1}^{k-1} (2k-2j-1)!! \frac{1}{j+1} \binom{2j}{j}
    \right) \\
    &\ = \  \frac{(k+1)}{2^k} \frac{k! k!}{(2k)!} \frac{(2k-1)!}{2^{k-1} (k-1)!}
    + \frac{1}{2^k}\\
    &\hspace{2cm} + 2k \frac{k+1}{2^k} \frac{k! k!}{(2k)!} \sum_{j=1}^{k-1}
    (2k-2j-1)!! \frac{1}{j+1} \binom{2j}{j}.
    \end{split}
\end{equation}

Simplifying and applying the index shift $j \to k-j-1$
to the sum, \eqref{oneofthelowerboundsindexshift} becomes
\begin{equation}
    \frac{(k+1)!}{ 4^{k} } + \frac{1}{2^k} + \frac{(k+1)! k!}{2^k (2k-1)!}
    \sum_{j=0}^{k-2} (2j+1)!! \frac{1}{k-j} \binom{2(k-j-1)}{k-j-1}.
    \label{simplified}
\end{equation}

We wish to bound \eqref{simplified} below by $1$. Observe
$(k+1)! / 4^k \geq 1$ for all $k \geq 6$ and each of the three terms in
\eqref{simplified} is non-negative, yielding the desired result for $k \geq 6$.
The remaining cases $1 \leq k \leq 5$ are easily verified numerically.



\end{proof}
\end{thm}


\subsection{Weak Convergence}\label{subsec:weak_convergence}

\begin{defi}[Weak Convergence]\label{def:weak_convergence}
A family of probability distributions $\mu_n$ weakly converges to $\mu$ if
and only if for any bounded, continuous function $f$ we have
\begin{equation}
    \lim_{n \to \infty} \int_{-\infty}^{\infty} f(x) \mu_n(dx) \ = \
    \int_{-\infty}^{\infty} f(x) \mu(dx).
\end{equation}
\end{defi}

By Theorem \ref{thm:weak_upper_bound_even_moments}, we know the moments
$M_k(\mathcal{D}_1)$ exist and are finite. To prove we have weak convergence
to the limiting spectral measure we need to show that the variances tend to
$0$. We must show
\begin{equation}
    \lim_{N \to \infty} \left( \mathbb{E}\left[ M_m(\mathcal{D}_1,N)^2 \right]
    - \mathbb{E}\left[ M_m(\mathcal{D}_1,N) \right]^2 \right) \ = \ 0.
\end{equation}
We observe that
\begin{align}\label{eq:square_of_EV}
    \begin{split}
        \mathbb{E}\left[ M_m(\mathcal{D}_1,N) \right]^2 &\ = \  \frac{1}{(2N)^{2k+2}}
    \left( \sum_{s_1=1}^N \cdots \sum_{s_k=1}^N  \mathbb{E}\left[ a_{s_1 s_2}
    \cdots b_{s_k s_1} \right] \right) \\
    &\hspace{2cm} \times \left( \sum_{t_1=1}^N \cdots \sum_{t_k=1}^N \mathbb{E}
    \left[ a_{t_1 t_2}  \cdots b_{t_k t_1} \right] \right)
    \end{split}
\end{align}
and
\begin{align}
    \mathbb{E} \left[ M_m(\mathcal{D}_1,N)^2 \right] &\ = \  \frac{1}{(2N)^{2k+2}}
    \sum_{s_1=1}^N \cdots \sum_{t_k=1}^N
    \mathbb{E} \left[ a_{s_1 s_2}  \cdots b_{s_k s_1} a_{t_1 t_2}  \cdots
    b_{t_k t_1} \right]. \label{eq:EV_of_square}
\end{align}

There are two possibilities: either the entries with subscripts $s$ are
completely disjoint from those with subscripts $t$, or there are ``crossover''
cases wherein entries with subscripts $s$ match to those with subscripts $t$.
In the former case, the entries with subscripts $s$ and the entries with
subscripts $t$ contribute equally to $\mathbb{E}\left[ M_m(\mathcal{D}_1,N)
\right]^2$ and $\mathbb{E} \left[ M_m(\mathcal{D}_1,N)^2 \right]$. However,
the latter case requires estimating the contribution incurred by crossovers;
i.e., $s_{\alpha +1} - s_\alpha = \pm (t_{\beta +1} - t_\beta) + C$, $C \in
\lbrace 0, N -1, 1 - N \rbrace$ if the crossover occurs in the $a$'s, or
$\abs{s_{\alpha +1} - s_\alpha} = \abs{t_{\beta+ 1} - t_\beta}$ if the
crossover occurs in the $b$'s. We assume $m=2k$; the proof is analogous
in the case of $m$ odd. The following two lemmas imply that the variance
tends to $0$.

\begin{lem}\label{lem:contribution_square_of_EV}
The contribution from crossovers in $\mathbb{E}\left[ M_m(\mathcal{D}_1,N)
\right]^2$ is $O_k \left( 1 / N \right)$.

\begin{proof}
We observe from Equation \eqref{eq:square_of_EV} that if anything is unpaired
among the entries with subscripts $s$ or $t$, then the expected value vanishes.
We may then assume that in $\mathbb{E}\left[ M_m(\mathcal{D}_1,N) \right]^2$
all entries are at least paired, and that there is at least one crossover arising
from a common value either between elements $a_{s_\alpha, s_{\alpha +1}}$,
$a_{t_\beta , t_{\beta+1}}$ or between elements $b_{s_\alpha, s_{\alpha+1}}$,
$b_{t_\beta , t_{\beta+1}}$. The maximum number of such possibilities occurs
when all elements with subscripts $s$ are paired among themselves, as are all
elements with subscripts $t$, and only one crossover occurs between a pair
index by $s$'s and a pair indexed by $t$'s.

If a pair of $a$'s cross over, there are two choices of sign and three choices
of the constant $C$, incurring a loss of 1 degree of freedom. If instead a pair
of $b$'s cross over, then the indices are determined up to permutation as $B$
is symmetric. This incurs a loss of 1 degree of freedom if the pair of
$b_{t_{\beta}, t_{\beta+1}}$'s are adjacent, and a loss of 2 degrees of freedom
otherwise. In both cases, there is a loss of at least 1 degree of freedom.
It follows that there are $k+1$ degrees of freedom from the $s$-indexed entries
and at most $k$ degrees of freedom from the $t$-indexed entries, where the loss
of degrees of freedom from the $t$-indexed entries occurs from the crossover.
As triple or higher pairings and two or more crossovers only further erode the
total degrees of freedom, we see these terms give $O_k \left( N^{2k+1}
\right)$, and so contribute $(1 / (2N)^{2k+2})O_k \left( N^{2k+1}
\right) = O_k \left( 1 / N \right)$ to $\mathbb{E}\left[
M_m(\mathcal{D}_1,N) \right]^2$.
\end{proof}
\end{lem}

\begin{lem}\label{lem:contribution_EV_of_square}
The contribution from crossovers in $\mathbb{E} \left[ M_m(\mathcal{D}_1,N)^2
\right]$ is $O_k \left( 1 / N \right)$.

\begin{proof}
We consider two cases: either all $s$-indexed entries and $t$-indexed entries
are paired among themselves and at least one crossover occurs, or there are
unpaired singletons among the $s$-indexed entries and/or the $t$-indexed entries.
In the former case, we may show in a manner analogous to the proof of Lemma \ref{lem:contribution_square_of_EV} that such terms contribute $O_k
\left( 1 / N \right)$ to $\mathbb{E} \left[ M_m(\mathcal{D}_1,N)^2 \right]$.

In the latter case, we assume there are unmatched singletons among the $s$-indexed
entries and/or the $t$-indexed entries. Let there be $w_s > 0$ singletons and
$k - w_s / 2$ pairs in the $s$-indexed entries; similarly, let there be
$w_t > 0$ singletons and $k - w_t / 2$ pairs among the $t$-indexed entries.
Observe that $w_s. w_t$ must be even, as the number of $a$'s and $b$'s in both
the $s$-indexed entries and the $t$-indexed entries must be even, and we have
assumed that all other entries are paired. Let $\mathcal{X}$ denote the number
of crossings; since all $w_s, w_t$ singletons must be paired off (otherwise,
the expected value vanishes), we see that $\mathcal{X} \geq \max\lbrace w_s,
w_t \rbrace$.

Among the $s$-indexed entries, we see that there are $\left( k - w_s / 2
+ 1 \right)$ degrees of freedom from choosing the indices of $k -
w_s / 2$ of the $s$-indexed pairs. The singletons contribute $w_s - 1$ degrees
of freedom, the $-1$ arising from the final singleton's indices being determined
once all other indices are selected. Thus, there are $\left( k - w_s / 2
+ 1 \right) + (w_s - 1)$ degrees of freedom from the $s$-indexed entries. If
$w_t = 0$ then $\mathcal{X} = w_s$ and the degrees of freedom from the $t$-indexed
entries is at most $(k+1) - \mathcal{X}$, since each crossover incurs a loss of
at least one degree of freedom. Then the total degrees of freedom in the case
$w_s \geq 2, w_t = 0$ is
\begin{align}
    \left( k - \frac{w_s}{2} + 1 \right) + (w_s - 1) + (k+1) - \mathcal{X} \
    &= \ 2k + \frac{w_s}{2} + 1 - \mathcal{X} \nonumber \\
    &\leq\ 2k - \frac{w_s}{2} + 1 \nonumber \\
    &\leq\ 2k
\end{align}
from which it follows that such terms contribute $(1 / (2N)^{2k+2})
O_k\left( N^{2k} \right) = O_k \left( 1 / N^2
\right)$ to $\mathbb{E} \left[ M_m(\mathcal{D}_1,N)^2 \right]$.

Now assume that $w_s > 0, w_t > 0$; as before, there are $\left( k - w_s / 2
+ 1 \right) + (w_s - 1)$ from the $s$-indexed entries. From $\mathcal{X}$ crossovers,
we lose at least $\mathcal{X}-1$ degrees of freedom from the $t$-indexed entries,
where $1$ is subtracted from $\mathcal{X}$ in the case that the final singleton of
forced value in the $t$-indexed entries is matched to an existing pair among the
$s$-indexed entries. Then the degrees of freedom from the $t$-indexed entries is
$\left( k - w_t / 2 + 1 \right) + (w_t - 1) - (\mathcal{X}-1)$. The total
degrees of freedom in the case $w_s \geq 2, w_t \geq 2$ is then
\begin{align}
    \begin{split}
    &\left( k - \frac{w_s}{2} + 1 \right) + (w_s - 1) + \left( k - \frac{w_t}{2}
    + 1 \right) + (w_t - 1) - (\mathcal{X}-1)\\
    &\hspace{5cm}\ =\ 2k + 1 + \frac{1}{2}\left( 2\mathcal{X} - (w_s + w_t)  \right)
    \end{split} \nonumber \\
    &\hspace{5cm}\ \leq \ 2k + 1
\end{align}
where the final line follows from $\mathcal{X} \geq \lbrace w_s, w_t \rbrace$.
It follows that such terms contribute $1 / (2N)^{2k+2} O_k\left(
N^{2k+1} \right) = O_k \left( 1 / N \right)$ to $\mathbb{E} \left[
M_m(\mathcal{D}_1,N)^2 \right]$.

Additional crossovers only further erode the available degrees of freedom, and we
note that cases of triple or higher matchings in either the $s$-indexed or
$t$-indexed entries can be reduced to pairs and singletons, thus falling under
the purview of previously considered cases. Thus, there are at most $2k + 1$
degrees of freedom, and all crossover terms contribute $O_k \left(
1 / N \right)$ to $\mathbb{E} \left[ M_m(\mathcal{D}_1,N)^2 \right]$.
\end{proof}
\end{lem}

\begin{thm}\label{thm:1-disco_weak_con}
Let $p$ have mean $0$, variance $1$ and finite higher moments. The measures of
$\mu_{\mathcal{D}_1,2N}(x)$ weakly converge to a universal measure of unbounded
support independently of $p$.
\end{thm}

\begin{proof}
By Theorem \ref{thm:weak_upper_bound_even_moments} the moments $M_k$ exist and
are finite. Since $\mathbb{E}\left[ M_k(\mathcal{D}_1, N) \right] \to M_k$, and
since by Lemmas \ref{lem:contribution_square_of_EV},
\ref{lem:contribution_EV_of_square} the variances tend to zero, Chebyshev's inequality
and the Moment Convergence Theorem (Theorem \ref{thm:momct}) give weak convergence.
As Theorem \ref{thm:weak_upper_bound_even_moments}
gives that $M_k$ is bounded above by the the moments of the Gaussian, it follows
that the moments of the disco $\mathcal{D}_1$ uniquely determine a probability
measure, which by Theorem \ref{thm:unbounded_support} has unbounded support.
\end{proof}


\subsection{Almost Sure Convergence}\label{subsec:almost_sure_convergence}

Almost sure convergence follows from showing that for each
non-negative integer $m$ that
\begin{equation}
X_{m;N}(A) \ \to \ X_m(A) \ = \ M_m\left(\mathcal{D}_1\right) \ \
{\rm almost\ surely,}
\end{equation}
and then applying the Moment Convergence Theorem (Theorem \ref{thm:momct}).
The key step in proving this is showing that
\begin{equation}\label{eq:thmprestrongB}
\lim_{N\to\infty} \E\left[ |M_m(\mathcal{D}_1, A, \mathbf{B}, N) - \E
[M_m(\mathcal{D}_1, A, \mathbf{B}, N)]|^4 \right] \ = \
O_m\left(\frac{1}{N^2}\right).
\end{equation}
The proof is completed by three steps. By the triangle inequality,
\begin{align}
\begin{split}
&|M_m(\mathcal{D}_1, A, \mathbf{B}, N) - M_m(\mathcal{D}_1, N)| \\
&\hspace{1cm} \leq\ |M_m(\mathcal{D}_1, A, \mathbf{B}, N) - M_m(
\mathcal{D}_1, N)| + |M_m(\mathcal{D}_1, N) - M_m(\mathcal{D}_1)|.
\end{split}
\end{align}
As the second term tends to zero, it suffices to show the first tends
to zero for almost all $\mathcal{D}_1$.

Chebychev's inequality states that, for any random variable $X$ with
mean zero and finite $\ell$\textsuperscript{th} moment,
\be
\text{Prob}(|X| \ge \gep) \ \le \ \frac{\E[|X|^\ell]}{\gep^\ell}.
\ee
Note $\E[ M_m(\mathcal{D}_1, A, \mathbf{B}, N) - M_m(\mathcal{D}_1, N)] = 0$,
and following \cite{HM} one can show the fourth moment of
$M_m(\mathcal{D}_1, A, \mathbf{B}, N) - M_m(\mathcal{D}_1, N)$ is
$O_m\left(1 / N^2\right)$; we will discuss this step in greater
detail below. Then Chebychev's inequality (with $\ell = 4$) yields
\begin{align}
\begin{split}
&\p_\N(|X_{m;N}(\mathcal{D}_1) - X_m(\mathcal{D}_1)| \geq \gep) \\
&\hspace{3cm}\le\  \frac{\E[|M_m(\mathcal{D}_1, A, \mathbf{B}, N)
- M_m(\mathcal{D}_1, N)|^4]}{\gep^4}
\end{split}\nonumber \\
&\hspace{3cm} \le\  \frac{C_m}{N^2 \gep^4}.
\end{align}

The proof of almost sure convergence is completed by applying the
Borel-Cantelli Lemma and proving \eqref{eq:thmprestrongB}; we sketch
the proof below.

We assume $p$ is even for convenience (though see Remark 6.17 of
\cite{HM}). We expand the expected value on the left hand side of
\eqref{eq:thmprestrongB} into
\begin{equation}
    \begin{split}
        &\mathbb{E}\left[M_{m}(\mathcal{D}_1,N)^4\right]-4\mathbb{E}\left[M_{m}
        (\mathcal{D}_1,N)^3\right]\mathbb{E}\left[M_{m}(\mathcal{D}_1,N)\right] \\
        & \: + 6\mathbb{E}\left[M_{m}(\mathcal{D}_1,N)^2\right] \mathbb{E}\left[M_{m}
        (\mathcal{D}_1,N)\right]^2 -3\mathbb{E}\left[M_{m}(\mathcal{D}_1,N)\right]\mathbb{E}\left[M_{m}
        (\mathcal{D}_1,N)\right]^3.
    \end{split}\label{4thMoment}
\end{equation}
The terms in \eqref{4thMoment} can be expressed in terms of entries of
$\mathcal{D}_1$, which we denote by $c_j$. For example, the first term becomes
\begin{equation}
    \mathbb{E}\left[M_{2m}(\mathcal{D}_1,N)^4\right] \ =\ \frac{1}{(2N)^{4m+4}}\sum\limits_t\sum\limits_u\sum\limits_v\sum\limits_w
    \mathbb{E}\left[c_{ts}c_{us}c_{vs}c_{ws}\right]
\end{equation}
where $c_j$ can either be entries in $A$ or $\mathbf{B}$. The proofs in \S6 of
\cite{HM} can be applied analogously to the disco case as well, as most of the
proofs are simple calculations based on the number of degrees of freedom. The
only difference is that crossovers can occur between entries from $A$ or
$\mathbf{B}$. For crossovers between entries from the PST matrix $A$, the same
results from \S3 of \cite{MMS} hold true. For crossovers between entries from
the RS matrix $\mathbf{B}$, more degrees of freedom will be lost. Thus the
analogues of the proofs in \cite{HM} hold in the disco case as well, which
completes the proof of almost sure convergence.


\section{Combinatorics of Moment Calculations}\label{sec:combo_prob}
Method of moments and the pairing arguments have long been used in random matrix
theory to compute contribution of terms. In particular, the application of
enumerative combinatorics in some involved computation offers a new view and helps
simplify proofs. In the case of the disco of a real RS matrix $B$
and a PST random matrix $A$, the contribution of mixed
terms are characterized by the crossings of pairings of $a$'s and $b$'s in a mixed
product. From \cite{SS} we know that crossings of $b$'s contribute zero, crossings
of $a$ and $b$ contribute zero, and everything else contributes fully. However,
there seems to be no obvious way to compute the contributing pairings by hand as
we go to higher moments. Thus, we seek a new perspective in the hope of getting a
closed form expression.

In this section, we transform the problem of computing the contribution of mixed
products in \eqref{expression: even moment of 1-disco} into a purely combinatorial
problem. Predictably, the problem looks like an analogue of variations of the
Catalan numbers. Although we cannot get a closed form which can be used to compute
an exact bound of the even moments, we obtain a beautiful compact expression that
is computable.

We start by defining some notations.

\begin{defi}
    A tree is a connected undirected graph with no cycles. It is a spanning tree
    of a graph $G$ if it spans $G$ (that is, it includes every vertex of $G$) and
    is a subgraph of $G$ (every edge in the tree belongs to $G$). A spanning tree
    of a connected graph $G$ can also be defined as a maximal set of edges of $G$
    that contains no cycle, or as a minimal set of edges that connect all vertices.
\end{defi}

\begin{defi}
    Given a spanning tree $G=(V,E)$, we define the number of rotational symmetry
    of $G$, denoted $\sigma_r(G)$, to be the number of graph automorphisms of $G$
    preserving the cyclic order of labels. Equivalently, embedding $G$ in a plane,
    and labeling vertices in $V$, $\sigma_r(G)$ gives the number of graphs one can
    obtain by rotating the labels of $V$ on the plane.
\end{defi}

\begin{thm}\label{Combo}
    Consider placing $2\alpha$ red dots and $2\beta$ blue dots onto the
    circumference of a circle. The blue dots all need to be paired with
    blue chords; the red dots all need to be paired with red chords. Blue
    chords cannot intersect any chords, but red chords are allowed to
    intersect other red chords. Let $\mathcal{P}(\alpha, \beta)$ denote
    the total number of ways to configure and pair up the $2\alpha+2\beta$
    dots (order doesn't matter). Then

    \begin{equation}
    \mathcal{P}(\alpha,\beta) = \sum_{\substack{G\text{ spanning tree, }\\
    |V|=\beta+1\\
    deg(v)=d_1, \dots, d_{\beta+1}, v\in V}}\frac{1}{\sigma_r(G)} \sum_{\substack{\gamma_1+\cdots+\gamma_{\beta+1}=\alpha\\ \gamma_s\in
    \mathbb{N},s=1,\dots,\beta+1}} \prod_{s=1}^{\beta+1}(2\gamma_s-1)!!\binom{2\gamma_s
    +d_s-1}{d_s-1}.
    \end{equation}

\begin{proof}
    We first place the $2\beta$ blue dots onto the circumference of a
    circle, and pair them without crossing. The circle is cut into
    $\beta+1$ regions by the blue chords. We construct a mapping $\phi$
    from the set of pairings of $2 \beta$ of the $b$'s (two pairings are
    identical up to rotation) to the set of spanning trees on $\beta+1$
    vertices by constructing a dual graph $G$ as follows: place a vertex
    in each of the $\beta+1$ regions of the circle; for any two vertices,
    draw an edge connecting them if and only if their corresponding regions
    share a blue chord. This way, we get a spanning tree $G$ on $\beta+1$
    vertices. Notice that the degree of each vertex in $G$ is exactly the
    number of arcs on the circle surrounding the corresponding region.

    We show that $\phi$ is a bijection. By the construction above, $\phi$
    is injective. To show surjectivity, given a spanning tree $G$ on
    $\beta+1$ vertices, we know the degrees of each vertex. We enclose
    $G$ by a circle, then draw a blue chord across each edge, with the
    chord's endpoints on the circumference of the circle, and avoiding
    crossing while drawing these chords. We can recover a pairing graph
    representation uniquely from $G$.

    Now consider any spanning tree $G=(V,E)$ with $|V|=\beta+1$. Recovering
    a pairing of $2\beta$ of the $b$'s from $G$, we then put the $a$'s onto the
    circumference. Note that if an $a$ is paired with another in a different
    region formed by the blue chords and the arcs, this pairing of $a$ will
    intersect at least one blue chord, which is not allowed. Therefore all
    the $a$'s have to be paired within their own region, and we must have an
    even number of $a$'s within each region. We count the ways to put $2\alpha$
    of the $a$'s into the $\beta+1$ regions and pair them up within each region. Let $2\gamma_1,\dots,2\gamma_{\beta+1}$ denote the number of $a$'s in each
    region. $\gamma_1+\cdots+\gamma_{\beta+1}=\alpha$. Then we have
    $(2\gamma_s-1)!!$ ways to pair the $a$'s insider the $s$\textsuperscript{th} region.
    Recall that the $s$\textsuperscript{th} region contains $d_s$ arcs. By the result of the
    well-known ``stars and bars'' problem,
    there are
    \begin{equation}
    \binom{2\gamma_s+d_s-1}{d_s-1}
    \end{equation}
    ways to distribute $\gamma_s$ of the $a$'s in the region onto these arcs.
    In total, for a given vector $(\gamma_1,\dots,\gamma_{\beta+1})$, we have
    \begin{equation}
        \prod_{s=1}^{\beta+1}(2\gamma_s-1)!!\binom{2\gamma_s+d_s-1}{d_s-1}
    \end{equation}
    ways to place and pair the $a$'s.

    Summing over all possible vectors $(\gamma_1,\dots,\gamma_{\beta+1})$ with $\gamma_1+\cdots+\gamma_{\beta+1}=\alpha$, we obtain the number of ways to
    pair $2\alpha$ of the $a$'s with the given tree $G$.

    However, we are not precisely counting the desired number. Because we can
    potentially rotate the tree to get an exact same tree, we are over counting
    by the rotational symmetry of $G$, defined above as $\sigma_r(G)$. Dividing
    by $\sigma_r(G)$, we get the desired formula.

\end{proof}
\end{thm}

\begin{thm}
    The contribution of all the terms in the equivalence class $\left[A^I B^J\right]$
    is given by
    \begin{equation}
        \frac{(I+J)}{2^{(I+J)/2}} \ \mathcal{P} \left(\frac{I}{2}, \ \frac{J}{2} \right).
    \end{equation}
    \begin{proof}
        First note from Lemma \ref{oddmoments} and \eqref{eq1.2} that $I$ and
        $J$ have to be even. Regarding entries in $A$ as red dots and entries
        in $B$ as blue dots, we can then apply Theorem \ref{Combo} to
        $I / 2$ and $J / 2$. By the cyclic property of trace, a
        given term $\prod_{k=1}^p A^{i_k} B^{j_k}$ where $\sum_{k=1}^p
        (i_k+j_k)=I+J$ appears exactly $(I+J)$ times in the expansion
        \eqref{eq1.1}. All of these $(I+J)$ terms correspond to one same
        configuration and pairing of $I$ red dots and $J$ blue dots. By Theorem
        \ref{Combo}, the total number of configurations and pairings of $I$ red
        dots and $J$ blue dots is $\mathcal{P}(I / 2, \ J / 2)$.
        Multiplying, then dividing the product by the normalization factor $2(I+J)/2$,
        we complete the proof.
    \end{proof}
\end{thm}

\begin{center}
\begin{table}[ht]
\renewcommand{\arraystretch}{1.5}
\begin{tabular}{|c|c|c|c|c|}\hline
\hspace{0.5cm} $k$ \hspace{0.5cm} & \hspace{0.5cm} $I$ \hspace{0.5cm} & \hspace{0.5cm} $J$ \hspace{0.5cm} & $(I+J) \ \mathcal{P} \left(\frac{I}{2}, \ \frac{J}{2} \right)$ & Contribution of $\left[A^IB^J\right]$ \\ \hline
4 & 2 & 2 & 4 & 1  \\ \hline
6 & 2 & 4 & 15 & 1.875 \\ \hline
6 & 4 & 2 & 21 & 2.625 \\ \hline
8 & 2 & 6 & 56 & 3.5 \\ \hline
8 & 4 & 4 & 112 & 7 \\ \hline
8 & 6 & 2 & 144 & 9 \\ \hline
\end{tabular}
\caption{\label{tab:combo-data}Contribution of $\left[A^IB^J\right]$ for smaller values of $I$ and $J$.}
\end{table}
\end{center}


\section{Bounds on Mixed Products}\label{sec:bounds_mixed_products}

The primary obstacle in proving the convergence of moments of
$\mathcal{D}_d$ is bounding the contribution of the mixed product
terms arising in the expansion of \eqref{eq1.1}, with form
\begin{equation}
    \mathbb{E} \left[ \Tr\left( A^{I_1} B^{J_1} A^{I_2} B^{J_2} \cdots
    A^{I_p} B^{J_p} \right) \right].
\end{equation}
In this section we demonstrate that a generalization of
H\"{o}lder's Inequality in the $p$-Schatten norm allows us to
bound the contribution of arbitrary products of Hermitian
matrices. We begin by defining singular values and
highlight their
key relation to the eigenvalues of Hermitian matrices.

\begin{defi}\label{def:singular_value}
Given an $N \times N$ matrix $X$, let $\alpha_1, \alpha_2, \cdots, \alpha_N$ be the eigenvalues of the $N \times N$ matrix $X^*X$, where $X^*$ is the conjugate transpose of $X$. Then the singular values of $X$ are
\begin{equation}
    \sigma_i(X)\ =\ \sqrt{\alpha_i}.
\end{equation}
In particular, if $X$ is Hermitian, then $\sigma_i(X) = \abs{\lambda_i(X)}$, where $\lambda_i(X)$'s are eigenvalues of $X$.
\end{defi}

\begin{defi}\label{def:p-Schatten}
Given an $N \times N$ matrix $X$, the $p$-Schatten norm is defined by
\begin{equation}
    ||X||_p \ =\ \left(\sum\limits_{i=1}^N\sigma_i^p(X)\right)^{\frac{1}{p}}
\end{equation}
where $\sigma_i(X)$ for $i = 1, 2 \ldots, N$ are the singular values of $X$.
\end{defi}

We state the following well-known result; citations and reproduction
of proof are provided in the appendix.

\begin{thm}{(Generalized H\"older's Trace Inequality.)}
\label{thm:GeneralHolderTrace}
Let $X_1,X_2,\dots,X_k$ be $N\times N$ matrices and $p_1,p_2,\dots,p_k \in \mathbb{R}^+$
such that $\sum_{i=1}^k p_1^{-1} = 1$. Then
\begin{equation}
    \abs{\Tr(X_1 X_2 \cdots X_k)} \ \leq\ \prod_{i=1}^k \norm{X_i}_{p_i}.
\end{equation}

\begin{proof}
See Theorem \ref{thm:GeneralHolderTrace_appendix} in the Appendix.
\end{proof}
\end{thm}

The preceding result allow us to prove a key bound
on the contribution of arbitrary non-commutative products
of two Hermitian matrices.

\begin{thm}\label{thm:expected_val_trace_bound}
Let $A$, $B$ be $N \times N$ Hermitian matrices and let
$I_i, J_i \in 2\mathbb{Z}^+$ such that $\sum_{i=1}^n I_i = I$ and
$\sum_{i=i}^n J_i = J$, where $I + J = K$. Then
\begin{equation}
    \mathbb{E} \left[\Tr \left( \prod_{i=1}^n A^{I_i} B^{J_i} \right) \right]
    \ \leq\ \mathbb{E}[\Tr (A^K)]^{\frac{I}{K}} \: \mathbb{E}[\Tr (B^K)]^{\frac{J}{K}}.
\end{equation}

\begin{proof}
As $A$, $B$ are Hermitian, it follows that
\begin{equation}
    \sigma_i(A) = \abs{\lambda_i(A)} \qquad\text{and}\qquad \sigma_i(B) \ = \
    \abs{\lambda_i(B)}
\end{equation}
so that
\begin{equation}
    \sigma_i^K (A) = \lambda_i^K (A) \qquad\text{and}\qquad \sigma_i(B) \ = \
    \lambda_i^K(B)
\end{equation}
for even $K = I + J$. Taking $p_i = 1 / K$ in Theorem
\ref{thm:GeneralHolderTrace} and applying the Eigenvalue Trace Lemma give
\begin{align}
    \Tr \left( \prod_{i=1}^n A^{I_i} B^{J_i} \right) \ &\leq\ \left(\sum_{i=1}^N\sigma_i^K(A)\right)^{\frac{I}{K}}\left(\sum_{i=1}^N
    \sigma_i^K(B)\right)^{\frac{J}{K}} \nonumber \\
    \ &=\ \left(\sum_{i=1}^N\lambda_i^K(A)\right)^{\frac{I}{K}}\left(\sum_{i=1}^N
    \lambda_i^K(B)\right)^{\frac{J}{K}} \nonumber \\
    \ &=\ \Tr(A^K)^{\frac{I}{K}}\Tr(B^K)^{\frac{J}{K}}.
\end{align}
Taking the expected value and applying Jensen's Inequality \cite{Jen} yield
\begin{align}
    \mathbb{E} \left[ \Tr \left( \prod_{i=1}^n A^{I_i} B^{J_i} \right) \right]
    \ &\leq\ \mathbb{E} \left[ \Tr \left( A^K \right)^{\frac{I}{K}} \right]
    \mathbb{E}\left[ \Tr \left( B^K \right)^{\frac{J}{K}} \right] \nonumber \\
    \ &\leq\ \mathbb{E} \left[ \Tr \left( A^K \right) \right]^{\frac{I}{K}}
    \mathbb{E}\left[ \Tr \left( B^K \right) \right]^{\frac{J}{K}}.
\end{align}
\end{proof}
\end{thm}


\section{Bounds on the Moments of \texorpdfstring{$\mathcal{D}_1$}{Lg}}\label{sec:bounds_on_moments}
Given $\mathcal{D}_1$ constructed as in Definition \ref{def: 1-disco},
numerical experiments across various ensembles $\mathcal{E}_A$ and
$\mathcal{E}_B$ suggest that the $k$\textsuperscript{th} moments of
the limiting eigenvalue distribution of $\mathcal{D}_1$ are bounded
by the $k$\textsuperscript{th} moments of the component matrices $A$
and $B_0$, up to a scaling constant $C_k$ dependent on $k$. Define
\begin{equation}\label{eq:1_disco_bound_decomp}
    \hat{A}\ = \ \begin{bmatrix}
    A & 0 \\
    0 & A
    \end{bmatrix}
    \qquad \text{and} \qquad
    \hat{B}\ = \ \begin{bmatrix}
    0 & B_0 \\
    B_0 & 0
    \end{bmatrix}
\end{equation}
so that $\mathcal{D}_1 = \hat{A} + \hat{B}$.

\begin{lem}\label{lem:D1_traces}
For $k \in 2\mathbb{Z}^+$, $\Tr \left[ \hat{A}^k \right] = 2 \Tr
\left[A^k \right]$ and $\Tr \left[ \hat{B}^k \right] =2 \Tr
\left[ B_0^k \right]$.
\end{lem}
\begin{proof}
The first equality follows from the observation that
\begin{equation}
    \hat{A}^k = \begin{bmatrix}
    A^k & 0 \\
    0 & A^k
    \end{bmatrix}.
\end{equation}
In the second equality, we note that for all $k \geq 2$ that
\begin{equation}
    \hat{B}^k = \begin{bmatrix}
    B_0^k & 0 \\
    0 & B_0^k
    \end{bmatrix}
\end{equation}
from which the claim follows.
\end{proof}

\begin{lem}\label{lem:D_1_Mixed_Products_Bound}
Let $I,J \in 2\mathbb{Z}^+$ and $i_h, j_h \in \mathbb{Z}^+$ such that
$\sum_{h=1}^p i_h = I$, $\sum_{h=1}^p j_h = J$, and $I + J = k$. Then
\begin{equation}
    \lim_{N \to \infty} \frac{1}{(2N)^{\frac{k}{2}+1}} \mathbb{E}
    \left[ \Tr \left( \prod_{h=1}^p \hat{A}^{i_h} \hat{B}^{j_h} \right)
    \right] \ \leq\ \frac{1}{2^{\frac{k}{2}}} M_k(A)^{\frac{I}{k}}
    M_k(B_0)^{\frac{J}{k}}.
\end{equation}

\begin{proof}
Applying Theorem \ref{thm:expected_val_trace_bound} gives
\begin{align}
     \frac{1}{(2N)^{\frac{k}{2}+1}} \mathbb{E} \left[ \Tr \left( \prod_{h=1}^p
     \hat{A}^{i_h} \hat{B}^{j_h} \right) \right]
    \ &\leq\ \frac{1}{(2N)^{\frac{k}{2}+1}} \mathbb{E} \left[ \Tr \left(
    \hat{A}^k \right) \right]^{\frac{I}{k}} \mathbb{E} \left[ \Tr \left(
    \hat{B}^k \right) \right]^{\frac{J}{k}}.
\end{align}
From Lemma \ref{lem:D1_traces} we then have
\begin{align}
    \begin{split}
    & \frac{1}{(2N)^{\frac{k}{2}+1}} \mathbb{E} \left[ \Tr \left( \prod_{h=1}^p
    \hat{A}^{i_h} \hat{B}^{j_h} \right) \right] \\
    &\hspace{2cm} \leq\ \frac{2^{\frac{I}{k}
    + \frac{J}{k}}}{2^{\frac{k}{2}+1}} \left(\frac{\mathbb{E} \left[ \Tr
    \left( A^k \right) \right]}{N^{\frac{k}{2}+1}}\right)^{\frac{I}{k}}
    \left(\frac{\mathbb{E} \left[ \Tr \left( B_0^k \right) \right]}{N^{
    \frac{k}{2}+1}}\right)^{\frac{J}{k}}.
    \end{split}
\end{align}
Noting that $I / k + J / k = 1$ and taking the limit as $N \to \infty$
yield the desired result.
\end{proof}
\end{lem}

The preceding lemmas allow us to establish the following bound on the moments
of the limiting eigenvalue distribution of $\mathcal{D}_1$ in terms of the
moments of the component matrices $A$ and $B_0$.

\begin{thm}
Let $A$ and $B_0$ be $N \times N$ random Hermitian matrices chosen from ensembles
$\mathcal{E}_A$ and $\mathcal{E}_B$, respectively, whose limiting eigenvalue
distributions have moments all finite and appropriately bounded. Let $\mathcal{D}_1$
be constructed by $\hat{A} + \hat{B}$ as in \eqref{eq:1_disco_bound_decomp}. Then
\begin{equation}
    2^{1-\frac{k}{2}} \min \left\lbrace M_k(A),M_k(B_0) \right\rbrace \ \leq\
    M_k(\mathcal{D}_1) \leq 2^ {\frac{k}{2}-1} \max \left \lbrace M_k(A),M_k(B_0)
    \right\rbrace.
\end{equation}

\begin{proof}
Applying Lemmas \ref{lem:D1_traces} and \ref{lem:D_1_Mixed_Products_Bound} to the
$k$\textsuperscript{th} moment of $\mathcal{D}_1$ gives
\begin{align}
    \begin{split}
        \lim_{N \to \infty} &\frac{\mathbb{E} \left[ \Tr \left( \mathcal{D}_1^k
        \right) \right]}{(2N)^{\frac{k}{2}+1}} \ = \ \lim_{N \to \infty}
        \frac{\Tr \left( (\hat{A}+\hat{B})^k \right)}{(2N)^{\frac{k}{2}+1}}
    \end{split} \nonumber \\
    \ &=\ \lim_{N \to \infty} \frac{1}{(2N)^{k/2+1}} \mathbb{E}\left[ \Tr \left(
    \hat{A}^k \right) + \sum \Tr\left( \prod_{h=1}^p \hat{A}^{i_h} \hat{B}^{j_h}
    \right)+\Tr \left( \hat{B}^k \right) \right] \nonumber \\
    \ &\leq\ \frac{1}{2^{\frac{k}{2}}}\left( M_k(A)+M_k(B)+\sum_{\substack{I=2\\
    I\text{ even }}}^{k-2}\binom{k}{I}M_k(A)^{\frac{I}{k}}M_k(B)^{\frac{J}{k}}\right).
\end{align}
Let $M = \max \left\lbrace M_k(A), M_k(B) \right\rbrace$. Then
\begin{align}
    \lim_{N \to \infty} \frac{\mathbb{E} \left[ \Tr \left( \mathcal{D}_1^k
    \right) \right]}{(2N)^{\frac{k}{2}+1}} \ &\leq\ \frac{1}{2^{\frac{k}{2}}}
    \left[2M +\sum_{\substack{I=2 \\ I\text{ even }}}^{k-2}\binom{k}{I}M\right] \nonumber \\
    \ &=\ \frac{1}{2^{\frac{k}{2}}} M \sum_{i=0}^{k-1}\binom{k-1}{i} \nonumber \\
    \ &=\ \frac{2^{k-1}}{2^{\frac{k}{2}}} M
\end{align}
which yields the desired upper bound. The lower bound is obtained by observing
\begin{align}
    \lim_{N \to \infty} \frac{ \mathbb{E} \left[ \Tr \left( \mathcal{D}_1^k
    \right) \right]}{(2N)^{\frac{k}{2}+1}} \ &\geq\ \lim_{N \to \infty}
    \frac{1}{(2N)^{\frac{k}{2}+1}} \mathbb{E} \left[ \Tr \left( \hat{A}^k \right)
    +\Tr \left( \hat{B}^k \right) \right] \nonumber \\
    \ &=\  \frac{2}{2^{\frac{k}{2}+1}} \lim_{N \to \infty}  \mathbb{E}
    \left[ \frac{\Tr \left( A^k \right)}{N^{\frac{k}{2}+1}}+ \frac{\Tr
    \left( B^k \right)}{N^{\frac{k}{2}+1}} \right] \nonumber \\
    \ &\geq\ \frac{1}{2^{\frac{k}{2}-1}}\min \left\lbrace M_k(A), M_k(B)
    \right\rbrace.
\end{align}
\end{proof}
\end{thm}


\section{Disco of Ensembles with Same Limiting Spectral Measure}\label{sec:LW_disco}

We now consider a simpler instance of the disco matrix
\begin{equation}
    \mathcal{D}_1 \left( X, \mathbf{B} \right) \ =\ 
    \begin{bmatrix}
    X   & B_1 \\
    B_1 & X
    \end{bmatrix}
\end{equation}
where $X$ and $B_1$ are drawn from possibly different ensembles that have the same limiting eigenvalue distribution $\nu$.
We show that the limiting distribution of $\mathcal{D}_1 \left( X, \mathbf{B} \right)$
converges to that of $\mathcal{E}_B$. Analysis of disco matrices $\mathcal{D}_d (A, \mathbf{B})$
where $A$ and the $B_i$'s are drawn from arbitrary pairs of ensembles (with different limiting eigenvalue distributions)
contains sub-problems that may be reduced to this simpler case.


\begin{thm}\label{thm:luntzlara-wang-tweak}
Let $X$ and $B_1$ be $N \times N$ random matrices chosen to have the same
limiting eigenvalue distribution with entries independent and identically distributed from a fixed
distribution function $p(x)$ with mean $0$, variance $1$, and appropriately
bounded finite moments. Then the moments of the limiting eigenvalue
distribution of $\mathcal{D}_1 (X, \mathbf{B})$
satisfy $M_{k} \left( \mathcal{D}_1 (X, \mathbf{B}) \right) = M_{k} ( X ) = M_{k} ( B_1 )$
for all $k \in \mathbb{N}$.
\end{thm}

\begin{proof}
By the Eigenvalue Trace Lemma, we can express the moments of $\mathcal{D}_1 = \mathcal{D}_1 (X, \mathbf{B})$
in terms of the trace of powers of $\mathcal{D}_1$. Let $M_k \left(
\mathcal{D}_1 \right)$ denote the $k$\textsuperscript{th} moment of the
normalized eigenvalue distribution of $\mathcal{D}_1$. The average
$k$\textsuperscript{th} moment of $\mathcal{D}_1$ is
\begin{align}  \label{eq:disAvgkMoment}
M_k(\mathcal{D}_1) \ &=\ \int \cdots \int \frac{1}{2N} \frac{\Tr
(\mathcal{D}_1^k)}{(2N)^{k/2}} \prod p(s_{ij})d(s_{ij})  \nonumber \\
    \ &=\  \frac{1}{(2N)^{k/2+1}} \int \cdots \int \Tr(\mathcal{D}_1^k)
    \prod p(s_{ij})d(s_{ij})
\end{align}
where $s_{ij}$ is the $(i,j)$ entry of $\mathcal{D}_1$. To compute the
trace of $\mathcal{D}_1^k$, we first diagonalize $\mathcal{D}_1$, then
take the $k$\textsuperscript{th} power, giving
\be \label{eq:disPowerTrace}
\mathcal{D}_1^k\ =\ \begin{bmatrix} I/2 & I/2 \\ I/2 & -I/2 \end{bmatrix}
\begin{bmatrix} (X + B_1)^k & 0 \\ 0 & (X - B_1)^k \end{bmatrix}
\begin{bmatrix} I & I \\ I & -I \end{bmatrix}.
\ee
By the properties of the trace function (linearity and basis-independence),
we have that
\begin{align}
\Tr(\mathcal{D}_1^k) &= \Tr( (X +B_1)^k) + \Tr( (X-B_1)^k) \nonumber \\
        \begin{split}
        \ &=\ \sum\limits_{l=0}^{k} \sum_{\substack{i_1+\cdots+i_p=k-l\\
         j_1+\cdots+j_p=l}} \Tr( X^{i_1} B_1^{j_1} \cdots X^{i_p} B_1^{j_p})\\
         &\qquad + \sum_{l=0}^{k} (-1)^l \sum_{\substack{i_1+\cdots+i_p=k-l\\
         j_1+\cdots+j_p=l}} \Tr(X^{i_1} B_1^{j_1} \cdots X^{i_p} B_1^{j_p})
        \end{split}\nonumber \\
         \ &=\ 2\sum_{\substack{l=0 \\ l:\text{even}}}^{k}
         \sum\limits_{\substack{i_1+\cdots+i_p=k-l\\ j_1+\cdots+j_p=l}}
         \Tr(X^{i_1} B_1^{j_1} \cdots X^{i_p} B_1^{j_p}). \label{eq:traceD^k}
\end{align}
Denoting by $x$, $b$ entries of $X$, $B_1$, respectively,
it then suffices to calculate
\begin{multline}
     \int \cdots \int \Tr(X^{i_1} B_1^{j_1} \cdots X^{i_p} B_1^{j_p}) \\
 \ = \ \int \cdots \int \sum_{t_1=1}^{N} \cdots \sum_{t_k=1}^{N} x_{{t_1}{t_2}}
 \cdots x_{t_{i_1}t_{{i_1}+1}} b_{t_{{i_1}+1}t_{{i_1}+2}} \cdots
 b_{t_{i_1+j_1}t_{i_1+j_1+1}} \cdots b_{{t_k}{t_1}}
 \label{eq:traceOfEachTerm}
\end{multline}
where $0 \leq l \leq k$, $l$ is even, $i_1+\cdots+i_p=k-l$, and
$j_1+\cdots+j_p=l$. By the method of moments and the Eigenvalue
Trace Lemma, it suffices to compute the trace of the
$k$\textsuperscript{th} power of $\mathcal{D}_1$. We use the
computation \eqref{eq:disPowerTrace}, \eqref{eq:traceD^k}, and
\eqref{eq:traceOfEachTerm}.

For an odd moment, a term in \eqref{eq:traceOfEachTerm} either
integrates to 0 (if some $x_{ij}$ or $b_{ij}$ are not paired
up) or contributes 0 to the $k$\textsuperscript{th} moment in the
limit (if the entries are paired up, and this will result in a
term with less degrees of freedom). Thus, all the odd moments of
$\mathcal{D}_1$ are 0.

Then we consider the even moments and compute $\Tr(\mathcal{D}_1)$
by \eqref{eq:traceD^k}. Each term of \eqref{eq:traceOfEachTerm}
should be paired up in order to contribute in the limit. Utilizing
the result of the normalized $2k$\textsuperscript{th} moment of
the original matrix ensemble, we have that
\begin{equation}
    M_{2k}(X) \ = \ \dfrac{1}{(2N)^{k+1}} \int \cdots \int \Tr(X^{2k})
    \prod p(x_{ij})d(x_{ij}),
\end{equation}
Since each pair of the entries of $\mathcal{D}_1$ can be chosen from
$X$ or $B_1$, we get $2^k$ distinct configurations of pairing up
the entries of $\mathcal{D}_1$ from one configuration of $X$. So
\begin{align}
    2\int \cdots \int \Tr(\mathcal{D}_1^{2k}) \prod p(s_{ij})d(s_{ij})
    \ &=\ 2^{k+1} \int \cdots \int \Tr(X^{2k}) \prod p(x_{ij})d(x_{ij})
    \nonumber \\
    \ &=\  (2N)^{k+1} M_{2k}(X).
\end{align}
Therefore, the normalized $2k$\textsuperscript{th} moment of
$\mathcal{D}_1$ is
\begin{align}
M_{2k}(\mathcal{D}_1) \ = \ \frac{ (2N)^{k+1} M_{2k}(X)}{(2N)^{k+1}}
\ = \ M_{2k}(X),
\end{align}
which gives the desired result.
\end{proof}

Theorem \ref{thm:luntzlara-wang-tweak} shows convergence of the
expected value of the limiting distribution moments of
$\mathcal{D}_1$ to those of $X$ and $B_1$. To prove that the
limiting eigenvalue distribution of $\mathcal{D}_1$ converges to
that of $X$ and $B_1$ we must further show that the variance
tends to zero, which is sufficient to establish weak convergence
(see Definition \ref{def:weak_convergence}).

By Theorem \ref{thm:luntzlara-wang-tweak}, we know the moments
$M_k \left( \mathcal{D}_1 \right)$ exist
and are finite. To prove we have weak convergence to the limiting
spectral measure we need to show that the variances tend to 0. We
must show
\begin{equation}\label{eq:var_to_zero}
    \lim_{N \to \infty} \left( \mathbb{E} \left[ M_{m} \left(
    \mathcal{D}_1 \right)^2 \right] - \mathbb{E} \left[ M_{m}
    \left( \mathcal{D}_1 \right) \right]^2 \right) \ = \ 0.
\end{equation}
 We assume $m = 2k$; a similar proof works for odd $m$. By equation
 \eqref{eq:1-disco_decomp}, we have
\begin{align}
    \begin{split}
        \mathbb{E} &\left[ M_{2k} \left( \mathcal{D}_1 \right)^2 \right]
    \end{split} \nonumber \\
    \ &=\ \frac{1}{(2N)^{ 2k + 2 }} \mathbb{E} \left[ \left(\Tr
    \left( X^{2k} + \sum_{T=1}^k \prod_{h=1}^p X^{I_h} B_1^{J_h}
    + B_1^{2k} \right) \right)^2 \right] \nonumber \\
    \begin{split}
        \ &=\ \frac{1}{(2N)^{ 2k + 2 }} \Bigg( \mathbb{E}\left[ \Tr
        \left( X^{2k} \right)^2 \right] + \mathbb{E}\left[ \Tr
        \left( B_1^{2k} \right)^2 \right] \\
        &\quad+ 2 \mathbb{E}\left[ \Tr \left( X^{2k} \right) \right]
        \mathbb{E} \left[ \Tr \left( B_1^{2k} \right) \right] +
        \mathbb{E}\left[ \left( \sum_{T=1}^k \Tr \left( \prod_{h=1}^p
        X^{I_h} B_1^{J_h} \right) \right)^2 \right]\\
        &\quad + 2 \mathbb{E}\left[ \Tr \left( X^{2k} \right)
        \sum_{T=1}^k \Tr \left( \prod_{h=1}^p X^{I_h} B_1^{J_h} \right)
        \right]\\
        &\quad+ 2 \mathbb{E}\left[ \Tr \left( B_1^{2k} \right) \sum_{T=1}^k
        \Tr \left( \prod_{h=1}^p X^{I_h} B_1^{J_h} \right) \right] \Bigg)
        \label{eq:ev_of_square}
    \end{split}
\end{align}
and
\begin{align}
        \begin{split}
        \mathbb{E} &\left[ M_{2k} \left( \mathcal{D}_1 \right) \right]^2
    \end{split} \nonumber \\
    \ &=\ \frac{1}{(2N)^{ 2k + 2 }} \mathbb{E} \left[ \Tr \left( X^{2k}
    + \sum_{T=1}^k \prod_{h=1}^p X^{I_h} B_1^{J_h} + B_1^{2k} \right)
    \right]^2 \nonumber \\
    \begin{split}
        \ &=\ \frac{1}{(2N)^{ 2k + 2 }} \Bigg( \mathbb{E}\left[ \Tr
        \left( X^{2k} \right) \right]^2 + \mathbb{E}\left[ \Tr
        \left( B_1^{2k} \right) \right]^2 \\
        &\quad+ 2 \mathbb{E}\left[ \Tr \left( X^{2k} \right) \right]
        \mathbb{E} \left[ \Tr \left( B_1^{2k} \right) \right] + \mathbb{E}
        \left[ \sum_{T=1}^k \Tr \left( \prod_{h=1}^p X^{I_h} B_1^{J_h}
        \right)  \right]^2 \\
        &\quad + 2 \mathbb{E}\left[ \Tr \left( X^{2k} \right) \right]
        \mathbb{E} \left[ \sum_{T=1}^k \Tr \left( \prod_{h=1}^p X^{I_h}
        B_1^{J_h} \right) \right]\\
        &\quad+ 2 \mathbb{E}\left[ \Tr \left( B_1^{2k} \right) \right]
        \mathbb{E} \left[ \sum_{T=1}^k \Tr \left( \prod_{h=1}^p X^{I_h}
        B_1^{J_h} \right) \right] \Bigg) \label{eq:square_of_ev}
    \end{split}
\end{align}
where $\sum_{h=1}^p I_h = 2k - 2T$ and $\sum_{h=1}^p J_h = 2T$. Since
the limiting eigenvalue distributions of $X$ and $B_1$ converge to
that of $\mathcal{E}_B$ by assumption, we have
\begin{align}
    \begin{split}
        \lim_{N\to \infty} \frac{1}{(2N)^{2k+2}} & \left( \mathbb{E}
        \left[ \Tr \left( X^{2k} \right)^2 \right]-  \mathbb{E}
        \left[ \Tr \left( X^{2k} \right) \right]^2 \right)\\
        \ &=\ \frac{1}{2^{2k+2}} \lim_{N\to \infty} \left( \mathbb{E}
        \left[ M_{2k}(X, N)^2 \right]-  \mathbb{E} \left[ M_{2k}(X, N)
        \right]^2 \right)
        \end{split} \nonumber \\
        \ &=\ 0
\end{align}
and similarly
\begin{align}
    \begin{split}
        \lim_{N\to \infty} \frac{1}{(2N)^{2k+2}} & \left( \mathbb{E}
        \left[ \Tr \left( B_1^{2k} \right)^2 \right]-  \mathbb{E}
        \left[ \Tr \left( B_1^{2k} \right) \right]^2 \right)\\
        \ &=\ \frac{1}{2^{2k+2}} \lim_{N\to \infty} \left( \mathbb{E}
        \left[ M_{2k}(B_1, N)^2 \right]-  \mathbb{E} \left[
        M_{2k}(B_1, N) \right]^2 \right)
        \end{split} \nonumber \\
        \ &=\ 0.
\end{align}
To prove \eqref{eq:var_to_zero}, it then suffices to show

\begin{align}\label{eq:first_term_var}
    \begin{split}
        &\lim_{N \to \infty} \frac{1}{(2N)^{2k+2}} \Bigg( \mathbb{E}
        \left[ \left( \sum_{T=1}^k \Tr \left( \prod_{h=1}^p X^{I_h}
        B_1^{J_h} \right) \right)^2 \right] \\
        &\hspace{6cm} - \mathbb{E} \left[  \sum_{T=1}^k \Tr\left(
        \prod_{h=1}^p X^{I_h} B_1^{J_h} \right) \right]^2 \Bigg) \ = \ 0,
    \end{split}
\end{align}

\begin{equation}\label{eq:second_term_var}
    \begin{split}
        &\lim_{N\rightarrow\infty} \frac{1}{(2N)^{2k+2}} \Bigg(
        \mathbb{E}\left[ \Tr \left( X^{2k} \right) \sum_{T=1}^k
        \Tr \left( \prod_{h=1}^p X^{I_h} B_1^{J_h} \right) \right]\\
        &\hspace{4cm}- \mathbb{E} \left[ \Tr \left( X^{2k}\right)
        \right] \mathbb{E} \left[  \sum_{T=1}^k \Tr\left( \prod_{h=1}^p
        X^{I_h} B_1^{J_h} \right) \right] \Bigg) \ = \ 0,
    \end{split}
\end{equation}
and
\begin{equation}\label{eq:third_term_var}
    \begin{split}
        &\lim_{N\rightarrow\infty} \frac{1}{(2N)^{2k+2}} \Bigg( \mathbb{E}
        \left[ \Tr \left( B_1^{2k} \right) \sum_{T=1}^k \Tr \left(
        \prod_{h=1}^p X^{I_h} B_1^{J_h} \right) \right]\\
        &\hspace{4cm} - \mathbb{E} \left[ \Tr \left( B_1^{2k}\right) \right]
        \mathbb{E} \left[  \sum_{T=1}^k \Tr\left( \prod_{h=1}^p X^{I_h}
        B_1^{J_h} \right) \right]\Bigg) \ = \ 0.
    \end{split}
\end{equation}
As before, we denote by $x_{i,j}$, $b_{i,j}$ the entries of $X$, $B_1$,
respectively. Expanding the summands of \eqref{eq:first_term_var} yields
\begin{align}\label{eq:first_term_var_exp1}
    \begin{split}
        &\mathbb{E} \left[ \left( \sum_{T=1}^k \Tr \left( \prod_{h=1}^p
        X^{I_h} B_1^{J_h} \right) \right)^2 \right] \\
        &\hspace{3cm} =\ \sum_{I,J} \sum_{I', J'} \sum_{i} \sum_{j} \mathbb{E}
        \left[ x_{i_1, i_2} \cdots x_{i_{2k}, i_{1}} b_{j_1, j_2} \cdots
        b_{j_{2k}, j_1} \right]
    \end{split}
\end{align}
and
\begin{align}\label{eq:first_term_var_exp2}
    \begin{split}
        &\mathbb{E} \left[  \sum_{T=1}^k \Tr\left( \prod_{h=1}^p X^{I_h}
        B_1^{J_h} \right) \right]^2 \\
        &\hspace{3cm} \ = \ \sum_{I,J} \sum_{i} \mathbb{E}\left[x_{i_1, i_2}
        \cdots x_{i_{2k}, i_{1}}\right] \sum_{I', J'} \sum_{j} \mathbb{E}
        \left[ b_{j_1, j_2} \cdots b_{j_{2k}, j_1} \right].
    \end{split}
\end{align}
Similarly, expanding the summands of \eqref{eq:second_term_var} yields
\begin{align}\label{eq:second_term_var_exp1}
    \begin{split}
        &\mathbb{E}\left[ \Tr \left( X^{2k} \right) \sum_{T=1}^k \Tr
        \left( \prod_{h=1}^p X^{I_h} B_1^{J_h} \right) \right]\\
        &\hspace{3cm} \ = \ \sum_{i} \sum_{I,J} \sum_{j} \mathbb{E}
        \left[x_{i_1,i_2}\dots x_{i_{2k},i_1} b_{j_1,j_2}\cdots
        b_{j_{2k},j_1}\right]
    \end{split}
\end{align}
and
\begin{align}\label{eq:second_term_var_exp2}
    \begin{split}
        &\mathbb{E} \left[ \Tr \left( X^{2k}\right) \right] \mathbb{E}
        \left[  \sum_{T=1}^k \Tr\left( \prod_{h=1}^p X^{I_h} B_1^{J_h}
        \right) \right]\\
        &\hspace{3cm} \ = \ \sum_{i} \mathbb{E}\left[x_{i_1,i_2}\cdots
        x_{i_{2k},i_1}\right] \sum_{I,J} \sum_{j} \mathbb{E}\left[b_{j_1,j_2}
        \cdots b_{j_{2k},j_1}\right].
    \end{split}
\end{align}
The expansion of \eqref{eq:third_term_var} is done similarly. For each pair
of summands, there are two possibilities: if the elements of $X$, $B_1$
indexed by $i$'s are completely disjoint from those indexed by $j$'s, then
these contribute equally to both summands. We are left with estimating the
difference for the crossover cases, where an $i$-indexed element is equal
to a $j$-indexed element. Note there are $2k + 2$ degrees of freedom. The
proof of the analogous result in \cite{HM} is done entirely by counting
degrees of freedom, and showing that at least one degree of freedom is lost
if there is a crossover. The generality of the arguments in \cite{HM} are
immediately applicable here, and yield weak convergence. All that changes
is that in equations \eqref{eq:first_term_var_exp1},
\eqref{eq:first_term_var_exp2}, \eqref{eq:second_term_var_exp1}, and
\eqref{eq:second_term_var_exp2} we now pair entries of two random matrices;
entries from the component matrix $X$ (denoted by $x_{i,j}$) must be paired with
with entries from $X$, and entries from the component matrix $B_1$ (denoted
by $b_{i,j}$) must be paired with entries from $B_1$. Showing that the
contribution from crossovers vanishes in the limit as $N \to \infty$ then
follows from the arguments in \cite{HM}. The claims of \eqref{eq:first_term_var},
\eqref{eq:second_term_var}, and \eqref{eq:third_term_var} then immediately follow.

\begin{thm}\label{thm:weakconv} Let $p$ have mean zero, variance one and
finite higher moments. The measures $\mu_{\mathcal{D}_1,N}(x)$ weakly
converge to the measures $\mu_{X, N}(x) = \mu_{B_1, N}(x)$ of the
component matrices.
\end{thm}

\begin{proof} By Theorem \ref{thm:luntzlara-wang-tweak} the moments
$M_k \left( \mathcal{D}_1 \right) = M_k \left( X \right) = M_k
\left( B_1 \right)$ are finite and appropriately bounded.  As $\mathbb{E}
[M_k( \mathcal{D}_1 ,N)] \to M_k \left( \mathcal{D}_1 \right)$ and the
variances tend to zero, standard arguments give weak convergence. As the
$\mu_{X, N}(x)$ satisfies Carleman's Condition, the $M_k \left( X \right)$
uniquely determine a probability measure and establish weak convergence of
$\mu_{\mathcal{D}_1,N}(x)$ to $\mu_{X, N}(x)$.
\end{proof}

Applying Theorems \ref{thm:luntzlara-wang-tweak}, \ref{thm:weakconv}
and inducting on the parameter $d$ of $\mathcal{B}_d$ then gives the
following corollary.

\begin{cor}\label{lem:curly_B_goto_B_dist}
For $i \in \mathbb{N}$, let $\mathbf{B}=\lbrace B_i \rbrace$ be a
sequence of independent $2^{i-1} N \times 2^{i-1}N$ random real symmetric
matrices chosen from the ensemble $\mathcal{E}_B$, with entries independent
and identically distributed from a fixed distribution function $p(x)$
with mean 0, variance 1, and finite higher moments. Then $M_{k}
\left( \mathcal{B}_1 \right) = M_{k} (B_0)$ for all $k \in \mathbb{N}$,
and the limiting eigenvalue distribution of $\mathcal{B}_d$ as
$N \to \infty$ converges weakly to that of $\mathcal{E}_B$.
\end{cor}

\begin{proof}
The claim follows from repeated applications of Theorems
\ref{thm:luntzlara-wang-tweak}, \ref{thm:weakconv} and induction on the
parameter $d$. In the base case, let $d=1$. Then the limiting eigenvalue
distribution of
\begin{equation}
    \mathcal{B}_1 \ = \ \begin{bmatrix}
    B_0 & B_1\\
    B_1 & B_0
    \end{bmatrix}
\end{equation}
has all moments equal to $M_k (B_0)$ by Theorem
\ref{thm:luntzlara-wang-tweak} and converges weakly by Theorem
\ref{thm:weakconv}. Now assuming that $\mathcal{B}_d$ has the same normalized
eigenvalue distribution as $\mathcal{E}_B$ as $N \to \infty$, it follows
immediately from Theorem \ref{thm:luntzlara-wang-tweak} that $\mathcal{B}_{d+1}$
has the same limiting normalized eigenvalue distribution, since $B_{d+1}$ is
drawn from the same ensemble $\mathcal{E}_B$, completing the proof.
\end{proof}


\section{Convergence of Moments, Distributions of \texorpdfstring{$\mathcal{D}_d$}{Lg}}

Recall in equations \eqref{eq:curlyC_construction} and \eqref{eq:curlyB}
that the disco matrix $\mathcal{D}_d \left( A, \mathbf{B} \right)$
may be linearly decomposed into the sum of two real symmetric matrices
$\mathcal{B}_d$ and $\mathcal{C}_d$. We may then study powers of the
disco matrix $\mathcal{D}_d^k$ as powers of the sum
$(\mathcal{B}_d + \mathcal{C}_d)^k$. We first show that the moments
of the limiting eigenvalue distribution of $\mathcal{C}_d$ are
finite.


\begin{lem}\label{lem:all_moments_C_finite}
Given random $N \times N$ real symmetric matrices $A$ and $B_0$ chosen from ensemble $\mathcal{E}_A$ and $\mathcal{E}_B$, respectively, let
$C = A - B_0$ be a dependent $N \times N$ random matrix.
Then all moments of $C$ are finite.
\end{lem}

\begin{proof}
For $k \in \mathbb{N}$, we observe that
\begin{align}
    M_k (C) \ &=\ \lim_{N \to \infty} \frac{1}{N^{\frac{k}{2}+1}}
    \mathbb{E} \left[ \Tr \left( (A - B_0)^k \right) \right] \nonumber \\
    \begin{split}
         \ &=\ \lim_{N \to \infty} \frac{1}{N^{\frac{k}{2}+1}}
         \Bigg( \mathbb{E}\left[ \Tr \left( A^k \right) \right]\\
         &\qquad + \sum_{m=1}^{k-1}\; \sum_{ \substack{i_1 + \cdots
         + i_p = k-m \\ j_1 + \cdots + j_p = m}} \mathbb{E}\left[
         \Tr\left( (-1)^m \prod_{n=1}^p A^{i_n} B_0^{j_n} \right) \right]  \\
         &\qquad + (-1)^k \mathbb{E}\left[ \Tr \left( B_0^k \right)
         \right] \Bigg).
    \end{split}
\end{align}
As $A$ and $B_0$ are drawn from ensembles each with finite moments,
it follows that
\begin{equation}
    \lim_{N \to \infty} \frac{1}{N^{\frac{k}{2}+1}} \mathbb{E} \left[
    \Tr\left( A^k \right) \right] \ <\ \infty \quad \text{and} \quad
    \lim_{N \to \infty} \frac{1}{N^{\frac{k}{2}+1}} \mathbb{E} \left[
    \Tr\left( B_0^k \right) \right]\ <\ \infty.
\end{equation}
It remains only to show
\begin{equation}
    \lim_{N \to \infty} \frac{1}{N^{\frac{k}{2}+1}} \sum_{m=1}^{k-1}\;
    \sum_{ \substack{i_1 + \cdots + i_p \ = \ k-m \\ j_1 + \cdots
    + j_p = m}} \mathbb{E}\left[ \Tr\left( (-1)^m \prod_{n=1}^p A^{i_n}
    B_0^{j_n} \right) \right]\ <\ \infty.
\end{equation}
For $k$ odd, either $\sum_{k=1}^p i_k= k-m$ or $\sum_{k=1}^p j_k = m$
is odd. It follows that at least one element in the product $\Tr
\left( (-1)^m \prod_{n=1}^p A^{i_n} B^{j_n} \right)$ is a singleton or
paired in at least a triple. In the case of a singleton whose expected
value is 0, the entire product collapses to zero; in the latter case of
a triple or higher pairing, losses to degrees of freedom causes the
product to vanish in the limit as $N \to \infty$.

We now consider $k$ even; applying Theorem
\ref{thm:expected_val_trace_bound} gives
\begin{align}
    \begin{split}
        \sum_{ \substack{i_1 + \cdots + i_p = k-m \\ j_1 + \cdots
        + j_p = m}}& \abs{\mathbb{E}\left[ \Tr\left( (-1)^m \prod_{n=1}^p
        A^{i_n} B_0^{j_n} \right) \right]}\\
        \ &=\ \sum_{ \substack{i_1 + \cdots + i_p = k-m \\ j_1 + \cdots
        + j_p = m}} \abs{\mathbb{E}\left[ \Tr\left( \prod_{n=1}^p A^{i_n}
        B_0^{j_n} \right) \right]}
    \end{split} \nonumber \\
    \ &\leq \ \binom{k}{k-m} \mathbb{E} \left[ \Tr\left( A^k \right)
    \right]^\frac{k-m}{k}   \mathbb{E} \left[ \Tr\left( B_0^k \right)
    \right]^\frac{m}{k}
\end{align}
from which it follows that
\begin{align}
    \begin{split}
        \lim_{N \to \infty}& \frac{1}{N^{\frac{k}{2}+1}} \sum_{m=1}^{k-1}
        \; \sum_{ \substack{i_1 + \cdots + i_p \ = \ k-m \\ j_1 + \cdots
        + j_p = m}} \mathbb{E}\left[ \Tr\left( (-1)^m \prod_{n=1}^p A^{i_n}
        B_0^{j_n} \right) \right] \\
        \ &\leq\ \lim_{N \to \infty} \frac{1}{N^{\frac{k}{2}+1}}
        \sum_{m=1}^{k-1} \binom{k}{k-m}  \mathbb{E} \left[ \Tr\left(
        A^k \right) \right]^\frac{k-m}{k}   \mathbb{E} \left[ \Tr\left(
        B_0^k \right) \right]^\frac{m}{k}
    \end{split} \nonumber \\
    \ &=\ \sum_{m=1}^{k-1} \binom{k}{k-m} \left( \lim_{N \to \infty}
    \left( \frac{\mathbb{E} \left[ \Tr\left( A^k \right) \right]}{N^{\frac{k}{2}+1}}
    \right)^\frac{k-m}{k}  \left( \frac{\mathbb{E} \left[ \Tr\left( B_0^k \right)
    \right]}{N^{\frac{k}{2}+1}} \right)^\frac{m}{k}  \right) \nonumber \\
    \ &=\ \sum_{m=1}^{k-1} \binom{k}{k-m} M_k \left( A^k \right)^\frac{k-m}{k}
    M_k \left( B_0^k \right)^\frac{m}{k} \nonumber \\
    \ &<\ \infty \label{eq:upper_bound_mixed_T_terms}
\end{align}
completing the proof for even $k$.
\end{proof}

As the moments of the limiting eigenvalue distribution of $\mathcal{B}_d$
are finite by assumption, the preceding lemma shows that for finite $d$,
the moments of $\mathcal{D}_d = \mathcal{B}_d + \mathcal{C}_d$ are finite.
We now prove the cornerstone result concerning $d$-disco matrices
$\mathcal{D}_d (A, \mathbf{B})$. We first show that the moments of
$\mathcal{D}_d$ converge geometrically to the moments of the $\mathbf{B}$
ensemble as $d \to \infty$.


\begin{thm}\label{thm:inf_disco_convergence_of_moments}
The moments of $\mathcal{D}_d (A, \mathbf{B} )$ converges with order
$O\left(2^{-d} \right)$ to the moments of the limiting
normalized eigenvalue distribution of $\mathcal{E}_B$ as $N \to \infty$
and $d \to \infty$.
\end{thm}

\begin{proof}
As in \eqref{eq:curlyB}, we consider the moments of
$\mathcal{D}_d (A, \mathbf{B} ) = \mathcal{B}_d + \mathcal{C}_d$.
Suppressing arguments, for $k \in \mathbb{N}$ we observe that
\begin{align}
    &\lim_{d \to \infty} M_k \left( \mathcal{D}_d^k \right) \nonumber \\
    \ &=\ \lim_{d \to \infty} \lim_{N \to \infty} \frac{1}{\left( 2^d N
    \right)^{\frac{k}{2}+1}} \mathbb{E} \left[ \Tr\left( \mathcal{D}_d^k
    \right) \right] \nonumber \\
    \ &=\ \lim_{d \to \infty} \lim_{N \to \infty} \frac{1}{\left( 2^d N
    \right)^{\frac{k}{2}+1}}  \mathbb{E} \left[ \Tr\left( (\mathcal{B}_d
    + \mathcal{C}_d)^k \right) \right] \nonumber \\
    \begin{split}
        \ &=\ \lim_{d \to \infty} \lim_{N \to \infty} \frac{1}{\left( 2^d
        N \right)^{\frac{k}{2}+1}}  \left( \mathbb{E} \left[ \Tr\left(
        \mathcal{B}_d^k \right) \right] \vphantom{\mathbb{E} \left[
        \Tr\left( \sum_{m=1}^{k-1}\; \sum_{ \substack{i_1 + \cdots + i_p
        = k-m \\ j_1 + \cdots + j_p = m}}\left(  \prod_{n=1}^p
        \mathcal{B}_d^{i_n} \mathcal{C}_d^{j_n} \right)   \right) \right]}
        \right. \\
        &\qquad  \left.  + \mathbb{E} \left[ \Tr\left( \sum_{m=1}^{k-1}\;
        \sum_{ \substack{i_1 + \cdots + i_p = k-m \\ j_1 + \cdots + j_p = m}}
        \left(  \prod_{n=1}^p \mathcal{B}_d^{i_n} \mathcal{C}_d^{j_n} \right)
        \right) \right] + \mathbb{E} \left[ \Tr \left( \mathcal{C}_d^k \right)
        \right] \vphantom{\mathbb{E} \left[ \Tr\left( \sum_{m=1}^{k-1}\;
        \sum_{ \substack{i_1 + \cdots + i_p = k-m \\ j_1 + \cdots + j_p = m}}
        \left(  \prod_{n=1}^p \mathcal{D}_d^{i_n} \mathcal{H}_d^{j_n} \right)
        \right) \right]} \right).
    \end{split}\label{eq:trace_exp_infinite_disco_d,n_to_infty}
\end{align}

We consider each summand separately, first noting
\begin{align}
    \lim_{d \to \infty} \lim_{N \to \infty} \frac{\mathbb{E} \left[
    \Tr\left( \mathcal{C}_d^k \right) \right]}{\left( 2^d N \right)^{\frac{k}{2}+1}}
    \ &=\ \lim_{d \to \infty} \lim_{N \to \infty} \frac{ 2^d \mathbb{E} \left[
    \Tr\left( C^k \right) \right]}{2^\frac{dk+2d}{2} N^{\frac{k}{2}+1}}
    \label{eq:bound_on_pollution_matrix} \\
    \ &=\ \lim_{d \to \infty} \left(2^{-\frac{dk}{2}} \right) \lim_{N \to \infty}
    \frac{ \mathbb{E} \left[ \Tr\left( C^k \right) \right]}{ N^{\frac{k}{2}+1}} \nonumber \\
    \ &=\ \lim_{d \to \infty} \left(2^{-\frac{dk}{2}} \right) M_k \left( C \right)
    \label{eq:bounded_C_mom} \\
    \ &=\ 0 . \nonumber
\end{align}
It follows immediately from \eqref{eq:bounded_C_mom} that
\eqref{eq:bound_on_pollution_matrix} goes to zero with order $O(2^{-d})$.

We then consider the term
\begin{align}\label{eq:statement_of_second_term}
    \lim_{d \to \infty} \lim_{N \to \infty} \frac{1}{\left( 2^d N \right)
    ^{\frac{k}{2}+1}} \mathbb{E} \left[ \Tr\left( \sum_{m=1}^{k-1}\;
    \sum_{ \substack{i_1 + \cdots + i_p = k-m \\ j_1 + \cdots + j_p = m}}
    \left(  \prod_{n=1}^p \mathcal{B}_d^{i_n} \mathcal{C}_d^{j_n} \right)
    \right) \right].
\end{align}
Recalling Theorem \ref{thm:expected_val_trace_bound}, we have
\begin{align}
\begin{split}
    \mathbb{E}& \left[ \Tr\left( \sum_{m=1}^{k-1}\; \sum_{
    \substack{i_1 + \cdots + i_p = k-m \\ j_1 + \cdots + j_p = m}}
    \left(  \prod_{n=1}^p \mathcal{B}_d^{i_n} \mathcal{C}_d^{j_n}
    \right)   \right) \right]\\
    &\qquad\qquad \leq\ \sum_{m=1}^{k-1} \binom{k}{k-m}  \mathbb{E}
    \left[ \Tr\left( \mathcal{B}_d^k \right) \right]^\frac{k-m}{k}
    \mathbb{E} \left[ \Tr\left( \mathcal{C}_d^k \right) \right]^\frac{m}{k}
\end{split} \nonumber \\
&\qquad\qquad \ = \ \sum_{m=1}^{k-1} \binom{k}{k-m} 2^{\frac{dm}{k}}
\mathbb{E} \left[ \Tr\left( \mathcal{B}_d^k \right) \right]^\frac{k-m}{k}
\mathbb{E} \left[ \Tr\left( C^k \right) \right]^\frac{m}{k}.
\end{align}
Applying this inequality to \eqref{eq:statement_of_second_term} yields
\begin{align}
    \begin{split}
        &\lim_{d \to \infty} \lim_{N \to \infty} \frac{1}{\left( 2^d N
        \right)^{\frac{k}{2}+1}} \mathbb{E} \left[ \Tr\left(
        \sum_{m=1}^{k-1}\; \sum_{ \substack{i_1 + \cdots + i_p = k-m \\
        j_1 + \cdots + j_p = m}}\left(  \prod_{n=1}^p \mathcal{B}_d^{i_n}
        \mathcal{C}_d^{j_n} \right)   \right) \right]\\
        \ &\leq\ \lim_{d \to \infty} \lim_{N \to \infty} \frac{1}{\left(
        2^d N \right)^{\frac{k}{2}+1}} \sum_{m=1}^{k-1} \binom{k}{k-m}
        2^\frac{dm}{k} \mathbb{E} \left[ \Tr\left( \mathcal{B}_d^k \right)
        \right]^\frac{k-m}{k}   \mathbb{E} \left[ \Tr\left( C^k \right)
        \right]^\frac{m}{k}
    \end{split} \nonumber \\
    \ &=\ \lim_{d \to \infty}   \sum_{m=1}^{k-1} \binom{k}{k-m}
    2^{-\frac{dm}{2}} \lim_{N \to \infty} \left( \frac{\mathbb{E}
    \left[ \Tr\left( \mathcal{B}^k \right) \right]}{ (2^d N)^{\frac{k}{2}+1} }
    \right)^\frac{k-m}{k}  \left( \frac{ \mathbb{E} \left[ \Tr\left( C^k \right)
    \right]}{ N^{\frac{k}{2}+1} } \right)^\frac{m}{k} \nonumber  \\
    \ &=\ \lim_{d \to \infty}   \sum_{m=1}^{k-1} \binom{k}{k-m} 2^{-\frac{dm}{2}}
    M_k(B_i)^\frac{k-m}{k}  M_{k}(C)^\frac{m}{k} \nonumber \\
    \ &\leq\ \lim_{d \to \infty}  2^{-\frac{d}{2}} \sum_{m=1}^{k-1} \binom{k}{k-m}
    M_k(B_i)^\frac{k-m}{k}  M_{k}(C)^\frac{m}{k}.
\end{align}
As $M_k(\mathcal{B}_d)$ and $M_k (C)$ are both finite, there exists some constant
$L_k$ such that
\begin{equation}
    L_k \geq \sum_{m=1}^{k-1} \binom{k}{k-m} M_k (B_i)^\frac{k-m}{k}
    M_k (C)^\frac{m}{k}.
\end{equation}
We then observe
\begin{align}
    \begin{split}\label{eq:bound_on_second_summand}
        \lim_{d \to \infty} \lim_{N \to \infty} \frac{1}{\left( 2^d N
        \right)^{\frac{k}{2}+1}} & \mathbb{E} \left[ \Tr\left( \sum_{m=1}^{k-1}\;
        \sum_{ \substack{i_1 + \cdots + i_p = k-m \\ j_1 + \cdots + j_p = m}}
        \left(  \prod_{n=1}^p \mathcal{B}_d^{i_n} \mathcal{C}_d^{j_n} \right)
        \right) \right] \\
        \ &\leq\  \lim_{d \to \infty}  2^{-\frac{d}{2}} L_k
    \end{split} \\
    \ &=\ 0 \nonumber
\end{align}
from which it follows that \eqref{eq:bound_on_second_summand} goes to zero with
order $O\left( 2^{-d} \right)$ as $d \to \infty$. Finally, we consider
the first summand and note
\begin{align}
    \lim_{d \to \infty} \lim_{N \to \infty} \frac{\mathbb{E} \left[ \Tr \left(
    \mathcal{B}_d^k \right) \right]}{\left( 2^d N \right)^{\frac{k}{2}+1}}
    \ = \ M_k ( B_i ) \label{eq:third_summand_to_dist}
\end{align}
so that, substituting \eqref{eq:bounded_C_mom}, \eqref{eq:bound_on_second_summand},
and \eqref{eq:third_summand_to_dist} into
\eqref{eq:trace_exp_infinite_disco_d,n_to_infty} yields the desired result
\begin{equation}\label{d-disco moment converge}
    \lim_{d\to\infty}M_k\left( \mathcal{D}_d \right) \ = \ M_k \left( B_i \right).
\end{equation}

\end{proof}

Finally, we show that for all finite $d \in \mathbb{Z}^+$, the limiting distribution of
$\mathcal{D}_d (A, \mathbf{B})$ converges to a new, universal distribution; as
$d \to \infty$, the limiting distribution of $\mathcal{D}_d (A, \mathbf{B})$
converges to the distribution of the $\mathbf{B}$ ensemble.

\begin{thm}\label{thm:weakCon_infty_disco_to_curly_B}
For finite $d \in \mathbb{Z}^+$, the limiting spectral distribution of
$\mathcal{D}_d(A, \mathbf{B})$ whose independent entries are independently
chosen from a probability distribution $p$ with mean $0$, variance $1$ and
finite higher moments converges weakly to a new, universal distribution
independent of $p$. As $d \to \infty$, the limiting spectral distribution
of $\mathcal{D}_d(A, \mathbf{B})$ converges weakly to that of the $\mathbf{B}$
ensemble.
\begin{proof}
We first consider finite $d < \infty$. As in Theorem \ref{thm:weakconv},
it suffices to show that
\begin{equation}\label{conv_for_finite_d}
    \lim_{N \to \infty} \left( \mathbb{E} \left[ M_{m} \left( \mathcal{D}_d, A, \mathbf{B}, N \right)^2 \right] - \mathbb{E} \left[ M_{m} \left( \mathcal{D}_d, A, \mathbf{B}, N \right) \right]^2 \right) = 0.
\end{equation}
We decompose $\mathcal{D}_d$ as the sum $\mathcal{B}_d + \mathcal{C}_d$ and
invoke the Eigenvalue Trace Lemma to observe
\begin{align}
    \mathbb{E} &\left[ M_{2k} \left( \mathcal{D}_d, A, \mathbf{B}, N \right)^2 \right] \nonumber\\
    &\ = \  \frac{1}{(2^d N)^{ 2k + 2 }} \mathbb{E} \left[ \left(\Tr \left( \mathcal{D}_d^{2k} \right) \right)^2 \right] \nonumber \\
    &\ = \  \frac{1}{(2^d N)^{ 2k + 2 }} \mathbb{E} \left[ \Tr \left( (\mathcal{B}_d + \mathcal{C}_d)^{2k} \right)  \right] \nonumber \\
    &\ = \  \frac{1}{(2^d N)^{ 2k + 2 }} \mathbb{E} \left[ \left(\Tr \left( \mathcal{B}_d^{2k} + \sum_{T=1}^{2k-1} \prod_{h=1}^p \mathcal{B}_d^{I_h} \mathcal{C}_d^{J_h} + \mathcal{C}_d^{2k} \right) \right)^2 \right].
\end{align}
Applying the linearity of trace and expanding yields
\begin{align}
    \begin{split}
    &\mathbb{E} \left[ M_{2k} \left( \mathcal{D}_d, A, \mathbf{B}, N \right)^2 \right]\\
    &\ = \  \frac{1}{(2N)^{ 2k + 2 }} \Bigg( \mathbb{E}\left[ \Tr \left( \mathcal{B}_d^{2k} \right)^2 \right] + \mathbb{E}\left[ \Tr \left( \mathcal{C}_d^{2k} \right)^2 \right] + 2 \mathbb{E}\left[ \Tr \left( \mathcal{B}_d^{2k} \right) \right] \mathbb{E} \left[ \Tr \left( \mathcal{C}_d^{2k} \right) \right]\\
    &\qquad + \mathbb{E}\left[ \left( \sum_{T=1}^k \Tr \left( \prod_{h=1}^p \mathcal{B}_d^{I_h} \mathcal{C}_d^{J_h} \right) \right)^2 \right] + 2 \mathbb{E}\left[ \Tr \left( \mathcal{B}_d^{2k} \right) \sum_{T=1}^k \Tr \left( \prod_{h=1}^p \mathcal{B}_d^{I_h} \mathcal{C}_d^{J_h} \right) \right] \\
    &\hspace{2cm} + 2 \mathbb{E}\left[ \Tr \left( \mathcal{C}_d^{2k} \right) \sum_{T=1}^k \Tr \left( \prod_{h=1}^p \mathcal{B}_d^{I_h} \mathcal{C}_d^{J_h} \right) \right] \Bigg).
        \end{split}
\end{align}
Similarly,
\begin{align}
    \begin{split}
    &\mathbb{E} \left[ M_{2k} \left( \mathcal{D}_d, A, \mathbf{B}, N \right) \right]^2\\
    &\ = \  \frac{1}{(2N)^{ 2k + 2 }} \Bigg( \mathbb{E}\left[ \Tr \left( \mathcal{B}_d^{2k} \right) \right]^2 + \mathbb{E}\left[ \Tr \left( \mathcal{C}_d^{2k} \right) \right]^2 + 2 \mathbb{E}\left[ \Tr \left( \mathcal{B}_d^{2k} \right) \right] \mathbb{E} \left[ \Tr \left( \mathcal{C}_d^{2k} \right) \right]\\
    &\qquad + \mathbb{E}\left[ \sum_{T=1}^k \Tr \left( \prod_{h=1}^p \mathcal{B}_d^{I_h} \mathcal{C}_d^{J_h} \right) \right]^2 + 2 \mathbb{E}\left[ \Tr \left( \mathcal{B}_d^{2k} \right) \right] \mathbb{E} \left[ \sum_{T=1}^k \Tr \left( \prod_{h=1}^p \mathcal{B}_d^{I_h} \mathcal{C}_d^{J_h} \right) \right] \\
    &\hspace{2cm} + 2 \mathbb{E}\left[ \Tr \left( \mathcal{C}_d^{2k} \right) \right] \mathbb{E} \left[ \sum_{T=1}^k \Tr \left( \prod_{h=1}^p \mathcal{B}_d^{I_h} \mathcal{C}_d^{J_h} \right) \right] \Bigg).
    \end{split}
\end{align}
Recalling that
\begin{align}
    \lim_{N\to \infty} \frac{1}{(2^d N)^{2k+2}} & \left( \mathbb{E}
        \left[ \Tr \left( \mathcal{B}_d^{2k} \right)^2 \right]-  \mathbb{E}
        \left[ \Tr \left( \mathcal{B}_d^{2k} \right) \right]^2 \right)=0
\end{align}
and
\begin{align}
    \lim_{N\to \infty} \frac{1}{(2^d N)^{2k+2}} & \left( \mathbb{E}
        \left[ \Tr \left( \mathcal{C}_d^{2k} \right)^2 \right]-  \mathbb{E}
        \left[ \Tr \left( \mathcal{C}_d^{2k} \right) \right]^2 \right)=0
\end{align}
the results then follow by arguments analogous to those of Theorem \ref{thm:weakconv}.

For $d\to\infty$, Theorem \ref{thm:weakconv} also applies. It then
    suffices to show that
    \begin{equation}
        \lim_{d\to\infty}\lim_{N\to\infty}\left(\mathbb{E}\left[M_{m}
        \left( \mathcal{D}_d,N \right)^2\right]-\mathbb{E}\left[M_{m}
        \left( \mathcal{D}_d,N \right)\right]^2\right) \ = \ 0.
    \end{equation}
    By \eqref{conv_for_finite_d}, for any finite $d$, the inner limit
    equals zero. It follows trivially that the outer limit is also zero
    as we are just taking the limit of a constant sequence.

\end{proof}
\end{thm}

\section{Future Work}\label{sec:futurework}

So far we have investigated the density of the eigenvalues; we now
consider another problem, that of the spacings between adjacent
eigenvalues of $\mathcal{D}_d \left( A, \mathbf{B} \right)$. For
concreteness we restrict our purview to $A$ a symmetric palindromic
Toeplitz matrix and $\mathbf{B}$ a sequence of real symmetric
matrices. In this case $\mathcal{D}_d$ has only
\begin{equation}
    \frac{N}{2} + \sum_{i=1}^d \frac{2^{i-1} N (2^{i-1} N + 1)}{2}
\end{equation}
degrees of freedom, which is much smaller than $2^d N (2^d N + 1)/2$,
it is reasonable to believe the spacings between adjacent normalized
eigenvalues
$(\gl_{i+1}(\mathcal{D}_d) - \gl_i(\mathcal{D}_d)) / \sqrt{N}$
may differ from those of full real symmetric matrices.

\begin{figure}[ht]
    \centering
    \includegraphics[width=0.6\textwidth]{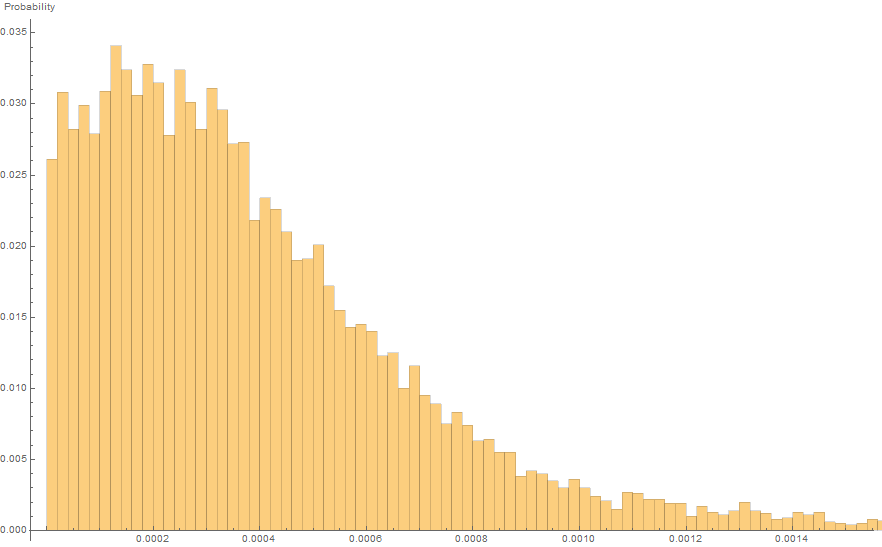}
    \caption{Eigenvalue gaps of a $20,000 \times 20,000$ matrix $\mathcal{D}_1 (A, B)$, with $A$ a random symmetric palindromic Toeplitz matrix and $B$ a random real symmetric matrix.}
    \label{fig:gaps}
\end{figure}

In \cite{MMS} it is conjectured that the limiting eigenvalue gap
distribution of symmetric palindromic Toeplitz matrices
is Poissonian, while the ensemble
of all real symmetric matrices is conjectured to have normalized
spacings given by the GOE distribution whenever the independent
matrix elements are independently chosen from a nice distribution
$p$. As $\mathcal{D}_1 (A, B)$ exhibits eigenvalue behavior
representing a hybrid of its component behaviors, we similarly
conjecture that the limiting eigenvalue gap distribution of
$\mathcal{D}_1 (A, B)$ is bounded by that of its submatrices
$A$ and $B$.

Additionally, numerical experiments in constructing $\mathcal{D}_1(A,B)$
with $A$, $B$ drawn from several pairs of random ensembles
whose limiting eigenvalue distributions are known suggests the following
conjecture.
\begin{conj}\label{conj:LW}
Let $A$, $B$ be $N \times N$ random matrices with independent entries i.i.d.
from a fixed probability distribution $p$ with mean 0 and variance 1. Suppose
that the limiting eigenvalue distributions of $A$, $B$ have all moments
finite and appropriately bounded. Then
\begin{equation}\label{eq:conj}
    \min\left\lbrace M_k(A), M_k(B) \right\rbrace \ \leq\ M_k\left(
    \mathcal{D}_1(A,B) \right) \leq \max\left\lbrace M_k(A), M_k(B) \right\rbrace.
\end{equation}
\end{conj}

Tables \ref{tab:conj_data1}, \ref{tab:conj_data2} show experimental results
supporting Conjecture \ref{conj:LW}. We compute small moments of $\mathcal{D}_1 (A, B)$
where $A$ is either a random real symmetric matrix or a random symmetric
palindromic Toeplitz matrix, with entries i.i.d.r.v. from a normal distribution
with mean $0$ and variance $1$. $B$ is constructed as a random $3$-period
block circulant matrix (see \cite{KMMSX} for a full treatment of the construction of block-
circulant matrices and their limiting eigenvalue distributions).

\begin{center}
\begin{table}[ht]
\begin{tabular}{|c|c|c|c|}\hline
Moment & $M_k(A)$ & \hspace{0.5cm} $\mathcal{D}_1(A,B)$ \hspace{0.5cm}
& $M_k(B)$ \\ \hline
4 & 2.000  & 2.071    & 2.183  \\ \hline
6 & 4.997  & 5.363    & 6.257  \\ \hline
8 & 13.985 & 15.759 & 21.974 \\ \hline
\end{tabular}
\caption{\label{tab:conj_data1} Numerical data from a $11994\times 11994$ Disco of a random RS $A$ and a random $3$-period block circulant matrix $B$ supporting \eqref{conj:LW}.}
\end{table}

\begin{table}[ht]
\begin{tabular}{|c|c|c|c|}\hline
Moment & $M_k(A)$ & \hspace{0.5cm} $\mathcal{D}_1(A,B)$ \hspace{0.5cm}
& $M_k(B)$ \\ \hline
4 & 2.948  & 2.544   & 2.330  \\ \hline
6 & 14.863  & 9.783    & 7.929  \\ \hline
8 & 102.518 & 50.681 & 36.884 \\ \hline
\end{tabular}
\caption{\label{tab:conj_data2} Numerical data from a $11994\times 11994$ Disco of a random PST $A$ and a random $3$-period block circulant matrix $B$ supporting \eqref{conj:LW}.}
\end{table}
\end{center}

The computation of $M_k \left( \mathcal{D}_1(A,B) \right)$ involves the
expansion previously shown in \eqref{eq1.1}. The primary obstacle
is bounding the contribution of arbitrary bi-variate matrix products in
the limit as $N \to \infty$. While the result of Theorem
\ref{thm:expected_val_trace_bound} gives one such bound, it is not
sufficiently sharp to establish the stated conjecture. A central challenge
in crafting sharp bounds on the contribution of such terms is, for arbitrary
$A$ and $B$, the lack of information on the pairing configurations of
entries that do not vanish in the limit.

By Eigenvalue Trace Lemma and \eqref{eq:1-disco-diagonalized}, \eqref{eq:conj} can be rewritten as 
\begin{align}\label{eq:conj_1}
    \begin{split}
    & \min\left\lbrace \lim_{N\to\infty}\mathbb{E}\left[\Tr(A^k)\right], \lim_{N\to\infty}\mathbb{E}\left[\Tr(B^k)\right] \right\rbrace \\
    &\qquad \ \leq\ \mathbb{E}\left[\lim_{N\to\infty}\frac{1}{2^{\frac{k}{2}+1}}\Tr\left((A+B)^k+(A-B)^k\right)\right]\\ 
    &\qquad \ \leq\ \max\left\lbrace \lim_{N\to\infty}\mathbb{E}\left[\Tr(A^k)\right], \lim_{N\to\infty}\mathbb{E}\left[\Tr(B^k)\right] \right\rbrace.
    \end{split}
\end{align}

Note that \eqref{eq:conj_1} would follow immediately if, for all $N\times N$ real symmetric matrices $A$ and $B$, it were true that 
\begin{align}\label{eq:conj_1_equiv}
    \min\left\lbrace \Tr(A^k), \Tr(B^k) \right\rbrace & \ \leq\ \frac{1}{2^{\frac{k}{2}+1}} \Tr\left((A+B)^k+(A-B)^k\right) \nonumber \\
    & \ \leq\ \max\left\lbrace \Tr(A^k), \Tr(B^k) \right\rbrace.
\end{align}
Unfortunately this is not the case, as evidenced by the following construction\footnote{This construction was suggested by Zhijie Chen, Jiyoung Kim, and Samuel Murray of Carnegie Mellon University. }. Let
\begin{equation}
    \mathbf{a} = \left[\begin{array}{rr}
        -33 & -31 \\
        -31 & -82
    \end{array} \right] \qquad\text{and}\qquad  \mathbf{b} = \left[\begin{array}{rr}
        26 & 78 \\
        78 & -15
    \end{array} \right] 
\end{equation}
so that, for any $m \in \mathbb{Z}^+$,
\begin{equation}
    A_{2m \times 2m} = \left[ \begin{array}{ccc}
        \mathbf{a} & & \\
         & \ddots &\\
         & & \mathbf{a}
    \end{array}\right] \qquad\text{and}\qquad B_{2m \times 2m} = \left[ \begin{array}{ccc}
        \mathbf{b} & & \\
         & \ddots &\\
         & & \mathbf{b}
    \end{array}\right]
\end{equation}
are matrices of equal dimension $2m \times 2m$ with $m$ instances of $\mathbf{a}$ and $\mathbf{b}$ along their main diagonals, respectively. For $m = 10$ and $k = 4$, we compute
\begin{align}
\Tr \left(A_{20 \times 20}^4 \right) &= 889,801,750 \nonumber\\
\Tr \left(A_{20 \times 20}^4 \right) &= 869,734,090 \nonumber \\ 
\dfrac{\Tr\left((A_{20 \times 20} + B_{20 \times 20})^4+(A_{20 \times 20}-B_{20 \times 20})^4\right)}{2^3} &= 1,336,343,790
\end{align}
which is clearly at odds with \eqref{eq:conj_1_equiv}.

\section{Appendix}
\subsection{Proof of Generalized H\"older's Inequality}
\begin{lem}{(Eigenvalue Trace Formula)}
Let $X$ be a $N \times N$ Hermitian matrix, denote $\lambda_1, \cdots, \lambda_N$ the eigenvalues of $X$, then for $k \in \mathbb{N}$
\begin{equation}
    \Tr(X^k) = \sum_{i=1}^N \lambda_i^k.
\end{equation}
\end{lem}

\begin{lem}{(von Neumann's Trace Inequality \cite{VN})}
For any $N \times N$ complex matrices $A,\; B$ with singular values $\sigma_1(A) \geq \cdots \geq \sigma_N(A)$ and $\sigma_1(B) \geq \cdots \geq \sigma_N(B)$
\begin{equation}
    |\Tr(AB)| \leq \sum_{i=1}^N \sigma_i(A)\sigma_i(B).
\end{equation}
\end{lem}

\begin{lem}{(Jensen's Inequality \cite{Jen})}
Let $\chi$ be a random variable, $\phi$ a convex function, then
\begin{equation}
    \phi(\mathbb{E}[\chi]) \leq \mathbb{E}[\phi(\chi)].
\end{equation}
\end{lem}

\begin{lem}\label{lem:SingularValue_appendix}
Let $X_1,X_2,\dots,X_k$ be $N\times N$ matrices. For  $j = 1, \ldots, k$, let
the singular values of $X_j$ be ordered by
\begin{equation}
    \sigma_1(X_j)\geq\sigma_2(X_j)\geq\cdots\geq \sigma_N(X_j) \geq 0.
\end{equation}
Similarly, for $\mathbf{X}_k = \prod_{j=1}^k X_j$ we let the singular values
of $\mathbf{X}_k$ be ordered by
\begin{equation}
    \sigma_1( \mathbf{X}_k ) \geq \sigma_2( \mathbf{X}_k ) \geq \cdots \geq
    \sigma_N( \mathbf{X}_k ) \geq 0.
\end{equation}
Then for $p>0$,
\begin{equation}\label{eq:sing_val_bound_appendix}
    \sum_{i=1}^N\sigma_i^p( \mathbf{X}_k )\leq\sum_{i=1}^N\prod_{j=1}^k
    \sigma_i^p(X_j).
\end{equation}

\begin{proof}
We proceed by induction on $k$; from Theorem 3.3.14 of \cite{HJ} we have for
arbitrary $N \times N$ matrices $X_1$, $X_2$ that
\begin{equation}
    \sum_{i=1}^N \sigma_i^p(X_1 X_2) = \sum_{i=1}^N \sigma_i^p (\mathbf{X}_2)
    \ \leq\ \sum_{i=1}^N \prod_{j=1}^2 \sigma_i^p (X_j),
\end{equation}
proving the base case $k=2$. Now assume \ref{eq:sing_val_bound_appendix} for $k=n$
and consider the case $k=n+1$. Observe
\begin{align}
    \sum_{i=1}^N \sigma_i^p (\mathbf{X}_{n+1}) \ &=\ \sum_{i=1}^N \sigma_i^p
    \left( X_1 \cdots X_{n-1} Y_n \right)
\end{align}
where $Y_n = X_n X_{n+1}$. By the inductive hypothesis, we have
\begin{align}\label{eq:inductive_step_ineq1}
    \sum_{i=1}^N \sigma_i^p \left( X_1 \cdots X_{n-1} Y_n \right) \ &\leq \
    \sum_{i=1}^N  \left(\prod_{j=1}^{n-1}\sigma_i^p (X_j) \right)
    \sigma_i^p(Y_n) \nonumber \\
    \ &=\ \sum_{i=1}^N  \left(\prod_{j=1}^{n-1}\sigma_i^p (X_j) \right)
    \sigma_i^p(X_n X_{n+1}).
\end{align}
We define
\begin{equation}
    \alpha_i \ = \ \left( \prod\limits_{j=1}^{n-1} \sigma_i^p (X_j) \right)
    \sigma_i^p(X_n X_{n+1}) \qquad \text{and}\qquad \beta_i \ = \
    \prod\limits_{j=1}^{n+1} \sigma_i^p (X_j).
\end{equation}
Since $\sigma_1(X_j) \geq \sigma_2(X_j) \geq\cdots \geq \sigma_N (X_j) \geq 0$,
it follows $\alpha_i \geq \alpha_2 \geq \cdots \geq \alpha_N \geq 0$ and
$\beta_1 \geq \beta_2 \geq \cdots \geq \beta_N \geq 0$. By Theorem 3.3.4 from
\cite{HJ} we know
\begin{equation}
    \prod_{i=1}^N \sigma_i^p(X_n X_{n+1}) \ \leq\ \prod_{i=1}^N \sigma_i^p(X_n)
    \sigma_i^p(X_{n+1})
\end{equation}
which implies
\begin{equation}\label{eq:alpha_beta_ineq}
\prod_{i=1}^N \sigma_i^p (X_n X_{n+1}) \prod_{i=1}^N \prod_{j=1}^{n-1}
\sigma_i^p(X_j) \ \leq \ \prod_{i=1}^N \sigma_i^p(X_n) \sigma_i^p(X_{n+1})
\prod_{i=1}^N \prod_{j=1}^{n-1} \sigma_i^p (X_j).
\end{equation}
Since
\begin{equation}
    \prod_{i=1}^N \sigma_i^p (X_n X_{n+1}) \prod_{i=1}^N \prod_{j=1}^{n-1}
    \sigma_i^p (X_j) \ = \ \prod_{i=1}^N \left( \sigma_i^p (X_n X_{n+1})
    \prod_{j=1}^{n-1}\sigma_i^p(X_j) \right)
\end{equation}
and
\begin{align}
    \prod_{i=1}^N \sigma_i^p(X_n) \sigma_i^p(X_{n+1}) \prod_{i=1}^N
    \prod_{j=1}^{n-1} \sigma_i^p(X_j) &= \prod_{i=1}^N \left(
    \sigma_i^p(X_n) \sigma_i^p(X_{n+1}) \prod_{j=1}^{n-1}\sigma_i^p(X_j)
    \right) \nonumber \\
    \ &=\ \prod_{i=1}^N\prod_{j=1}^{n+1}\sigma_i^p(X_j)
\end{align}
it follows that \eqref{eq:alpha_beta_ineq} becomes
\begin{align}
    \prod_{i=1}^N \left( \sigma_i^p(X_n X_{n+1})\prod_{j=1}^{n-1}
    \sigma_i^p(X_j) \right) \ &\leq\ \prod_{i=1}^N\prod_{j=1}^{n+1}
    \sigma_i^p(X_j) \nonumber \\
    \prod_{i=1}^N\alpha_i \ &\leq \ \prod_{i=1}^N\beta_i
\end{align}
which, by Corollary 3.3.10 of \cite{HJ}, implies
\begin{align}
\sum_{i=1}^N\alpha_i &\leq \sum_{i=1}^N\beta_i.
\end{align}
It then follows that
\begin{align}\label{eq:inductive_step_ineq2}
\sum_{i=1}^N \left(\sigma_i^p(X_k X_{k+1})\prod_{j=1}^{n-1}\sigma_i^p(X_j)
\right) &\leq \sum_{i=1}^N \left( \prod_{j=1}^{n+1}\sigma_i^p(X_j)\right).
\end{align}
Combining \eqref{eq:inductive_step_ineq1} and \eqref{eq:inductive_step_ineq2}
yields
\begin{equation}
     \sum_{i=1}^N \sigma_i^p (X_1 X_2\cdots X_k) \ \leq \ \sum_{i=1}^N
     \prod_{j=1}^k\sigma_i^p(X_j)
\end{equation}
as desired.
\end{proof}
\end{lem}

\begin{lem}{(H\"older's Trace Inequality.)}\label{lem:1-holderTrace_appendix}
Let $X$, $Y$ be $N \times N$ matrices, and $p,q>0$ be real numbers such that
$p^{-1} + q^{-1} = 1$. Then
\begin{equation}
    \abs{\Tr(XY)} \ \leq\ \norm{X}_p \norm{Y}_q .
\end{equation}

\begin{proof}
Let $\Vec{S}(X) = \left[ \sigma_1(X), \sigma_2(X), \ldots, \sigma_N(X) \right]$;
it follows immediately from Definition \ref{def:p-Schatten} that $\norm{X}_p =
\norm{\Vec{S}(X)}_p$. By Von Neumann's Trace Inequality \cite{VN}, we then have
\begin{equation}
    \abs{\Tr (XY)} \ \leq\ \left\langle \Vec{S}(X),\Vec{S}(Y) \right\rangle.
\end{equation}
A straightforward application of H\"older's Inequality yields the desired result:
\begin{equation}
    \left\langle \Vec{S}(X), \Vec{S}(Y) \right\rangle \ \leq\ \norm{\Vec{S}(X)}_p
    \norm{\Vec{S}(Y)}_q = \norm{X}_p \norm{Y}_q.
\end{equation}
\end{proof}
\end{lem}

\begin{thm}{(Generalized H\"older's Trace Inequality.)}
\label{thm:GeneralHolderTrace_appendix}
Let $X_1,X_2,\dots,X_k$ be $N\times N$ matrices and $p_1,p_2,\dots,p_k \in \mathbb{R}^+$
such that $\sum_{i=1}^k p_1^{-1} = 1$. Then
\begin{equation}
    \abs{\Tr(X_1 X_2 \cdots X_k)} \ \leq\ \prod_{i=1}^k \norm{X_i}_{p_i}.
\end{equation}

\begin{proof}
Suppose that we have $p_1,p_2,\dots,p_k \in \mathbb{R}^+$ such that
$\sum_{i=1}^k p_i^{-1} = 1$. As $k>1$, it follows that $p_i > 1$ and
$p_1 / (p_1-1)>0$. Hence by Lemma \ref{lem:1-holderTrace_appendix} we have
\begin{equation}\label{eq:trace_norm_prod_1}
    \abs{\Tr(X_1(X_2\cdots X_k))} \ \leq\ \norm{X_1}_{p_1}
    \norm{X_2\cdots X_k}_{\frac{p_1}{p_1 -1}}.
\end{equation}
By Lemma \ref{lem:SingularValue_appendix},
\begin{align}\label{eq:trace_norm_prod_2}
    \norm{X_2 \cdots X_k}_{\frac{p_1}{p_1-1}} \ &=\
    \left(\sum_{i=1}^N\sigma_i^{\frac{p_1}{p_1-1}}(X_2\cdots X_k)\right)^
    {\frac{p_1-1}{p_1}} \nonumber \\
    \ &\leq\ \left(\sum_{i=1}^N\prod_{j=2}^k\sigma_i^{\frac{p_1}{p_1-1}}(X_j)
    \right)^{\frac{p_1-1}{p_1}}.
\end{align}
Observe for $i = 1, 2, \ldots, k$ that $p_i(p_1-1) / p_1 > 0$, so that
\begin{align}
    \sum_{i=2}^k \frac{p_1}{p_i(p_1-1)} \ &=\ \frac{p_1}{p_1-1}
    \sum\limits_{i=2}^k\frac{1}{p_i} \nonumber \\
    \ &=\ \frac{p_1}{p_1-1} \left(1-\frac{1}{p_1} \right) \nonumber \\
    \ &=\ 1.
\end{align}
Thus, by Generalized H\"older's Inequality \cite{Cw}, we have
\begin{align}
    \left(\sum_{i=1}^N\prod_{j=2}^k\sigma_i^{\frac{p_1}{p_1-1}}(X_j)\right)^
    {\frac{p_1-1}{p_1}} \ &\leq\ \prod_{j=2}^k\left(\sum_{i=1}^N
    \sigma_i^{\frac{p_1}{p_1-1}\cdot\frac{p_i(p_1-1)}{p_1}}(X_j)\right)^
    {\frac{p_1-1}{p_1}\cdot \frac{p_1}{p_i(p_1-1)}} \nonumber \\
    \ &=\ \prod_{j=2}^k\left(\sum_{i=1}^N\sigma_i^{p_i}(X_j)\right)^{\frac{1}{p_i}} \nonumber \\
    \ &=\ \prod_{j=2}^k \norm{X_j}_{p_i}.
\end{align}
Substituting the above into \eqref{eq:trace_norm_prod_1} and
\eqref{eq:trace_norm_prod_2} yields the desired result
\begin{equation}
    \abs{\Tr (X_1 X_2 \cdots X_k)} \leq \prod_{i=1}^k \norm{X_i}_{p_i}.
\end{equation}
\end{proof}
\end{thm}

\subsection{The Kronecker Product} Another possible way of combining two random matrices is the Kronecker product. The Kronecker product $A \otimes B$ of square matrices $A$ and $B$ with sizes $n$ and $m$ is the square matrix of size $mn$ formed by replacing each entry $a_{ij}$ of $A$ by the block $a_{ij}B$.

\begin{defi}
The Kronecker product of an $n\times n$ matrix $A$ and an $m\times m$ matrix $B$ is the $nm \times nm$ block matrix
\begin{equation}
    A \otimes B \ = \
    \left[
    \begin{array}{cccc}
    a_{11}B & a_{12}B & \cdots & a_{1n}B \\
    a_{21}B & a_{22}B & \cdots & a_{2n}B \\
    \vdots  & \vdots  & \ddots & \vdots  \\
    a_{n1}B & a_{n2}B & \cdots & a_{nn}B
    \end{array}
    \right]
\end{equation}
formed by replacing each entry $a_{ij}$ of $A$ by the block $a_{ij}B$.
\end{defi}

The Kronecker product has the following useful property.
\begin{prop}\label{lem:kroneckerspectra}
Given square matrices $A$ and $B$ with eigenvalues $\lambda_i$ and $\mu_j$, respectively, the eigenvalues of $A\otimes B$ are precisely the set of pairwise products $\lambda_i\mu_j$ of eigenvalues from the spectra of $A$ and $B$, respectively.
\end{prop}

This proposition follows from the bilinearity of Kronecker product.  For a proof of Proposition \ref{lem:kroneckerspectra}, see \cite{La}.

\begin{thm}\label{thm:kroneckermoments}
Let $A$ and $B$ be square random matrices chosen independently from two (possibly different) matrix ensembles. If the moments of the eigenvalue distributions of $A$ and $B$ all exist, then the normalized $k$\textsuperscript{th} moment of the eigenvalue distribution of $A \otimes B$ is the product of the normalized $k$\textsuperscript{th} moments of the eigenvalue distributions of $A$ and $B$.
\end{thm}

Now we are able to give a proof of Theorem \ref{thm:kroneckermoments} using the above property and the independence of chosen matrices.

\begin{proof}[Proof of Theorem \ref{thm:kroneckermoments}]
Suppose that $A$ has eigenvalues $\lambda_i$, $1 \le i \le n$, and $B$ has eigenvalues $\mu_j$, $1 \le j \le m$. By Proposition \ref{thm:kroneckermoments}, the eigenvalues of $A \otimes B$ are $\lambda_i \mu_j$, $1 \le i \le n, 1 \le j \le m$. Since $A$ and $B$ are independent, by the Eigenvalue Trace Lemma the $k$\textsuperscript{th} moment of the eigenvalue distribution of $A \otimes B$ is
\begin{equation}
M_{k}(A \otimes B) \ =\ \dfrac{\sum\limits_{i=1}^{n} \sum\limits_{j=1}^{m} (\lambda_i \mu_j)^k}{(mn)^{k/2+1}} \ =\ \dfrac{\sum\limits_{i=1}^{n} (\lambda_i)^{k}}{m^{k/2+1}} \dfrac{\sum\limits_{j=1}^{m} (\mu_j)^k}{n^{k/2+1}} \ =\ M_k(A) M_k(B),
\end{equation}
as desired. 
\end{proof}

This allows us to take two matrix ensembles whose eigenvalue distributions are understood, and combine them via the Kronecker product into a family with a new distribution. We can use this approach to build families of matrices exhibiting behavior that is hybrid between previously studied behaviors, and which can possibly model analogous behaviors in number theory.

\section*{Acknowledgements}
This work was performed at the 2019 SMALL REU at Williams College; it is a
pleasure to thank Williams College and the REU organizers for their help
and support. We would also like to thank Shiliang Gao and Dr. Jun Yin. The first three authors were supported by
NSF Grant \#DMS-1659037; the last two authors were supported by the Department
of Mathematics at the University of Michigan, Ann Arbor.

\bigskip

\end{document}